\documentclass[11pt,a4paper]{article}
\usepackage[utf8]{inputenc}
\usepackage{amsmath}
\usepackage{amsfonts}
\usepackage{amsthm}
\usepackage{todonotes}
\usepackage{hyperref}
\usepackage{enumerate}
\numberwithin{equation}{section}
\usepackage{enumitem}
\usepackage{tikz-cd}
 
\usepackage{bm} 
\usepackage{tikz}

\newcounter{savecntr}

\DeclareMathOperator{\Var}{Var}
\DeclareMathOperator{\Cov}{Cov}
\newcommand{\ve}{\varepsilon} 
\newcommand{\skodual}{\mathcal D(\mathbf{R}_+,\mathcal H^{-L,\gamma}(\mathbf{R}^d))}
\newcommand{\sob}{\mathcal H^{L,\gamma}(\mathbf{R}^d)}

\newcommand{\fsn}{\frac{1}{N}\sum_{i=1}^N}

\newcommand{\dtki}{W_{k+1}^i-W_k^i}
\newcommand{\indep}{\perp \!\!\! \perp}
\newcommand{\noi}{\noindent{}} 
\newcommand{\di}{\mathrm{d}} 

\newcommand{\rp}{\mathbf{R}_+}
\usepackage{amssymb}

\newtheorem{theo}{Theorem} 
\newtheorem{deff}{Definition} 
\newtheorem{rem}{Remark} 
\newtheorem{lem}{Lemma}[section]
\newtheorem{cor}{Corollary}[section]
\newtheorem{prop}[lem]{Proposition}

 \usepackage{titlesec}

\titleformat*{\section}{\large\bfseries}
\titleformat*{\subsection}{\normalsize\bfseries}
\titleformat*{\subsubsection}{\normalsize\bfseries}
\titleformat*{\paragraph}{\normalsize\bfseries}
\titleformat*{\subparagraph}{\normalsize\bfseries}

\usepackage{graphicx}
\usepackage[left=2cm,right=2cm,top=2cm,bottom=2cm]{geometry}
\usepackage{geometry}
\usepackage{mathrsfs}

\title{\Large{\textbf{Law of large numbers and central limit theorem for wide two-layer neural networks: the mini-batch and noisy   case}}}

\author{Arnaud Descours\setcounter{savecntr}{\value{footnote}}\thanks{Laboratoire de Math\'ematiques Blaise Pascal UMR 6620, Université Clermont-Auvergne, Aubière, France. E-mail: \{arnaud.descours,boris.nectoux\}@uca.fr}, \, Arnaud Guillin\setcounter{savecntr}{\value{footnote}}\thanks{Laboratoire de Math\'ematiques Blaise Pascal UMR 6620, Université Clermont-Auvergne, Aubière, France,  and Institut Universitaire de France. E-mail:  arnaud.guillin@uca.fr}, \,  Manon Michel\thanks{CNRS, Laboratoire de Math\'ematiques Blaise Pascal UMR 6620, Université Clermont-Auvergne, Aubière, France. E-mail:  manon.michel@uca.fr},  \,  and Boris Nectoux\footnotemark[1]} 
\date{}

\begin{document}
\maketitle

\begin{abstract}
  In this work, we consider a wide two-layer neural network and study
  the behavior of its empirical weights under a dynamics set by a
  stochastic gradient descent along the quadratic loss with
  mini-batches and noise. Our goal is to prove a trajectorial law of
  large number as well as a central limit theorem for their
  evolution. When the noise is scaling as $1/N^\beta$ and
  $1/2<\beta\le\infty$, we rigorously derive and generalize the LLN
  obtained for example in \cite{Vanden1,mei2,sirignano2020lln}. When
  $3/4<\beta\le\infty$, we also generalize the CLT (see also
  \cite{sirignano2020clt}) and further exhibit the effect of
  mini-batching on the asymptotic variance which leads the
  fluctuations. The case $\beta=3/4$ is trickier and we give an
  example showing the divergence with time of the variance thus
  establishing the instability of the predictions of the neural
  network in this case. It is illustrated by simple numerical
  examples.
 \end{abstract}

\noindent \textbf{Keywords.} Machine learning, neural networks, law of large numbers, central limit theorem, empirical measures, particle systems,  mean field.   \\
 \textbf{AMS classification (2020).} 68T07,	 60F05,	60F15,	35Q70.

  \tableofcontents

\section{Setting and main results}

\subsection{Introduction}

\textbf{Setting and purpose of this work}. Thanks to their impressive
results, deep learning techniques have nowadays become standard
supervised learning methods in various fields of engineering or
research \cite{goodfellow2016deep}. A robust understanding of their
behavior and efficiency is however still lacking and a large effort is
put towards achieving mathematical foundations of empirical
observations. Among this effort, the case of wide two-layer single
network, and its connection with mean-field network, has particularly
been fruitful, as considered for example in \cite{Vanden2,mei2018mean,
  sirignano2020lln,sirignano2020clt}. In such setting, a convergence
towards a limit PDE system can be established when the neuron numbers
goes to infinity. The behavior in long time of this limit PDE may then
give an easier framework to establish the convergence towards
minimizers of the loss function of the neural network. Partial results
can be found in this direction \cite{mei2018mean,NEURIPS2018_a1afc58c}
but as underlined in \cite{E-2020}, a lot still remains to be
understood and proved mathematically rigorously. In this context, our
work is two-fold. First, we will concern ourselves with the
mathematical justification of the law of large numbers and central
limit theorems of the trajectory of the empirical measure of the
weights, under the optimization by a stochastic gradient descent
(SGD), with mini-batching and in the presence of noise with a range of
scalings.  Mini-batch SGD \cite{Bottou2018} is widely used in machine
learning since it allows for shorter training times thanks to
parallelisation, while reducing the variance in SGD estimates. How to
choose the optimal mini-batch size, and furthermore with theoretical
guarantees, remains an active research line
\cite{keskar2017large,smith2018,qian19}.  Introducing noise in SGD, as
considered in \cite{mei2018mean}, can lead to better generalisation
perfomance thanks to an improved ability to escape saddle points, as
shown in \cite{jin2021nonconvex}. Note that this differs from the
analysis approach consisting in directly modelizing the noise of SGD
as for instance done in \cite{Wu-2019,Sim-2019}. Second we will do so
by providing a rigorous framework which could be generalized to study
overparametrized limit of other neural networks (e.g. deep ensemble,
bayesian neural networks, ...). Thus, the benefit of the
overparametrized limit and its convexification of the loss landscape
through a non-linear PDE could lead in these different architectures
to derivations of theoretical guarantees of convergence, while it
remains hard to analyse these landscapes directly in the case of a
finite number of neurons, even large.

Let us now precise the framework for this paper.
Let $(\Omega, \mathcal F,\mathbf P)$ be a probability space, and
$\mathcal{X}$ and $\mathcal{Y}$ be subsets of $\mathbf{R}^n$
($n\geq 1$) and $\mathbf{R}$ respectively. In this work, we consider
the following two-layer neural network
\begin{equation}\label{neural network g}g_W^N(x):=\frac{1}{N}\sum_{i=1}^N\sigma_*(W^i,x),\end{equation}
where $x\in\mathcal{X}$ denotes the input data,
$g_W^N(x)\in\mathbf{R}$ the output returned by the neural network,
$\sigma_*:\mathbf{R}^d\times\mathcal{X}\rightarrow\mathbf{R}$ the
activation function, $N\geq 1$ the number of neurons on the hidden
layer, and $W=(W^1,\dots,W^N)\in(\mathbf{R}^d)^N$ are the weights to
optimize ($d\geq 1$).  In the supervised learning setting, a data
point $(x,y)\in\mathcal{X}\times\mathcal{Y}$ is distributed according
to $\pi\in\mathcal{P}(\mathcal{X}\times\mathcal{Y})$, where
$\mathcal{P}(\mathcal{X}\times\mathcal{Y})$ denotes the set of
probability measures on $\mathcal{X}\times\mathcal{Y}$.  Ideally, one
chooses the weights $W=(W^1,\dots,W^N)$ as a global minimizer of the
risk $\textbf{E}_{\pi}[\mathsf L(g_W^N(x),y)]$, where
$\mathsf L:\mathbf{R}\times\mathbf{R}\rightarrow\mathbf{R}$ is the
so-called loss function ($\textbf{E}_{\pi}$ stands for the expectation
when $(x,y)\sim \pi$). In this work, we consider the square loss
function out of simplicity, but other loss function or classification
problem could be considered, namely:
$$\mathsf L(g_W^N(x),y)=\frac{1}{2}\big |g_W^N(x)-y\big |^2.$$
Since the risk can not be computed (because $\pi$ is unknown), the
parameters are usually learned by stochastic gradient descent.  In
this work, we consider the mini-batch setting with weak noise,
which is defined as follows. First, for $k\ge 0$, consider
$((x_k^n,y_k^n))_{n\ge 1}$ a sequence of random elements on
$\mathcal{X}\times\mathcal{Y}$ (each $(x_k^n,y_k^n)$ being distributed
according to $\pi$), and $N_k$ a random element with values in
$\mathbf N^*=\{1,2,3,\ldots\}$. Then, the mini-batch $B_k$ is defined
by:
$$B_k=\{(x^1_k,y_k^1), \ldots, (x^{N_k}_k,y^{N_k}_k)\}, \text{ in particular $|B_k|=N_k$, where $|B_k|$ denotes the cardinality of $B_k$}.$$
In addition, at each
iteration of SGD, we add a Gaussian noise term, whose variance is
scaled according to $N^{-2\beta}$, with $\beta>\frac 12$, hence qualified
\emph{weak}. Note that the case of Gaussian noise with $\beta = 1/2$
is addressed in \cite{mei2018mean} and could also be considered here
in our setting, but with additional assumptions to integrate the noise
in the limit process.

Thus, the SGD algorithm we consider is the following : for $k\geq 0$
and $i\in \{1,\dots,N\}$,
\begin{equation}\begin{cases}\label{algorithm}
W_{k+1}^i=W_k^i+\displaystyle\frac{\alpha}{N|B_k|}\sum_{(x,y)\in B_k}(y-g_{W_k}^N(x))\nabla_W\sigma_*(W_k^i,x)+\frac{\ve_k^i}{N^\beta},\\
W_0^i\sim \mu_0,
\end{cases}\end{equation}
where $\ve_k^i\sim \mathcal{N}(0,I_d)$ and $\mu_0\in\mathcal{P}(\mathbf{R}^d)$.   
The evolution of the weights is tracked through their empirical distribution $\nu_k^N$ (for $k\geq 0$) and its scaled version $\mu_t^N$ (for $t\in\mathbf{R}_+$), which are defined as follows:  
$$
\nu_k^N:=\frac{1}{N}\sum_{i=1}^N\delta_{W_k^i} \ \text{ and }  \  
\mu_t^N:=\nu_{\lfloor Nt\rfloor}^N.
$$
For an element $\mu \in \mathcal M_b(\mathbf R^d)$ (the space of  bounded countably additive measures on $\mathbf R^d$),  we use the notation
$$\langle f,\mu\rangle_{\mathrm m} =\int_{\mathbf{R}^d} f(w)  \mu(\di w),$$
for any $f: \mathbf R^d\to \mathbf R$ such that $\int_{\mathbf{R}^d} f(w)  \mu(\di w)$ exists. If no confusion is possible, we simply denote $\langle f,\mu\rangle_{\mathrm m}$ by $\langle f,\mu\rangle$. 
For instance, considering the neural network~\eqref{neural network g}, we have, for any $x\in\mathcal{X}$, 
 $$g_{W_k}^N(x)=\frac{1}{N}\sum_{i=1}^N\sigma_*(W^i_k,x)=\langle\sigma_*(\cdot,x),\nu_k^N\rangle, \  \ k\ge 0.$$ 
 In this work, we prove that the the whole trajectory of the scaled
 empirical measures of the weights defined by~\eqref{algorithm}
 (namely $\lbrace t\mapsto\mu_t^N,t\in\mathbf{R}_+\rbrace_{N\geq1}$)
 satisfies a law of large numbers and a central limit theorem, see
 respectively Theorem~\ref{thm:lln} and Theorem~\ref{thm:clt}.  We
 also exhibit a particular fluctuation behavior depending on the value
 of the parameter $\beta$ ruling the weakness of the added noise.\medskip

 \noindent \textbf{Related works}.  Law of large numbers and central
 limits theorems have been obtained for several kinds of mean-field
 interacting particle systems in the mathematical literature, see for
 instance~\cite{sznitman_topics_1991,hitsuda1986tightness,fernandez1997hilbertian,jourdainAIHP,delarue,delmoral,kurtz2004stochastic}
 and references therein. When considering particle systems arising
 from the SGD-minimization problem in a two-layer neural network, we
 refer to~\cite{mei2018mean} for a law of large numbers on the
 empirical measure at fixed times, see also \cite{mei2}. We also refer
 to~\cite{Vanden2} where conditions for global convergence of the GD
 on the ideal loss and of the SGD with mini-batches increasing in size
 with $N$, as well as the scaling of the error with the size of the
 network, are established from formal asymptotic arguments. Doing so,
 they also observe with increasing mini-batch size in the SGD the
 reduction of the variance of the process leading the fluctuations of
 the empirical measure of the weights (see \cite[Arxiv-V2. Sec
 3.3]{Vanden2}), until the mini-batches are large enough to recover
 the situation of the idealized gradient descent (similar to an
 infinite batch), which leads to other order of fluctuations (see
 \cite[Arxiv-V2. Prop 2.3]{Vanden2}). We also refer to \cite{Vanden1}
 for a similar line of work on the GD on the empirical loss. A law of
 large numbers and a central limit theorem on the whole trajectory of
 the empirical measure are also obtained in
 \cite{sirignano2020lln,sirignano2020clt} for a standard SGD
 scheme. We also mention the work done in \cite{durmus-neural} on
 propagation of chaos for SGD with different step-size schemes. In
 this work, and compared to the existing literature dealing with the
 SGD minimization problem in two-layer neural networks, we provide a
 rigorous proof with precise justifications of all steps of the
 existence of the limit PDE (in particular, uniqueness and relative
 compactness) in the law of large numbers as well as the limit process
 for the central limit theorem on the trajectory of the empirical
 measure. This will be the basis for future works on deep ensembles or
 overparameterized bayesian neural networks. We furthermore do so in a
 more general variant of SGD with mini-batching of any size and weak
 noise (see~\eqref{algorithm}). A noisy SGD was also considered in
 \cite{mei2018mean}, corresponding to $\beta=1/2$ in our setting, for
 which they obtain for the LLN a different limit PDE than in the
 non-noisy case (presence of an additionnal regularizing Laplacian
 term in the limit equation). While we could recover in a
 straightforward manner a trajectorial version of \cite{mei2018mean},
 we consider here out of concision the range $\beta> 1/2$, showing a
 single limit PDE for the LLN, and obtain a similar result for
 $\beta > 3/4$ for the CLT, while showing analytically for $\beta=3/4$
 and numerically for $\beta\leq 3/4$ a particular fluctuation
 behavior. Furthermore, we analytically show the expected reduction,
 with the mini-batch size, of the variance of the process leading the
 fluctuations of the weight empirical measure and numerically display
 the reduction of the global variance.

\subsection{Main results}

The sequence
$\lbrace t\mapsto\mu_t^N,t\in\mathbf{R}_+\rbrace_{N\geq1}$ is studied
as a sequence of processes with values in the dual of some (weighted)
Hilbert space on $\mathbf R^d$. These Hilbert spaces are introduced in
the next section.

\subsubsection{Notation and assumptions}

\textbf{Weighted Sobolev spaces}. Following \cite[Chapter
3]{adams2003sobolev}, we consider, for a function
$g\in \mathcal C_c^\infty(\mathbf R^d)$ (the space of functions
$g:\mathbf{R}^d\rightarrow\mathbf{R}$ of class $\mathcal C^\infty$
with compact support), the following norm, defined for
$J\in\textbf{N}$ and $b\geq0$ :
$$\|g\|_{\mathcal H^{J,b}}:=\Big(\sum_{|k|\leq J}\int_{\mathbf{R}^d}\frac{|D^kg(x)|^2}{1+|x|^{2b}}\di x\Big)^{1/2}.$$
Let $\mathcal H^{J,b}(\mathbf{R}^d)$ be the closure of the set
$\mathcal C_c^\infty(\mathbf{R}^d)$ for this norm. The space
$\mathcal H^{J,b}(\mathbf{R}^d)$ is a Hilbert space when endowed
with the norm $\|\cdot\|_{\mathcal H^{J,b}}$. The associated
scalar product on $\mathcal H^{J, b}(\mathbf{R}^d)$ will be denoted
by $\langle\cdot,\cdot\rangle_{\mathcal H^{J,b}}$.  We denote by
$\mathcal H^{-J,b}(\mathbf{R}^d)$ its dual space.  For an element
$\Phi \in \mathcal H^{-J,b}(\mathbf{R}^d)$, we use the notation
$$\langle f,\Phi\rangle_{J,b}= \Phi[f], \ f\in \mathcal H^{J,b}(\mathbf{R}^d).$$
For ease of notation, and if no confusion is possible, we simply
denote $\langle f,\Phi\rangle_{J,b }$ by $\langle f,\Phi\rangle$.
Let us now define $\mathcal C^{J,b }(\mathbf{R}^d)$ as the space of
functions $g:\mathbf{R}^d\rightarrow\mathbf{R}$ with continuous
partial derivatives up to order $J\in\textbf{N}$ such that
$$\text{for all} \ |k|\leq J, \ \lim_{|x|\rightarrow\infty}\frac{|D^kg(x)|}{1+|x|^b }=0.$$
This space is endowed with the norm
$$\|g\|_{\mathcal C^{J,b }}:=\sum_{|k|\leq J}\ \sup_{x\in\mathbf{R}^d}\frac{|D^kg(x)|}{1+|x|^b }.$$
We also introduce $\mathcal C_b(\mathbf{R}^d)$, the space of bounded
continuous functions $g:\mathbf{R}^d\rightarrow\mathbf{R}$, endowed
with the supremum norm. We also denote by
$\mathcal C^\infty_b(\mathbf{R}^d)$ the space of smooth functions over
$\mathbf R^d$ whose derivatives of all order are bounded.  We have
$\mathcal C^\infty_b(\mathbf{R}^d)\subset \mathcal
H^{J,b }(\mathbf{R}^d)$ as soon as $b >d/2$ (more generally
$x\in \mathbf{R}^d \mapsto (1-\chi(x))|x|^{ a}\in \mathcal
H^{J,b }(\mathbf{R}^d)$ if $b - a>d/2$, where $\chi\in \mathcal C_c^\infty(\mathbf R^d, [0,1])$  equals $1$ near $0$).  \medskip

\noindent
\textbf{Weighted Sobolev embeddings}. We recall that from
\cite[Section 2]{fernandez1997hilbertian},
\begin{equation}\label{eq.Sobolev-embH}
  \text{$\mathcal H^{\ell+j,a}(\mathbf{R}^d)\hookrightarrow_{\text{H.S.}} \mathcal H^{j,a+b }(\mathbf{R}^d)$ when $\ell>d/2$, $b >d/2$, and $a,j\ge 0$}
\end{equation}
where $\hookrightarrow_{\text{H.S.}}$ means that the embedding is of Hilbert-Schmidt type, and 
\begin{equation}\label{eq.SEC}
\text{$\mathcal H^{\ell+j,a}(\mathbf{R}^d)\hookrightarrow \mathcal C^{j,a}(\mathbf{R}^d)$ when $\ell>d/2$,  and  $a,j\ge 0$}.
\end{equation}
We  set 
\begin{equation}\label{eq.Lgamma}
L= \lceil\frac{d}{2}\rceil+3 , \ \gamma=4\lceil\frac{d}{2}\rceil+5,  \text{ and } \gamma_* :=\gamma+1.
\end{equation} 
 According to~\eqref{eq.SEC} and since $\gamma_*>\gamma$, it holds:
\begin{equation}
\label{eq.SE1}
\mathcal H^{L,\gamma}(\mathbf{R}^d)\hookrightarrow \mathcal C^{2,\gamma}(\mathbf{R}^d)\hookrightarrow \mathcal C^{2,\gamma_*}(\mathbf{R}^d).
\end{equation}

We set throughout this work, for all $N\geq 1$:
$$\mu^N:=\lbrace t\mapsto\mu_t^N, t\in\mathbf{R}_+\rbrace.$$
When $E$ is a metric space, we denote by $E'$ its dual and by
 $\mathcal D(\textbf R_+, E)$ the set of càdlàg functions from
$\mathbf{R}_+$ to $E$.
 For $b \ge 0$ and for all $N\geq 1$,
$\mu^N$ is a random element of 
$\mathcal D(\textbf R_+, \mathcal C^{0,b }(\textbf R^d)')$, and thus also of $\mathcal D(\textbf R_+, \mathcal H^{-J,b }(\textbf R^d))$, as soon as $J>d/2$ (by \eqref{eq.SEC}).

Let for $k\ge 1$,
\begin{equation}\label{eq.Pk}
\mathcal{P}_k(\mathbf{R}^d):=\big\lbrace\mu\in\mathcal{P}(\mathbf{R}^d), \ \int_{\mathbf{R}^d}|w|^k\mu(\di w)<+\infty\big\rbrace,
\end{equation}
 which is endowed with the   Wasserstein distance 
 $$\mathsf W_k(\mu,\nu)= \big[\inf\{\textbf E[ |X-Y|^k|], \,\textbf P_X=\mu \text{ and }  \textbf P_Y=\nu \}\big]^{1/k}.$$
 We refer for instance to \cite[Chapter 5]{santambrogio2015optimal}
 for more about these spaces.  We recall that
 $\mathsf W_1(\mu,\nu)\le \mathsf W_k(\mu,\nu)$ ($k\ge 1$) and the
 dual formula for $\mathsf W_1(\mu,\nu)$:
\begin{equation}\label{Kantorovitch Rubinstein}
\mathsf W_1(\mu,\nu)=\sup\Big \lbrace\big|\int_{\mathbf{R}^d}f(w)\di\mu(w)-\int_{\mathbf{R}^d}f(w)\nu(\di w)\big|, \  \|f\|_{\text{Lip}}\leq 1\Big\rbrace.\end{equation}
 Note also that for all $N\geq 1$,
$\mu^N$ is a random element of
$\mathcal D(\textbf R_+, \mathcal P_{q}(\mathbf R^d))$, for all $q\ge 0$.
\medskip

\noindent
\textbf{Assumptions}.
For $N\ge 1$, we introduce  the $\boldsymbol{\sigma}$-algebras,   
\begin{equation}\label{eq.filtration}
\mathcal{F}_0^N=\boldsymbol{\sigma}\{ \{W_0^i\}_{i=1}^N\big \} \text{ and, for $k\ge1$,  } \mathcal{F}_k^N=\boldsymbol{\sigma}\big\{ W_0^i, \{B_j\}_{j=0}^{k-1},   \{\ve_j^i\}_{j=0}^{k-1},  \,  i\in \{1,\dots,N\}\big\}.
\end{equation}
The main assumptions of this work are the following:
\begin{enumerate}[label=\normalfont\textbf{A\arabic*}., ref=\normalfont\textbf{A\arabic*}]
\item\label{as:batch} For all $k,q\in\textbf{N}$, $|B_q|\indep ((x_k^n,y_k^n))_{n\ge 1}$. In addition, for all $k\in\textbf{N}$, $\big(|B_k|, ((x_k^n,y_k^n))_{n\ge 1}\big)\indep \mathcal  {F}_k^N$.
\item\label{as:sigma} The activation function $\sigma_* : \mathbf{R}^d\times\mathcal{X}\rightarrow\mathbf{R}$ belongs to $\mathcal C^\infty_b(\mathbf{R}^d\times \mathcal{X})$. 
\item\label{as:data} For all $\ell\neq k\in\textbf{N}$,
  $((x^n_\ell,y^n_\ell))_{n\ge 1} \indep ((x_k^n,y_k^n))_{n\ge 1}$. In
  addition, for all $k\in \textbf{N}$, $((x_k^n,y_k^n))_{n\ge 1}$ is a
  sequence of i.i.d random variables from
  $ \pi\in\mathcal{P}(\mathcal{X}\times\mathcal{Y})$, and
  $\mathbf{E}[|y|^{16\gamma_*}]$ is finite.
\item\label{as:moments initiaux}The randomly initialized parameters $\{W_0^i\}_{i=1}^N$ are i.i.d. with a distribution $\mu_0\in\mathcal{P}(\mathbf{R}^d)$ such that $\mathbf{E}[|W_0^1|^{8\gamma_*}]<+\infty$.
\item\label{as:noise} For all $k\in\textbf{N}$ and
  $i\in \{1,\dots,N\}$, $\ve_k^i\sim\mathcal{N}(0,I_{d})$ and
  $\ve_k^i\indep \mathcal F_k^N$. In addition, for all
  $k,l\in\textbf{N}$ and $i,j\in \{1,\dots,N\}$ such that
  $(i,k)\neq (j,l)$, $\ve_k^i\indep \ve_l^j$.
\end{enumerate}

\subsubsection{Law of large numbers for the empirical measure}
\label{sec.LLNR}
\textbf{Statement of the law of large numbers}. The first main result
of this work is a law of large numbers for the trajectory of the
scaled empirical measures.
\begin{theo} 
\label{thm:lln}
Let $\beta>1/2$ and assume~\ref{as:batch}-\ref{as:noise}.  Then,  the sequence $(\mu^N)_{N\ge 1}$ converges in probability to a deterministic element
  $\bar \mu$ in $\mathcal D(\mathbf R_+, \mathcal P_{\gamma}(\mathbf{R}^d))$. In addition, $\bar \mu \in \mathcal C(\mathbf R_+, \mathcal P_{1}(\mathbf{R}^d))$ and it is the unique solution in $\mathcal C(\mathbf R_+, \mathcal P_{1}(\mathbf{R}^d)$ of the following measure-valued equation:
\begin{align}
\nonumber
&\forall f\in \mathcal C_b^\infty(\mathbf R^d), t\in \mathbf R_+, \\
\label{eq limite} 
& \langle f,\bar \mu_t\rangle =\langle f,\mu_0\rangle  + \int_0^t \int_{\mathcal{X}\times\mathcal{Y}}\alpha(y-\langle \sigma_*(\cdot, x),\bar \mu_s\rangle )\langle\nabla f\cdot\nabla\sigma_*(\cdot,x),\bar \mu_s\rangle \, \pi(\di x,\di y)\, \di s. 
\end{align} 
\end{theo}
\begin{cor}\label{co.e}
 Assume~\ref{as:batch}-\ref{as:noise}. Then, 
$\bar \mu \in \mathcal C(\mathbf{R}_+,\mathcal H^{-L,\gamma}(\mathbf{R}^d))$. In addition,  $\bar \mu$ satisfies also \eqref{eq limite}  for test functions $f\in \mathcal H^{L,\gamma}(\mathbf{R}^d)$. 
\end{cor}

\begin{proof}[Proof of Corollary \ref{co.e}]
Note first that by \ref{as:moments initiaux}, $\mu_0\in \mathcal C^{0,\gamma}(\mathbf R^d)' \hookrightarrow \mathcal H^{-L,\gamma}(\mathbf{R}^d)$ according to \eqref{eq.SE1}. 
By \eqref{eq limite}, \ref{as:sigma}, and \ref{as:data}, it holds for all $f\in \mathcal C_c^\infty(\mathbf R^d)$ and  $0\le s\le t \le T$, 
\begin{align*}
|\langle f,\bar \mu_t\rangle-\langle f,\bar \mu_s\rangle| \le C|t-s| \Vert f \Vert_{\mathcal C^{1,\gamma}} \sup_{u\in [0,T]} |\langle 1+|\cdot |^{ \gamma},\bar \mu_u\rangle|.
\end{align*} 
Note that $\sup_{u\in [0,T]} |\langle 1+|\cdot |^\gamma,\bar \mu_u\rangle|<+\infty$ since $t\ge 0\mapsto \langle 1+|\cdot |^\gamma,\bar \mu_u\rangle\in  \mathcal D(\mathbf R_+, \mathbf R)$ (indeed this follows from the fact that $\bar \mu \in \mathcal D(\mathbf R_+, \mathcal P_{\gamma}(\mathbf{R}^d))$ and~\cite[Theorem 6.9]{villani2009optimal}). 
Thus, using \eqref{eq.SE1}, it holds $\bar \mu_t\in   \mathcal H^{-L,\gamma}(\mathbf{R}^d)$ and  $|\langle f,\bar \mu_t\rangle-\langle f,\bar \mu_s\rangle| \le C|t-s| \Vert f \Vert_{\mathcal H^{L,\gamma}}$, proving the first claim in Corollary~\ref{co.e}.  The second claim in Corollary~\ref{co.e} is obtained by a density argument and the fact that $\mathcal H^{L,\gamma}(\mathbf{R}^d) \hookrightarrow \mathcal C^{1,\gamma}(\mathbf R^d)$. 
\end{proof}

  \noindent
 \textbf{On the proof of Theorem~\ref{thm:lln}}. 
 Theorem~\ref{thm:lln} is proved in Section~\ref{sec.THMLLN}. The proof strategy is
  the following.  We first derive an
  identity satisfied by $(\mu^N)_{N\ge 1}$, namely the pre-limit
  equation~\eqref{eq: prelim mu}. This is done in
  Section~\ref{sec.PREL}.  Then, we show in Section~\ref{sec.RC-LLN}
 that $(\mu^N)_{N\ge 1}$ is relatively compact in $\mathcal D(\mathbf R_+, \mathcal P_{\gamma}(\mathbf{R}^d))$.   To this end we use \cite[Theorem 4.6]{jakubowski1986skorokhod}. The compact containment  of $(\mu^N)_{N\ge 1}$ relies on a characterization of the compact subsets of  $\mathcal P_{\gamma_0}(\mathbf R^{d+1})$ (see Proposition \ref{prop.compact_wasserstein}) and moment estimates on  $\{\theta_k^i, i\in \{1,\ldots N\}\}_{k=0, \ldots, \lfloor NT \rfloor}$ (see Lemma~\ref{le.W}). 
 We then use the pre-limit equation~\eqref{eq: prelim mu} to prove that any
  limit point of the sequence $(\mu^N)_{N\ge 1}$ in
$\mathcal D(\mathbf R_+, \mathcal P_{\gamma}(\mathbf{R}^d))$ satisfies~\eqref{eq limite}. This requires to study the continuity property of the involved operator (namely $\Lambda_t[f]$, see Lemma \ref{le.Fc}). This the purpose of Section~\ref{sec:conv to lim eq}, and more precisely of Proposition \ref{lem:convergence to lim eq} there.
With rough estimates on the jumps of the function $t\in \mathbf R_+\mapsto \langle f, \mu_t^N\rangle$ (where $f$ is uniformly Lipschitz over $\mathbf R^{d+1}$),  we also prove in Section \ref{subsec: continuity prop sob} that any limit point  $(\mu^N)_{N\ge 1}$ in
$\mathcal D(\mathbf R_+, \mathcal P_{\gamma}(\mathbf{R}^d))$ belongs a.s. to  
$\mathcal C(\mathbf R_+, \mathcal P_{1}(\mathbf{R}^d))$. This is indeed needed since we then prove in Section \ref{sec: uniqueness}  that~\eqref{eq limite} admits a unique solution in $\mathcal C(\mathbf R_+, \mathcal P_{1}(\mathbf{R}^d))$. 
To prove that there is at most one solution to~\eqref{eq limite}, 
we use arguments of~\cite{piccoli2015control} which are based on 
a  representation formula for solution to measure-valued equations~\cite[Theorem 5.34]{villani2021topics} together with time estimates in  Wasserstein distances between two solutions of~\eqref{eq limite} derived in \cite{piccoli2016properties}. 

\begin{rem} 
In view of their proofs, Theorem \ref{thm:lln} and Corollary \ref{co.e} are still valid for $\gamma>\frac d2$ and  $L>d/2+1$. 
\end{rem}

\begin{rem}
  When $\beta=1/2$, one can obtain a similar limit equation for
  $\bar\mu$, with an additionnal (regularizing) Laplacian term in the
  limit equation. To derive it, one should consider a Taylor expansion
  up to order 3 of the test function in the pre-limit
  equation~\eqref{eq: prelim mu}. Let us mention that the case
  $\beta=1/2$ is studied in \cite{mei2018mean} but only at fixed
  $t$. Straightforward application of our method would lead to a
  trajectorial version of \cite[Theorem~3]{mei2018mean} which we leave to
  the reader for the sake of brevity.
\end{rem}

\begin{rem}
  Of course, one important question is the convergence of $\bar\mu_t$
  in long time. It is not hard to see that the loss function decays
  (but not strictly a priori) along the training, i.e. with $t$.  This
  asymptotic behavior of $\bar\mu_t$ as $t\to+\infty$ has been studied
  in~\cite[Theorem~7]{mei2018mean} or \cite{NEURIPS2018_a1afc58c} who give
  partial results in the case without noise. Roughly speaking, they
  prove that if it is known that $\bar\mu_t$ is converging in
  Wasserstein distance then it converges to the minimum of the loss
  function. It is however quite hard to prove such a convergence. We
  refer also to~\cite{E-2020, ma2020towards} for what remains to do in
  this direction which is clearly a difficult open problem. In the
  case with noise $\beta=1/2$ then the situation is different as the
  limit PDE is a usual McKean-Vlasov diffusion and one can study the
  free energy and study convergence in long time
  \cite[Theorem~4]{mei2018mean}. \end{rem}

\subsubsection{Central limit theorem for the empirical measure}
 
\textbf{Fluctuation process and extra assumptions.}
Assume~\ref{as:batch}-\ref{as:noise}. The fluctuation process is the
process $\eta^N =\lbrace t\mapsto \eta^N_t, t\in\mathbf{R}_+\rbrace$
defined by:
\begin{equation}\label{eq.etaN}
\eta_t^N=\sqrt{N}(\mu_t^N -\bar \mu_t), \ N\ge 1, \ t\in \mathbf R_+, 
\end{equation}
where $\bar \mu =\lbrace t\mapsto\bar \mu_t, t\in\mathbf{R}_+\rbrace$
is the limit of $(\mu^N)_{N\ge 1}$ in $\mathcal D(\mathbf R_+, \mathcal P_{\gamma}(\mathbf{R}^d))$ (see
Theorem~\ref{thm:lln}). Let us introduce the following additional
assumptions:
\begin{enumerate}[label=\normalfont\textbf{A\arabic*}.,
  ref=\normalfont\textbf{A\arabic*},resume]
\item \label{as: mu0compact} The distribution
  $\mu_0\in\mathcal{P}(\mathbf{R}^d)$ is compactly supported.
\item\label{as:batch limite} $|B_k|\to |B_\infty|$ a.s. as $k\rightarrow\infty$.
\end{enumerate}
 
\noindent
 Let  
\begin{align}\label{eq.J0}
  J_0\ge 4\lceil\frac{d}{2}\rceil+8 \text{ and } j_0= \lceil\frac{d}{2}\rceil+2.
\end{align}
For later purpose, we also set
  \begin{align}\label{eq.J1}
\text{$J_1=2\lceil\frac{d}{2}\rceil+4$,  $j_1=3\lceil\frac{d}{2}\rceil+4 $,  $J_2=3\lceil\frac{d}{2}\rceil+6$, and $j_2=2\lceil\frac{d}{2}\rceil+3$.}
\end{align}
 By~\eqref{eq.Sobolev-embH},  we have the following embeddings:
\begin{align}\label{eq.SE2}
\mathcal H^{J_0-1,j_0}(\mathbf{R}^d)\hookrightarrow_{\text{H.S.}}\mathcal H^{J_2,j_2}(\mathbf{R}^d), \  \mathcal H^{J_2,j_2}(\mathbf{R}^d)\hookrightarrow_{\text{H.S.}}\mathcal H^{J_1+1,j_1}(\mathbf{R}^d),\ 
\mathcal H^{J_1,j_1}(\mathbf{R}^d)\hookrightarrow_{\text{H.S.}}\mathcal H^{L,\gamma}(\mathbf{R}^d).
\end{align}

\noindent
\textbf{G-process and the limit equation.}

\begin{deff} 
\label{de.gaussian}
We say that
$\mathscr G\in \mathcal C(\mathbf R_+,\mathcal
H^{-J_0,j_0}(\mathbf{R}^d))$ is a {\rm G-process} if for all $k\ge 1$
and $f_1\dots,f_{k}\in \mathcal H^{J_0,j_0}(\mathbf R^d)$,
$\{t\mapsto ({\langle f_{1}, \mathscr G}_t\rangle,\dots, \langle
f_{k}, \mathscr G_t\rangle)^T ,t\in\mathbf{R}_+\}\in \mathcal
C(\mathbf{R}_+,\mathbf{R}^{k})$ is a process with zero-mean,
independent Gaussian increments (and thus a martingale), and with
covariance structure given by: for all $1\leq i,j\leq {k}$ and all
$0\le s\le t$,
\begin{equation}\label{eq.cov2}
\Cov\big ( \langle f_i, \mathscr G_t \rangle,\langle f_j, \mathscr G_s \rangle\big )=\alpha^2 \mathbf{E}\left[\frac{1}{|B_{\infty}|}\right]\int_0^s\Cov (\mathrm Q_v[f_i](x,y),\mathrm Q_v[f_j ](x,y))\, \di v,
\end{equation}
 where $\mathrm Q_v[f](x,y):=(y-\langle \sigma_*(\cdot, x),\bar \mu_v\rangle)\langle\nabla f\cdot\nabla\sigma_*(\cdot,x),\bar \mu_v\rangle$ for $f\in \mathcal H^{J_0,j_0}(\mathbf{R}^d)$ and $\bar \mu$ is given by Theorem~\ref{thm:lln}. 
\end{deff}
\noindent
Let us make some comments about Definition~\ref{de.gaussian}.  The
first one is that we have decided to call such a process G-process to
ease the statement of the results.
In addition, notice that $\mathrm Q_s[f](x,y)$ is well defined for
$f\in \mathcal H^{J_0,j_0}(\mathbf{R}^d))$ (indeed for all
$k\in \{1,\ldots, d\}$,
$\partial_{e_k} f\in \mathcal H^{J_0-1,j_0}(\mathbf{R}^d)\hookrightarrow \mathcal
H^{L,\gamma}(\mathbf{R}^d))$ and
$\bar \mu\in \mathcal C(\mathbf R_+,\in \mathcal
H^{-L,\gamma}(\mathbf{R}^d))$).  Finally, we mention that by
Proposition~\ref{prop P Q coincide} below, the law of a process
$\mu \in \mathcal D(\mathbf R_+,\mathcal H^{-J,b}(\mathbf{R}^d))$
is fully determined by the family of laws of the processes
$ (\langle f_{1},{\mu}\rangle,\dots,\langle f_{k},{\mu}\rangle)^T \in
\mathcal D(\mathbf{R}_+,\mathbf{R})^k$, $k\ge 1$ and where
$\{f_a\}_{a\ge 1}$ is an orthonormal basis
$\mathcal H^{J,b}(\mathbf R^d)$.  \medskip

\noindent
For $\eta$ a
$\mathcal C(\mathbf R_+,\mathcal H^{-J_0+1,j_0}(\mathbf{R}^d))$-valued
process and
$\mathscr G\in \mathcal C(\mathbf R_+,\mathcal
H^{-J_0,j_0}(\mathbf{R}^d))$ a {\rm G}-process (see
Definition~\ref{de.gaussian}), define the following equation:
\begin{align} 
\nonumber 
&\text{A.s. } \forall f\in \mathcal H^{J_0,j_0}(\mathbf R^d), \forall t\in\mathbf R_+, \\
\nonumber
&\,  \langle f,\eta_t\rangle - \langle f,\eta_0 \rangle=\int_0^t\int_{\mathcal{X}\times\mathcal{Y}}\alpha(y-\langle\sigma_*(\cdot,x),\bar \mu_s\rangle)\langle\nabla f\cdot\nabla\sigma_*(\cdot,x),\eta_s\rangle \pi(\di x,\di y)\\\label{eq.CLT}
&\quad \quad \quad  \quad \quad \quad \quad \  -\int_0^t\int_{\mathcal{X}\times\mathcal{Y}}\alpha {\langle\sigma_*(\cdot,x),\eta_s\rangle \langle}\nabla f\cdot\nabla\sigma_*(\cdot,x),\bar \mu _s\rangle \pi(\di x,\di y) +\langle f, \mathscr G_t\rangle.
\end{align}

\begin{deff} 
\label{de.weak}
Let $\nu$ be a $\mathcal H^{-J_0+1 ,j_0}(\mathbf R^d)$-valued random
variable. We say that a
$\mathcal C(\mathbf R_+,\mathcal H^{-J_0+1,j_0}(\mathbf{R}^d))$-valued
process $ \eta$ on a probability space is a {\rm weak solution}
of~\eqref{eq.CLT} with initial distribution $\nu$ if there exist a
{\rm G}-process
$\mathscr G\in \mathcal C(\mathbf R_+,\mathcal
H^{-J_0,j_0}(\mathbf{R}^d))$ such that~\eqref{eq.CLT} holds and
$ \eta_0=\nu$ in distribution. In addition, we say that {\rm weak
  uniqueness} holds if for any weak two solutions $ \eta^1$ and $\eta^2$
of~\eqref{eq.CLT} (possibly defined on two different probability
spaces) with the same initial distributions, it holds $\eta_1=\eta_2$
in distribution.
\end{deff}

The second main result of this work is a central limit theorem for
the trajectory of the scaled empirical measures.

\begin{theo}\label{thm:clt}
\begin{sloppypar} 
Let $\beta>3/4$. Assume~\ref{as:batch}-\ref{as:batch limite}.   Then:
\begin{enumerate}

\item (Convergence) The sequence
  $(\eta^N)_{N\ge1}\subset \mathcal D(\mathbf R_+,\mathcal
  H^{-J_0+1,j_0}(\mathbf R^d))$ (see~\eqref{eq.etaN}) converges in
  distribution to a process
  $\eta^*\in \mathcal C(\mathbf R_+,\mathcal
  H^{-J_0+1,j_0}(\mathbf{R}^d))$.

\item (Limit equation) The process $\eta^*$ has the same distribution
  as the unique weak solution $\eta^\star$ of~\eqref{eq.CLT} with
  initial distribution $\nu_0$ (see Definition~\ref{de.weak}), where
  $\nu_0$ is the unique $\mathcal H^{-J_0+1 ,j_0}(\mathbf R^d)$-valued
  random variable such that for all $k\ge 1$ and
  $f_1\dots,f_{k}\in \mathcal H^{J_0-1,j_0}(\mathbf R^d)$,
$$(\langle f_{1}, \eta_0^*\rangle,\dots, \langle f_{k}, \eta^*_0\rangle)^T \sim \mathcal N (0, \Gamma(f_{1},\ldots,f_{k})),$$
where $\Gamma(f_{1},\ldots,f_{k})$ is the covariance matrix of the
vector $( f_{1}(W_0^1),\dots, f_{k}(W_0^1))^T$.

\end{enumerate}
\end{sloppypar}

\end{theo}

\begin{rem}\label{rem.red_var}
  By looking at the definition of the {\rm G}-process and in particular
  its covariance \eqref{eq.cov2}, one remarks the effect of
  mini-batching by the $|B_\infty|^{-1}$ prefactor, thus leading to a reduced
  variance of the {\rm G}-process. Note that this is quite intricate to
  deduce proper information on the variance of the fluctuation process
  $\eta$, since the terms appearing in~\eqref{eq.CLT} are a priori
  dependent. Nonetheless, it will be shown through the numerical
  experiments of the next subsection that the variance of fluctuation
  process reduces when the size of the mini-batches increases (see in
  particular Figure~\ref{fig:mont_red_var}).
\end{rem}

\noindent
Theorem~\ref{thm:clt} is proved in Section~\ref{sec.TH2}, following
inspiration from the previous
works~\cite{fernandez1997hilbertian,jourdainAIHP,delarue}.  The
starting point to prove Theorem~\ref{thm:clt}, consists in proving,
like in the current literature \cite{sirignano2020clt}, that
$(\eta^N)_{N\ge1}\subset \mathcal D(\mathbf R_+,\mathcal
H^{-J_0,j_0}(\mathbf R^d))$ is relatively compact (see
Propositions~\ref{prop relative compactness}). We then prove that the
whole sequence $(\eta^N)_{N\ge1}$ converges in distribution to the
unique weak solution of~\eqref{eq.CLT} in Section~\ref{sec.conTCL}.
\medskip

\noindent
When $\beta=3/4$, $(\eta^N)_{N\ge1}$ is still relatively compact in
$\mathcal D(\mathbf R_+,\mathcal H^{-J_0+1,j_0}(\mathbf R^d))$ (see
Proposition~\ref{prop relative compactness}) but the derivation of the
limit equation satisfied by its limit points is more tricky. However,
in a specific case (when $d=1$ and the test function is
$\mathsf f_2:x\in \mathbf R\mapsto |x|^2$),
Proposition~\ref{pr.beta34} below suggests how the
equation~\eqref{eq.CLT} might be perturbed, as shown numerically in
Figure~\ref{fig.explosion_w_inset} and more precisely in the inset.

\begin{prop}\label{pr.beta34}
  Let $\beta=3/4$ and assume that
  conditions~\ref{as:batch}-\ref{as:batch limite} hold.  Let $\eta$ be
  a limit point of $(\eta^N)_{N\ge 1}$ in
  $\mathcal D(\mathbf R_+,\mathcal H^{-J_0+1,j_0}(\mathbf R))$ (see
  Proposition~\ref{prop relative compactness}).  Then, $\eta_0=\nu_0$
  in distribution (see Lemma~\ref{le.eta0}), and there exist a
  $\mathcal D(\mathbf R_+,\mathcal
  H^{-J_0+1,j_0}(\mathbf{R}^d))$-valued process $ \eta^*$ and a {\rm
    G}-process
  $\mathscr G^*\in \mathcal C(\mathbf R_+,\mathcal
  H^{-J_0,j_0}(\mathbf{R}))$ such that $\eta=\eta_*$ in distribution,
  and a.s.  for every $t\in\mathbf R_+$,
\begin{align}\label{eq.clt.beta=34}
\langle \mathsf f_2,\eta^*_t\rangle-\langle \mathsf f_2,\eta_0^*\rangle&=\int_0^t\int_{\mathcal{X}\times\mathcal{Y}}\alpha(y-\langle\sigma_*(\cdot,x),\overline{\mu}_s\rangle)\langle\nabla \mathsf f_2\cdot\nabla\sigma_*(\cdot,x),\eta^*_s\rangle\pi(\di x,\di y)\nonumber\\
&\quad -\int_0^t\int_{\mathcal{X}\times\mathcal{Y}}\alpha\langle\sigma_*(\cdot,x),\eta^*_s\rangle\langle\nabla \mathsf f_2\cdot\nabla\sigma_*(\cdot,x),\overline \mu _s\rangle\pi(\di x,\di y) +\langle \mathsf f_2,\mathscr G^*_t\rangle+t\mathbf E[\mathsf f_2(\ve_1^1)].
\end{align}  
\end{prop}


\subsection{Numerical Experiments}
We now illustrate numerically the results derived in the previous
sections. First, we consider a regression task on simulated
data, based upon an example of \cite{mei2018mean}. More precisely, we
consider~\eqref{neural network g} with $\sigma_*(W^i,x)=f(W^i\cdot x)$
where
\begin{align*}
f(t)=
\begin{cases} -2.5\ \ \text{if}\ \  t\leq 0.5,\\
10t-7.5\ \ \text{if}\ \ 0.5\leq t\leq 1.5, \\
 7.5 \ \ \text{if} \ \ t\geq 1.5.
 \end{cases}
\end{align*}  
The distribution $\pi$ of the data is defined as follows: with probability $1/2$, $y=1$ and $x\sim\mathcal{N}(0, (1+0.2)^2I_d)$ and, with probability $1/2$, $y=-1$ and $x\sim\mathcal{N}(0, (1-0.2)^2I_d)$.
This setting satisfies the assumptions of Theorems~\ref{thm:lln}
and~\ref{thm:clt}, except~\ref{as:sigma}, due to the fact that $f$ is
not differentiable at $t=0.5$ and $t=1.5$ (a smooth modification of
$f$ around those points would tackle this problem and would not change
the numerical results).

Then, we consider a typical classification task on the
MNIST dataset. The neural network we consider is fully connected with
one-hidden layer of $N$ neurons and ReLU activation
function\footnote{ReLU function $(\cdot)_+$: $u\in\mathbf R\mapsto 0$
  if $u<0$, $u$ if $u\ge 0$.}. The last layer is a softmax layer (we
consider one-hot encoding and use Keras and Tensorflow
librairies). Given a data $x\in \mathbf R^d$ ($d=784$ here), the
neural network returns
$\hat y=\mathrm{softmax}((W^{\mathbf o, c}\cdot W^{\mathbf
  h}(x))_{c=0}^9)$ where
$W^{\mathbf h}(x)=((W^{\mathbf h,i}\cdot x)_+)_{i=1}^N$ is the hidden
layer ($W^{\mathbf h,i}\in\mathbf R^d$ is the weight of the $i$-th
neuron) and 
$W^{\mathbf o, c}\in\mathbf R^N$ is the weight of the output layer
corresponding to class $c$. The total number of trainable parameters
is thus $dN+10N$.
The neural network is trained with respect to the categorical
cross-entropy loss.  This case is not covered by our mathematical
analysis and the motivation here is to show numerical evidence that the variance
reduction derived in Theorem~\ref{thm:clt} is still valid in this case.

\noi \textbf{Variance Reduction with increasing mini-batch size.}  We
illustrate here that the variance of the limiting fluctuation process
decreases with the mini-batch size, even though we only have a mathematical
structure of the variance of the {\rm G}-process (see~\eqref{eq.cov2}
together with Remark~\ref{rem.red_var}). On both experiments, we
consider a fixed mini-batch size during the training (i.e. $|B_k|=|B|$ for
all $k\in\mathbf N$). We first consider the regression task. Consider
$\mathsf L=1000$ neural networks (initialized and trained
independently) whose $N=800$ initial neurons are drawn independently
according to $\mu_0=\mathcal{N}(0,\frac{0.8^2}{d}I_d)$. For each
neural network, we run $k=1000$ iterations of the SGD
algorithm~\eqref{algorithm} and compute
$\mathsf m_\ell:=\langle
\|\cdot\|_2,\mu_{t}^N\rangle=\frac{1}{N}\sum_{i=1}^N\|W_{k}^i\|_2$,
where $\ell\in\{1,\dots,\mathsf L\}$, $t=k/N=1.25$ and
$\|w\|_2:=\sqrt{\sum_{j=1}^dw_j^2}$. Finally, we compute the empirical
variance of this quantity, i.e.,
$$\mathsf V:=\widehat{\Var}(\mathsf m_1,\dots,\mathsf m_{\mathsf
  L})=\frac{1}{\mathsf L}\sum_{\ell=1}^{\mathsf L}\Big(\mathsf
m_\ell-\frac{1}{\mathsf L}\sum_{\ell'=1}^{\mathsf L}\mathsf
m_{\ell'}\Big)^2.$$ and display for different mini-batch sizes $|B|$ in
Figure~\ref{fig:mont_red_var} the obtained boxplots from 10 samples of
$\mathsf V$. The other parameters are $d=40$, $\alpha=0.1$, $\beta=1$,
and the noise is $\ve_k^i\sim\mathcal{N}(0,0.01I_d)$.

Second, we turn to the classification task. Consider $\mathsf L=30$
neural networks (initialized and trained independently) with $N=10000$
neurons on the hidden-layer, until iteration $k=3000$ of the SGD
algorithm ($t=k/N=0.3$), and compute the mean of the weight of the
output layer corresponding to class 0, $i.e.$, for each
$\ell=1,\dots,\mathsf L$, we compute
$\mathsf m_\ell:=\frac 1N\sum_{j=1}^{N}W^{\mathbf
  o,0,j}_{k}$. Finally, we compute the empirical variance of this
quantity, i.e.,
$\mathsf V=\widehat{\Var}(\mathsf m_1,\dots,\mathsf m_{\mathsf L})$
and exhibit for different sizes $|B|$ the boxplots obtained with 10
samples of $\mathsf V$ in Figure~\ref{fig:mont_red_var}.

\begin{figure}[!htbp]
  \caption{Variance $\mathsf V$ reduction of the fluctuation process
    with increasing mini-batch size. {\bf Left:} Regression task on
    simulated data. $\mathsf V$ is an empirical estimation from $1000$
    realisations of the variance of
    $\langle\|\cdot\|_2, \mu_t^N\rangle$, where $N =800$ and
    $t=1.25$. The other parameters are $d=40$, $\alpha=0.1$,
    $\beta=1$, and the noise is
    $\ve_k^i\sim\mathcal{N}(0,0.01I_d)$. The boxplots are obtained
    with 10 samples of $\mathsf V$.  {\bf Right:} Classification task
    on MNIST dataset. $\mathsf V$ is an empirical estimation from $30$
    realisations of the variance of
    $\frac 1N\sum_{j=1}^{N}W^{\mathbf o,0,j}_{k}$, where $N=10000$ and
    $k=3000$ ($t=0.3$). The boxplots are obtained with 10 samples of
    $\mathsf V.$}\label{fig:mont_red_var}
\hspace*{-1cm}\includegraphics[scale=0.5]{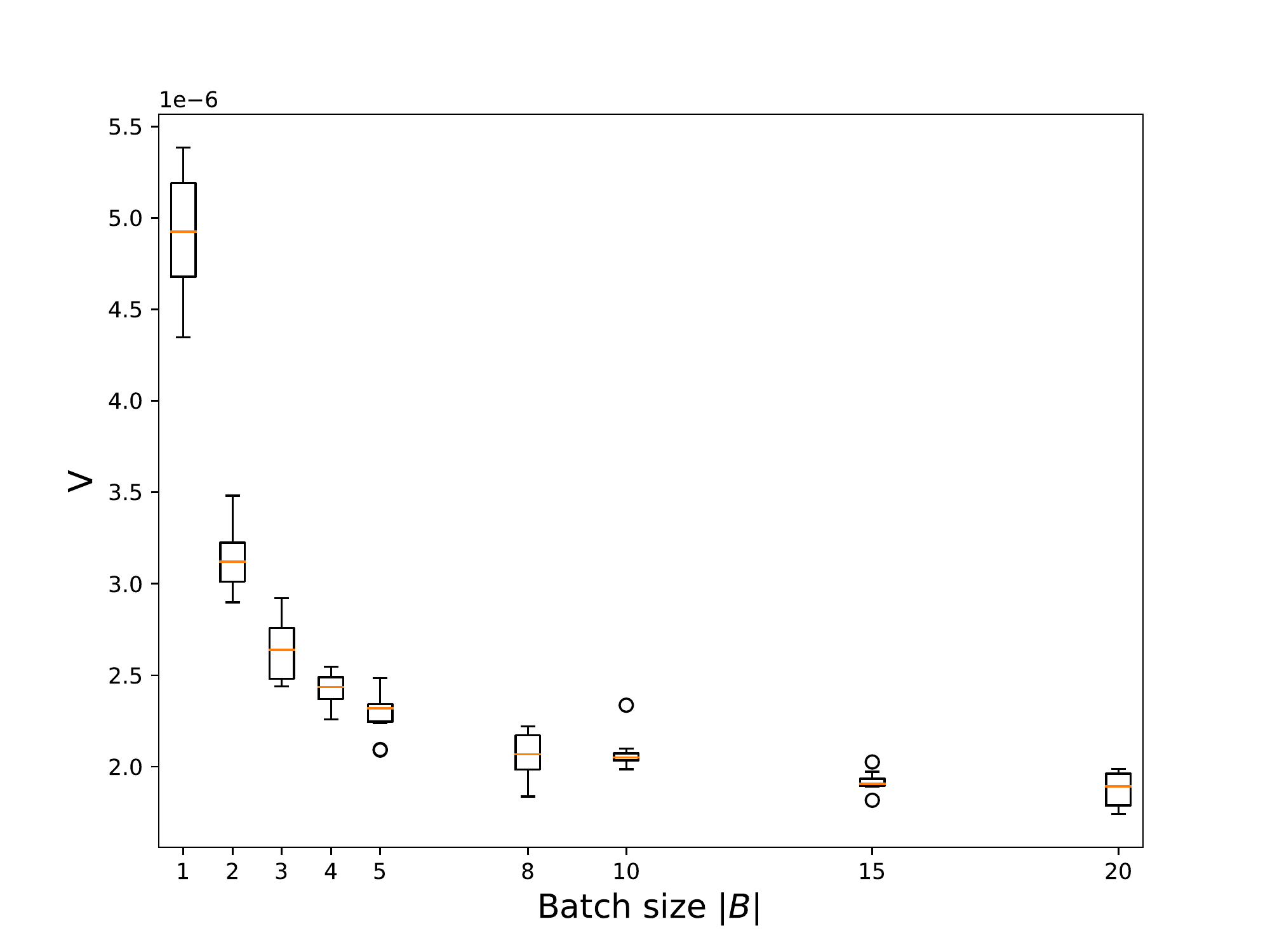} 
\includegraphics[scale=0.5]{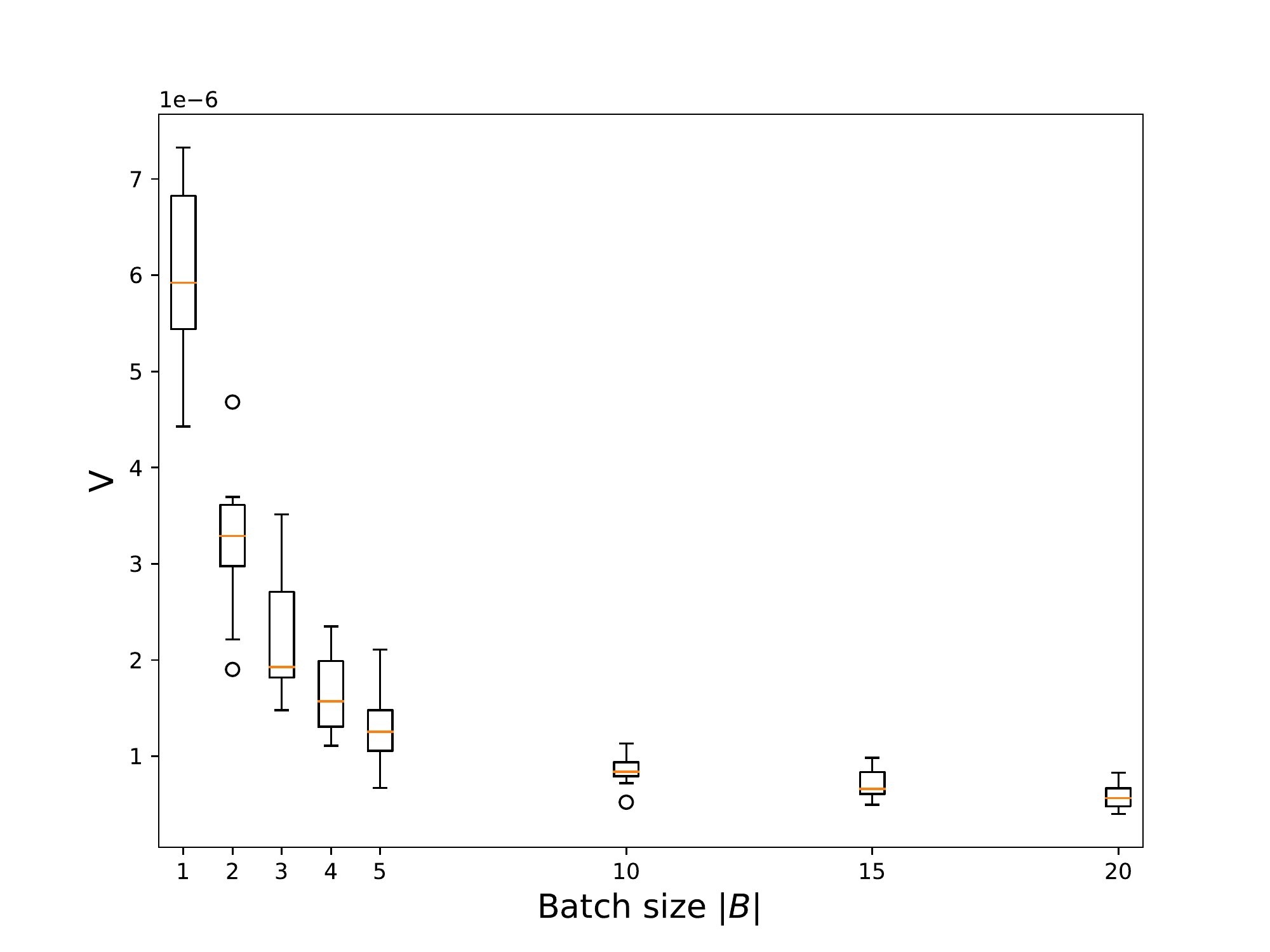} 
\end{figure}

\noi \textbf{Central Limit Theorem.} We focus here on the regression
task.
For different values of $\beta$, we plot in
Figure~\ref{fig.explosion_w_inset}
$\langle\mathsf f_2,\eta_t^N\rangle$ for $0\leq t\leq 8$ (recall
$\mathsf f_2(x)=|x|^2$), to show the agreement of
$\langle\mathsf f_2, \eta_t^N\rangle$ for different values of
$\beta > 3/4$, corresponding to the regime of~\eqref{eq.CLT}, and the
divergence from it when $\beta\leq 3/4$. For $\beta=3/4$, we also
illustrate the regime derived in Proposition~\ref{pr.beta34}. The
parameters chosen are $d=|B|=1$, $N=20000$, $\alpha=0.1$ and
$\ve_k^i\sim\mathcal N(0,0.01)$. The procedure to obtain the plots is
as follows. We first compute $\langle \mathsf f_2,\mu_t^N\rangle$ (we
repeat this procedure 20000 times to get confidence intervals).  Then,
we approximate $\langle \mathsf f_2,\bar\mu_t\rangle$ by
$\langle \mathsf f_2,\mu_t^{N'}\rangle$ where $N'=250000$. On
Figure~\ref{fig.explosion_w_inset}, we plot
$\sqrt N(\langle \mathsf f_2,\mu_t^N\rangle-\langle \mathsf
f_2,\mu_t^{N'}\rangle)\simeq\langle \mathsf f_2,\eta_t^N\rangle$ as a
function of $t$.

%

%

\begin{figure}\caption{Time evolution of the fluctuation process for
    different values of $\beta$ on the regression task, with
    $\mathsf f_2:x\in\mathbf R\mapsto|x|^2$, $N=20000$, $d=|B|=1$,
    $\alpha=0.1$ and $\ve_k^i\sim\mathcal N(0,0.01)$. Confidence
    intervals are obtained from $20000$ realisations. The case $\beta > 3/4$
    is driven by~\eqref{eq.CLT}. The case $\beta=3/4$ is driven
    by~\eqref{eq.clt.beta=34}. The case $\beta < 3/4$ is not covered
    by our analysis. The inset exhibits the linear term in time
    appearing in~\eqref{eq.clt.beta=34}. }
\label{fig.explosion_w_inset}
\includegraphics[scale=0.6]{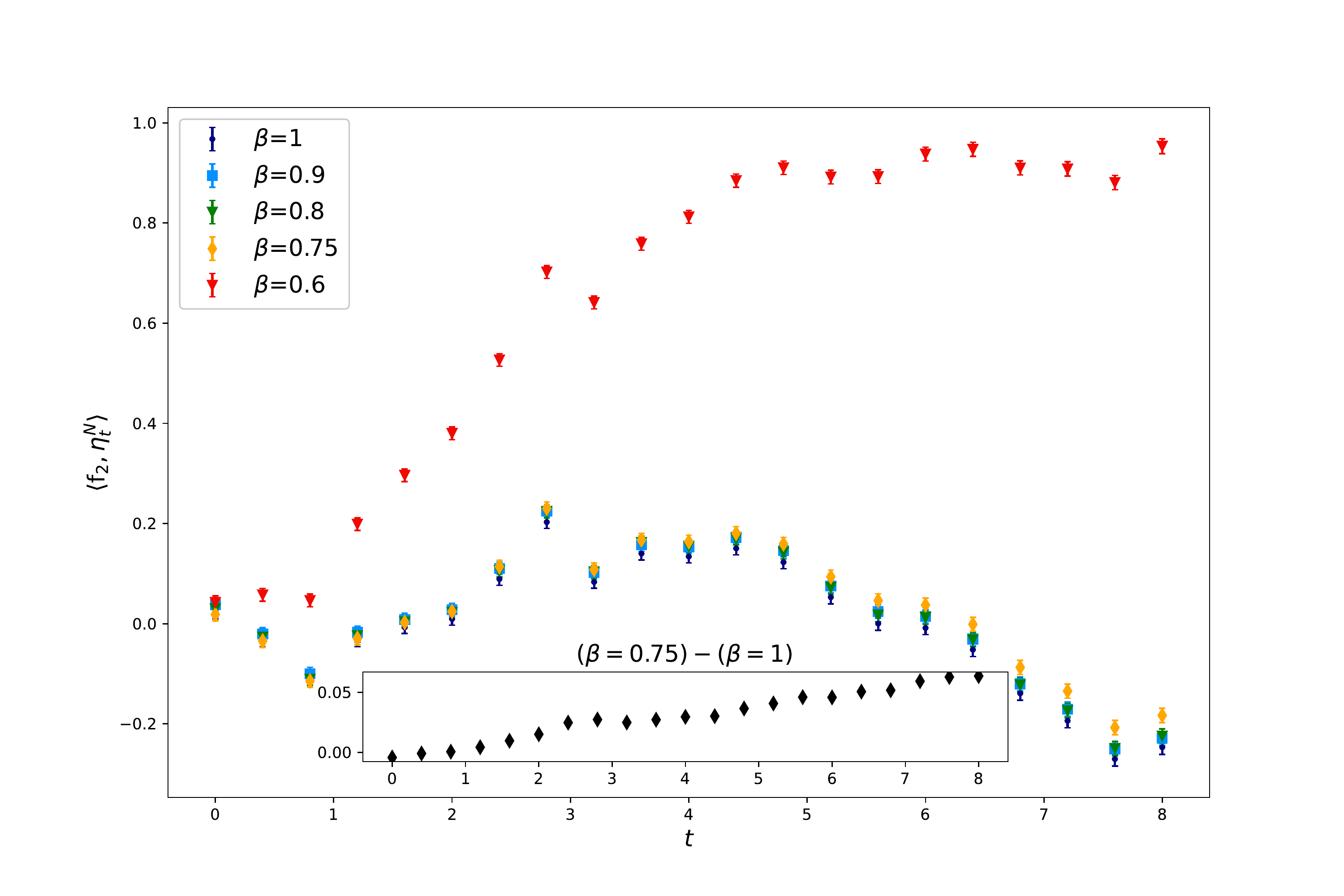} 
\end{figure}

\section{Proof of Theorem~\ref{thm:lln}}
 \label{sec.THMLLN}

\subsection{Pre-limit equation  and remainder terms}
\label{sec.PREL}
In this section, we derive the so-called pre-limit equation~\eqref{eq:
  prelim mu}. We then show that the remainder terms in this equation
are negligible as $N\to +\infty$.

\subsubsection{Pre-limit equation\label{subsec:prelim}}
In this section, we introduce several (random) operators acting on
$\mathcal C^{2,\gamma_*}(\mathbf R^d)$. Using~\ref{as:sigma}
and~\ref{as:data}, it is easy to check that all these operators
belong a.s. to the dual of $\mathcal C^{2,\gamma_*}(\mathbf
R^d)$. The duality bracket we use in this section then is the one for
the duality in $\mathcal C^{2,\gamma_*}(\mathbf R^d)$.  Let us
consider $f\in \mathcal C^{2,\gamma_*}(\mathbf R^d)$.  The
Taylor-Lagrange formula yields, for $N\geq 1$ and $k\in\textbf{N}$,
\begin{align*}
\langle f,\nu_{k+1}^N\rangle -\langle f,\nu_k^N\rangle &=\fsn f(W_{k+1}^i)-f(W_k^i)\nonumber \\
&=\fsn \nabla f(W_k^i)\cdot (\dtki) +
\frac{1}{2N}\sum_{i=1}^N(W_{k+1}^i-W_k^i)^T\nabla^2f(\widehat{W}_k^i)(W_{k+1}^i-W_k^i),
\end{align*}
where, for all $i\in \{1,\dots,N\}$, $\widehat{W}_k^i\in (W_k^i,W_{k+1}^i)$. 
Using~\eqref{algorithm}, we have \begin{align}\label{W. eq : nk1-nk}
\langle f,\nu_{k+1}^N\rangle -\langle f,\nu_k^N\rangle &=\fsn \nabla f(W_k^i)\cdot \left[\frac{\alpha}{N|B_k|}\sum_{(x,y)\in B_k}(y-g_{W^k}^N(x))\nabla_W\sigma_*(W_k^i,x)+\frac{\ve_k^i}{N^\beta}\right]+\langle f,R_k^{N}\rangle\nonumber\\
&=\frac{\alpha}{N|B_k|}\sum_{(x,y)\in B_k}(y-\langle\sigma_*(\cdot,x),\nu_k^N\rangle )\langle\nabla f\cdot\nabla\sigma_*(\cdot,x),\nu_k^N\rangle  \nonumber\\
&\quad +\frac{1}{N^{1+\beta}}\sum_{i=1}^N\nabla f(W_k^i)\cdot\ve_k^i+\langle f,R_k^{N}\rangle,
\end{align} 
where, for $N\geq1$, $k\in\textbf{N}$ and $i=1,\dots,N$,
\begin{equation}\label{R_k^{2,N}}
\langle f,R_k^{N}\rangle:=\frac{1}{2N}\sum_{i=1}^N(W_{k+1}^i-W_k^i)^T\nabla^2f(\widehat{W}_k^i)(W_{k+1}^i-W_k^i).
\end{equation}
For $k\in\textbf{N}$, we define:
\begin{align}
\langle f,D_k^N\rangle &:=\frac{\alpha}{N}\int_{\mathcal{X}\times\mathcal{Y}} (y-\langle\sigma_*(\cdot,x),\nu_k^N\rangle \langle\nabla f\cdot\nabla\sigma_*(\cdot,x),\nu_k^N\rangle \pi(\di x,\di y),\label{def Dk}
\\
\langle f,M_k^N\rangle &:=\frac{\alpha}{N|B_k|}\sum_{(x,y)\in B_k}(y-\langle\sigma_*(\cdot,x),\nu_k^N\rangle )\langle\nabla f\cdot\nabla\sigma_*(\cdot,x),\nu_k^N\rangle -\langle f,D_k^N\rangle .\label{def Mk}
\end{align}
Equation~\eqref{W. eq : nk1-nk} writes, for $k\in\textbf{N}$,
\begin{equation}\label{f nu_k+1-f nu_k}
\langle f,\nu_{k+1}^N\rangle -\langle f,\nu_k^N\rangle =\langle f,D_k^N\rangle +\langle f,M_k^N\rangle +\langle f,R_k^{N}\rangle+\frac{1}{N^{1+\beta}}\sum_{i=1}^N\nabla f(W_k^i)\cdot\ve_k^i.
\end{equation}
Define  for $N\geq1$  and $t\in \mathbf R_+$:   
\begin{equation}\label{eq.def-M}
\langle f,D^N_t\rangle :=\sum_{k=0}^{\lfloor Nt\rfloor-1}\langle f,D_k^N\rangle \  \text{ and } \ \langle f,M^N_t\rangle :=\sum_{k=0}^{\lfloor Nt\rfloor-1}\langle f,M_k^N\rangle, 
\end{equation}
with the convention that $\sum_0^{-1}=0$ (which occurs if and only if
$0\le t<1/N$).  It will be proved later that
$\{t \mapsto \langle f, M_t^N\rangle,t\in \mathbf R_+\}$ is a
martingale (see indeed Lemma~\ref{le.martingale}), hence the notation.
One has, for $t\in\rp$,
\begin{equation*}\begin{split}
\langle f,D^N_t\rangle &=\sum_{k=0}^{\lfloor Nt\rfloor-1}\int_{\frac{k}{N}}^{\frac{k+1}{N}}\int_{\mathcal{X}\times\mathcal{Y}}\alpha(y-\langle\sigma_*(\cdot,x),\nu_k^N\rangle )\langle\nabla f\cdot\nabla\sigma_*(\cdot,x),\nu_k^N\rangle \pi(\di x,\di y)\di s\\
&=\sum_{k=0}^{\lfloor Nt\rfloor-1}\int_{\frac{k}{N}}^{\frac{k+1}{N}}\int_{\mathcal{X}\times\mathcal{Y}}\alpha(y-\langle\sigma_*(\cdot,x),\mu_s^N\rangle )\langle\nabla f\cdot\nabla\sigma_*(\cdot,x),\mu_s^N\rangle \pi(\di x,\di y)\di s\\
&=\int_0^t\int_{\mathcal{X}\times\mathcal{Y}}\alpha(y-\langle\sigma_*(\cdot,x),\mu_s^N\rangle )\langle\nabla f\cdot\nabla\sigma_*(\cdot,x),\mu_s^N\rangle \pi(\di x,\di y)\di s+\langle f,V_t^N\rangle,
\end{split}\end{equation*}
where $\langle f,V_t^N\rangle$, for $t\in\rp$: 
\begin{equation}\label{eq.defV}
\langle f,V_t^N\rangle:=-\int_{\frac{\lfloor Nt\rfloor}{N}}^t\int_{\mathcal{X}\times\mathcal{Y}}\alpha(y-\langle\sigma_*(\cdot,x),\mu_s^N\rangle )\langle\nabla f\cdot\nabla\sigma_*(\cdot,x),\mu_s^N\rangle \pi(\di x,\di y)\di s.
\end{equation}
Therefore, using~\eqref{f nu_k+1-f nu_k}, we obtain that the scaled
empirical measure process $\lbrace t\mapsto\mu_t^N, t\in\rp\rbrace$
satisfies the following pre-limit equation : for
$f\in\mathcal C^{2,\gamma_*}(\mathbf R^d)$, $N\geq 1$ and $t\in\rp$,
\begin{align}\label{eq: prelim mu}
\langle f,\mu_t^N\rangle -\langle f,\mu_0^N\rangle &=\sum_{k=0}^{\lfloor Nt\rfloor-1}\langle f,\nu_{k+1}^N\rangle -\langle f,\nu_k^N\rangle \nonumber\\
&=\langle f,D^N_t\rangle +\langle f,M^N_t\rangle +\sum_{k=0}^{\lfloor Nt\rfloor-1}\langle f,R_k^{N}\rangle+ \frac{1}{N^{1+\beta}}\sum_{k=0}^{\lfloor Nt\rfloor-1}\sum_{i=1}^N\nabla f(W_k^i)\cdot\ve_k^i\nonumber\\
&=\int_0^t\int_{\mathcal{X}\times\mathcal{Y}}\alpha(y-\langle\sigma_*(\cdot,x),\mu_s^N\rangle )\langle\nabla f\cdot\nabla\sigma_*(\cdot,x),\mu_s^N\rangle \pi(\di x,\di y)\di s\nonumber\\
&\quad +\langle f,M^N_t\rangle +\langle f,V_t^N\rangle+\sum_{k=0}^{\lfloor Nt\rfloor-1}\langle f,R_k^{N}\rangle+\frac{1}{N^{1+\beta}}\sum_{k=0}^{\lfloor Nt\rfloor-1}\sum_{i=1}^N\nabla f(W_k^i)\cdot\ve_k^i. 
\end{align} 
In the next section, we study the four last terms of~\eqref{eq: prelim mu}.

\subsubsection{The remainder terms\label{subsec: study rem terms} in
  (\ref{eq: prelim mu}) are negligible}


The aim of this section is to show that the last four terms
of~\eqref{eq: prelim mu} vanish as $N\rightarrow+\infty$. This is the
purpose of Lemma~\ref{lem:remainder terms}. The following result will
be used several times in this work.

\begin{lem}
  \label{le.W} Let $\beta\ge 1/2$ and
  assume~\ref{as:batch}-\ref{as:noise}. Then, for all $T>0$, there
  exists a constant $C<+\infty$ such that for all $N\geq 1,$
  $i\in\lbrace1,\dots,N\rbrace$ and
  $k\in\lbrace0,\dots,\lfloor NT\rfloor\rbrace$,
$$\mathbf{E}\left[|W_k^i|^{8\gamma_*}\right]\leq C.$$
\end{lem}

\begin{proof} 
  Let us recall the following convexity inequality : for $m,p\geq 1$
  and $x_1,\dots,x_p\in\mathbf{R}_+$,
\begin{equation}\label{convexity inequality}
\Big(\sum_{l=1}^mx_l\Big)^p\leq m^{p-1}\sum_{l=1}^mx_l^p.
\end{equation}
Let $C>0$ denotes a constant, independent of
$i\in\lbrace1,\dots,N\rbrace$ and $0\leq k\leq\lfloor NT\rfloor$,
which can change from one occurence to another. Set $p=8\gamma_*$. For
$i\in\lbrace 1,\dots,N\rbrace$ and $1\leq k\leq\lfloor NT\rfloor$, we
have, using~\eqref{algorithm} and~\ref{as:sigma} :
\begin{align*} 
|W_k^i|&\leq |W_0^i|+\Big|\sum_{j=0}^{k-1}W_{j+1}^i-W_j^i\Big| \leq |W_0^i|+\frac{C}{N }\sum_{j=0}^{k-1} \frac{1}{|B_j|}\sum_{(x,y)\in B_j}(|y|+C)+\frac{1}{N^\beta}\Big|\sum_{j=0}^{k-1}\ve_j^i\Big|.
\end{align*}
Thus, by~\eqref{convexity inequality}, 
\begin{align*}
|W_k^i|^p&\leq C\Big[|W_0^i|^p+\frac{1}{N}\sum_{j=0}^{k-1} \frac{1}{|B_j|^p}\Big(\sum_{(x,y)\in B_j}(|y|+C)\Big)^p+\frac{1}{N^{p\beta}}\Big|\sum_{j=0}^{k-1}\ve_j^i\Big|^p\Big]\\
&\le C\Big[|W_0^i|^p+\frac{1}{N}\sum_{j=0}^{k-1} \frac{1}{|B_j|} \sum_{(x,y)\in B_j}(|y|+C)^p+\frac{1}{N^{p\beta}}\Big|\sum_{j=0}^{k-1}\ve_j^i\Big|^p\Big].
\end{align*}
We have:    
\begin{align*}
\textbf{E}\Big [\frac{1}{|B_j|} \sum_{(x,y)\in B_j}(|y|+C)^p\Big ]&\le C\Big[\textbf{E}\Big [\frac{1}{|B_j|} \sum_{n=1}^{|B_j|}|y_j^n|^p \Big ] +1\Big],
\end{align*}
and, using~\ref{as:batch} and~\eqref{as:data}, it holds for $j\ge 0$:
\begin{align}
\nonumber
\textbf{E}\Big [\frac{1}{|B_j|} \sum_{n=1}^{|B_j|}|y_j^n|^p \Big ]=   \sum_{q=1}^{+\infty} \, \textbf{E}\Big [  \frac{ \mathbf 1_{|B_j|=q}}{q}\, \sum_{n=1}^{q}|y_j^n|^p \Big  ]&=   \sum_{q=1}^{+\infty} \,  \frac{1}{q}\, \sum_{n=1}^{q} \, \textbf{E}\big [  |y_j^n|^p  \mathbf 1_{|B_j|=q} \big  ]\\
\nonumber
&=  \sum_{q=1}^{+\infty} \,  \frac{1}{q}\, \sum_{n=1}^{q} \, \textbf{E}\big [  |y_j^n|^p  \big  ]\, \textbf{E}\big [  \mathbf 1_{|B_j|=q} \big  ]\\
\label{eq.conditio}
&= \textbf{E}\big [  |y_1^1|^p  \big  ]<+\infty.
\end{align}
Thus, using the two previous inequalities, we deduce that: 
$$ \textbf{E}\Big[\frac{1}{N}\sum_{j=0}^{k-1} \frac{1}{|B_j|} \sum_{(x,y)\in B_j}(|y|+C)^p\Big]\leq C.$$
 By~\ref{as:moments initiaux}, $\textbf{E}\left[|W_0^i|^p\right]\le C$.  In addition, we have that, for  $i\in\lbrace 1,\dots,N\rbrace$, 
\begin{align*}
\Big|\sum_{j=0}^{k-1}\ve_j^i\Big|^p \leq C\sum_{l=1}^d\Big|\sum_{j=0}^{k-1}\ve_j^{i,l}\Big|^p
\end{align*}
Since we deal with the sum of centered independent Gaussian random
variables, we have that,
for all $i\in\lbrace1,\dots,N\rbrace$ and $l\in\lbrace1,\dots,d\rbrace$, 

\begin{equation*}\label{eq.BGau}
\textbf{E}\Big[\Big|\sum_{j=0}^{k-1}\ve_j^{i,l}\Big|^p\Big]
\leq Ck^{p/2}\leq CN^{p/2}.
\end{equation*}
Putting all these inequalities together, we obtain that
$\textbf{E}\left[|W_k^i|^p\right]\leq
C\left[1+\frac{N^{p/2}}{N^{p\beta}}\right]\leq C$ (recall $\beta\ge
1/2)$. This concludes the proof of the lemma.
\end{proof}



\begin{lem}
\label{lem:remainder terms}  Let $\beta\ge 1/2$  and assume~\ref{as:batch}-\ref{as:noise}. Then, for all $T>0$ there exists $C<\infty$ such that for all $N\geq1$ and $f\in\mathcal C^{2,\gamma_*}(\mathbf R^d)$,
\begin{enumerate}[label=(\roman*), ref=\textit{(\roman*)}]
\item \label{borne E[R_k]} $\max_{0\leq k<\lfloor NT\rfloor}\mathbf{E}\left[|\langle f,R_k^{N}\rangle|\right]\leq C\|f\|_{\mathcal C^{2,\gamma_*
}}\left[\frac{1}{N^2}+\frac{1}{N^{2\beta}}\right]$. 
\item \label{borne E[V]} $\sup_{t\in[0,T]}\mathbf{E}\left[|\langle f,V_t^N\rangle|\right]\leq  {C\|f\|_{\mathcal C^{2,\gamma_*}}}/ {N}$.
\item \label{borne E[f,M^N]} 
$\sup_{t\in[0,T]}\mathbf{E}\left[|\langle f,M_t^N\rangle |^2\right]\leq  {C\|f\|_{\mathcal C^{2,\gamma_*}}^2}/ {N}$. 
\item \label{borne E[noise]}$\sup_{t\in[0,T]}\mathbf{E}\left[\left|\frac{1}{N^{1+\beta}}\sum_{k=0}^{\lfloor Nt\rfloor-1}\sum_{i=1}^N \nabla f(W_k^i)\cdot\ve_k^i\right|^2\right]\leq  {C\|f\|_{\mathcal C^{2,\gamma_*}}^2}/ {N^{2\beta}}$. 
\end{enumerate} 
\end{lem}

\begin{proof}Let $T>0$ and $f\in\mathcal C^{2,\gamma_*}(\mathbf
  R^d)$. In what follows, $C>0$ is a constant, independent of
  $N\geq1$, $t\in[0,T]$, $f$, and
  $k\in\lbrace0,\dots,\lfloor NT\rfloor-1\rbrace$, which can change
  from one line to another.  \medskip

\noindent
\textbf{Proof of item~\ref{borne E[R_k]}}.  For
$k\in\lbrace0,\dots,\lfloor NT\rfloor-1\rbrace$, by~\eqref{R_k^{2,N}},
we have
\begin{equation}\label{ineg R_k}
|\langle f,R_k^{N}\rangle|\leq \frac{C\|f\|_{\mathcal C^{2,\gamma_*}}}{N}\sum_{i=1}^N|W_{k+1}^i-W_k^i|^2(1+|\widehat{W}_k^i|^{\gamma_*}).
\end{equation}
On the other hand, by~\eqref{algorithm}, we have:
\begin{equation}\label{pour wasser}
|W_{k+1}^i-W_k^i|\leq \frac{C}{N|B_k|}\sum_{(x,y)\in B_k}\big (|y|+|g_{W_k}^N(x)|\big )+\frac{|\ve_k^i|}{N^{\beta}}.
\end{equation}
By~\eqref{convexity inequality} and the triangle inequality, we deduce 
\begin{equation*}
|W_{k+1}^i-W_k^i|^2\leq C\Big[\frac{1}{N^2|B_k|}\sum_{(x,y)\in B_k}(|y|^2+|g_{W_k}^N(x)|^2)+\frac{|\ve_k^i|^2}{N^{2\beta}}\Big].
\end{equation*}
By definition of $\widehat{W}_k^i$, there exists $\alpha_k^i\in (0,1)$
such that $\widehat{W}_k^i=\alpha_k^iW_k^i+(1-\alpha_k^i)W_{k+1}^i$,
leading, by~\eqref{convexity inequality}, to
$ |\widehat{W}_k^i|^{\gamma_*}\leq
C\left[|W_k^i|^{\gamma_*}+|W_{k+1}^i|^{\gamma_*}\right]$.  Therefore,
\begin{align}
|W_{k+1}^i-W_k^i|^2(1+|\widehat{W}_k^i|^{\gamma_*})&\leq C\Big[\frac{1}{N^2|B_k|}\sum_{(x,y)\in B_k}(|y|^2+|g_{W_k}^N(x)|^2)+\frac{|\ve_k^i|^2}{N^{2\beta}}\Big](1+|W_k^i|^{\gamma_*}+|W_{k+1}^i|^{\gamma_*})\nonumber\\
&\leq \frac{C}{N^2}(1+|W_k^i|^{\gamma_*}+|W_{k+1}^i|^{\gamma_*})^2+\frac{C}{N^2|B_k|}\sum_{(x,y)\in B_k}(|y|^4+|g_{W_k}(x)|^4)\nonumber\\
&\quad +\frac{C}{N^{2\beta}}\left[|\ve_k^i|^4+(1+|W_k^i|^{\gamma_*}+|W_{k+1}^i|^{\gamma_*})^2\right].\label{17}
\end{align}
Plugging~\eqref{17} in~\eqref{ineg R_k}, we obtain
\begin{align}\label{borne r_k}
|\langle f,R_k^{N}\rangle|&\leq \frac{C\|f\|_{\mathcal C^{2,\gamma_*}}}{N}\sum_{i=1}^N\Big[\frac{1}{N^2}(1+|W_k^i|^{\gamma_*}+|W_{k+1}^i|^{\gamma_*})^2+\frac{1}{N^2|B_k|}\sum_{(x,y)\in B_k}(|y|^4+C)\nonumber\\
&\quad +\frac{1}{N^{2\beta}}\left[|\ve_k^i|^4+(1+|W_k^i|^{\gamma_*}+|W_{k+1}^i|^{\gamma_*})^2\right]\Big].
\end{align}
Finally, using Lemma~\ref{le.W},~\ref{as:data}, and~\ref{as:noise},
one deduces that
$ \textbf{E}\left[|\langle f,R_k^{N}\rangle|\right]\leq
C\|f\|_{\mathcal C^{2,\gamma_*}} ({1}/{N^2}+ {1}/{N^{2\beta}})$.
This proves item~\ref{borne E[R_k]}. 
\medskip

\noindent
\textbf{Proof of item~\ref{borne E[V]}}. Let $t\in[0,T]$.   
Since $\sigma_*$ and all its derivatives are bounded
(see~\ref{as:sigma}), it holds for all $s\ge 0$:
\begin{equation}\label{bound sigma,mu}
|\langle\sigma_*(\cdot,x),\mu_s^N\rangle |=\Big|\frac{1}{N}\sum_{i=1}^N\sigma_*(W_{\lfloor Ns\rfloor}^i,x)\Big|\leq C,
\end{equation}
and 
\begin{align}
  |\langle\nabla f\cdot\nabla\sigma_*(\cdot,x),\mu_s^N\rangle |&=\Big|\frac{1}{N}\sum_{i=1}^N\nabla f(W_{\lfloor Ns\rfloor}^i)\cdot\nabla_W\sigma_*(W_{\lfloor Ns\rfloor}^i,x)\Big| \leq \frac{C\|f\|_{\mathcal C^{2,\gamma_*}}}{N}\sum_{i=1}^N(1+|W_{\lfloor Ns\rfloor}^i|^{\gamma_*}). \label{bound nabla f nabla sigma, mu}
\end{align}
Notice that $C$ above is also independent of $x\in \mathcal X$.  Since
$\mathbf E[|y|]<+\infty$ (see (\ref{as:data})), we obtain
\begin{equation}\label{eq.boundLs}
\left|\int_{\mathcal{X}\times\mathcal{Y}}\alpha(y-\langle\sigma_*(\cdot,x),\mu_s^N\rangle)\langle\nabla f\cdot\nabla\sigma_*(\cdot,x),\mu_s^N\rangle\pi(\di x,\di y)\right|\leq \frac{C\|f\|_{\mathcal C^{2,\gamma_*}}}{N}\sum_{i=1}^N(1+|W_{\lfloor Ns\rfloor}^i|^{\gamma_*}).
\end{equation}
Noticing that
$s\in(\frac{\lfloor Nt\rfloor}{N},t)\Rightarrow\lfloor
Ns\rfloor=\lfloor Nt\rfloor$, we obtain (see~\eqref{eq.defV})
\begin{equation}\label{borne v}
|\langle f,V_t^N\rangle|\leq \left( t-\frac{\lfloor Nt\rfloor}{N}\right)\frac{C\|f\|_{\mathcal C^{2,\gamma_*}}}{N}\sum_{i=1}^N(1+|W_{\lfloor Nt\rfloor}^i|^{\gamma_*}).
\end{equation}
Then, by Lemma~\ref{le.W},
$\textbf{E}\left[|\langle f,V_t^N\rangle|\right]\leq {C\|f\|_{\mathcal
    C^{2,\gamma_*}}}/{N}$. This proves item~\ref{borne E[V]}.
\medskip

\noindent
\textbf{Proof of item~\ref{borne E[f,M^N]}}.  Let $t\in[0,T]$.  Recall
the definition of $\mathcal{F}_k^N$ in~\eqref{eq.filtration} and
$\langle f,M_k^N\rangle$ in~\eqref{def Mk}.  \medskip

\noindent
\underline{Step 1}.  In this step we prove that 
\begin{equation}\label{eq.MTT}
\textbf{E}\left[|\langle f,M_k^N\rangle| ^2\right]\leq {C\|f\|_{\mathcal C^{2,\gamma_*}}^2}/{N^2}.
 \end{equation}  
 With the same arguments as those used to get~\eqref{bound sigma,mu}
 and~\eqref{bound nabla f nabla sigma, mu}, we have
\begin{align}\label{ineg <sigma, nu_k>}
|\langle \sigma_*(\cdot,x),\nu_k^N\rangle |\leq C\quad\text{and}\quad
|\langle\nabla f\cdot\nabla\sigma_*(\cdot,x),\nu_k^N\rangle |\leq \frac{C\|f\|_{\mathcal C^{2,\gamma_*}}}{N}\sum_{i=1}^N(1+|W_k^i|^{\gamma_*}). 
\end{align}
Note that $C$ above is also independent of $x\in \mathcal X$.
By~\eqref{convexity inequality} and~\eqref{ineg <sigma, nu_k>}, we
have:
\begin{align}
|\langle f,M_k^N\rangle| ^2&\leq\frac{C}{N^2|B_k|}\sum_{(x,y)\in B_k}(y-\langle \sigma_*(\cdot,x),\nu_k^N\rangle )^2\langle\nabla f\cdot\nabla\sigma_*(\cdot,x),\nu_k^N\rangle ^2+C|\langle f,D_k^N\rangle |^2\nonumber\\
&\leq \frac{C}{N^2|B_k|}\sum_{(x,y)\in B_k}(|y|^2+C)\langle\nabla f\cdot\nabla\sigma_*(\cdot,x),\nu_k^N\rangle^2+C|\langle f,D_k^N\rangle |^2\nonumber\\
&\leq\frac{C\|f\|_{\mathcal C^{2,\gamma_*}}^2}{N^3|B_k|}\sum_{(x,y)\in B_k}(|y|^2+C)\sum_{i=1}^N(1+|W_k^i|^{\gamma_*})^2+C|\langle f,D_k^N\rangle| ^2.\label{borne m_k^2}
\end{align}
On the other hand, it holds since $(W_k^1,\dots,W_k^N)$ is
$\mathcal{F}_k^N$-measurable and by~\ref{as:batch}:
 \begin{align}
 \nonumber
 \mathbf E\Big[\frac{1}{|B_k|} \sum_{(x,y)\in B_k}(|y|^2+C)\sum_{i=1}^N(1+|W_k^i|^{\gamma_*})^2\Big]&=\sum_{q\ge 1}\, \frac 1q \, \mathbf E\Big[   \mathbf 1_{|B_k|=q}  \sum_{n=1}^{q}(|y_k^n|^2+C)\sum_{i=1}^N(1+|W_k^i|^{\gamma_*})^2  \Big]\\
 \nonumber
 &=\sum_{q\ge 1}\, \frac 1q \, \mathbf E\Big[   \mathbf 1_{|B_k|=q}  \sum_{n=1}^{q}(|y_k^n|^2+C)   \Big] \mathbf E\Big[  \sum_{i=1}^N(1+|W_k^i|^{\gamma_*})^2  \Big]\\
 \nonumber
 &=\sum_{q\ge 1}\, \frac 1q \, \mathbf E\Big[   \mathbf 1_{|B_k|=q}     \Big]\mathbf E\Big[  \sum_{n=1}^{q}(|y_k^n|^2+C)   \Big]  \mathbf E\Big[  \sum_{i=1}^N(1+|W_k^i|^{\gamma_*})^2  \Big]\\
 \nonumber
 &=\sum_{q\ge 1} \mathbf E\Big[   \mathbf 1_{|B_k|=q}     \Big]\mathbf E\Big[ |y_1^1|^2+C  \Big]  \mathbf E\Big[  \sum_{i=1}^N(1+|W_k^i|^{\gamma_*})^2  \Big]\\
 \label{eq.EcalculM}
 &=\mathbf E\Big[ |y_1^1|^2+C   \Big]  \mathbf E\Big[  \sum_{i=1}^N(1+|W_k^i|^{\gamma_*})^2  \Big] \le CN,
 \end{align}
 where we have used Lemma~\ref{le.W} and~\ref{as:data} for the last
 inequality.  Consequently, one has:
\begin{equation}\label{eq.Mk2}
\textbf{E}\left[|\langle f,M_k^N\rangle |^2\right] 
 \leq \frac{C\|f\|_{\mathcal C^{2,\gamma_*}}^2}{N^2}+C\textbf{E}\left[\langle f,D_k^N\rangle ^2\right].
\end{equation}
On the other hand, we easily obtain with similar arguments that
$\textbf{E}\left[|\langle f,D_k^N\rangle| ^2\right]\le
{C\|f\|_{\mathcal C^{2,\gamma_*}}^2}/{N^2}$. Together
with~\eqref{eq.Mk2}, this ends the proof of~\eqref{eq.MTT}.  \medskip

\noindent
\underline{Step 2}.  
In this step we prove that  for all $k\ge 0$:
\begin{align}\label{eq.MkE=0}
\textbf{E}\left[\langle f,M_k^N\rangle |\mathcal{F}_k^N\right]=0.
\end{align}
For ease of notation, we set 
\begin{equation}\label{eq.QQ-d}
\mathrm Q^N[f](x,y,\{W_k^i\}_{i=1,\ldots,N})=\big (y-\langle\sigma_*(\cdot,x),\nu_k^N\rangle \big )\langle\nabla f\cdot\nabla\sigma_*(\cdot,x),\nu_k^N\rangle.
\end{equation}
With this notation, we have (see~\eqref{def Dk})
$\langle f,D_k^N\rangle=\frac{\alpha}{N}\int_{\mathcal X\times
  \mathcal Y}\mathrm Q^N[f](x,y,\{W_k^i\}_{i=1,\ldots,N}) \pi(\di x,\di
y)$ and
$\langle f,M_k^N\rangle= \frac{\alpha}{N|B_k|}\sum_{(x,y)\in B_k}
\mathrm Q^N[f](x,y,\{W_k^i\}_{i=1,\ldots,N})- \langle f,D_k^N\rangle$.
It then holds:
\begin{align*}
\mathbf E\Big [\frac{1}{|B_k|}\sum_{(x,y)\in B_k}\mathrm Q^N[f](x,y,\{W_k^i\}_{i=1,\ldots,N}) \, \big |\, \mathcal{F}_k^N\Big ]&=\mathbf E\Big [\frac{1}{|B_k|}\sum_{n=1}^{|B_k|}\mathrm Q^N[f](x_k^n,y_k^n,\{W_k^i\}_{i=1,\ldots,N})  \, \big |\, \mathcal{F}_k^N\Big ]\\
&=\sum_{q\ge 1}  \frac1q \mathbf E\Big [  \mathbf 1_{|B_k|=q}  \sum_{n=1}^{q}\mathrm Q^N[f](x_k^n,y_k^n,\{W_k^i\}_{i=1,\ldots,N})\, \big |\, \mathcal{F}_k^N \Big ].
\end{align*}
Since $(W_k^1,\dots,W_k^N)$ is $\mathcal{F}_k^N$-measurable,
$\big(|B_k|, ((x_k^n,y_k^n))_{n\ge 1}\big)\indep \mathcal {F}_k^N$,
and $|B_k|\indep ((x_k^n,y_k^n))_{n\ge 1}$ (see~\ref{as:batch}), we
deduce that
\begin{align*}
 & \mathbf E\Big [  \mathbf 1_{|B_k|=q}  \sum_{n=1}^{q}\mathrm Q^N[f](x_k^n,y_k^n,\{W_k^i\}_{i=1,\ldots,N}) \, \big |\, \mathcal{F}_k^N \Big ]\\
  &=\mathbf E\Big [  \mathbf 1_{|B_k|=q}  \sum_{n=1}^{q}\mathrm Q^N[f](x_k^n,y_k^n,\{w_k^i\}_{i=1,\ldots,N}) \Big ]\Big|_{\{w_k^i\}_i=\{W_k^i\}_i}\\
 &=\mathbf E\big [  \mathbf 1_{|B_k|=q}   \big ]\mathbf E\Big [  \sum_{n=1}^{q}\mathrm Q^N[f](x_k^n,y_k^n,\{w_k^i\}_{i=1,\ldots,N})   \Big ] \Big|_{\{w_k^i\}_i=\{W_k^i\}_i}\\
 &=q \mathbf E\big [  \mathbf 1_{|B_k|=q}   \big ]\mathbf E\Big [  \mathrm Q^N[f](x_1^1,y_1^1,\{w_k^i\}_{i=1,\ldots,N})   \Big ] \Big|_{\{w_k^i\}_i=\{W_k^i\}_i}\\
 &= q \, \frac N\alpha  \mathbf E\big [ \mathbf 1_{|B_k|=q}  \big ] \, \langle f,D_k^N\rangle,
\end{align*}
where we have used~\ref{as:data} to deduce the last two equalities. We
have thus proved that
$$\mathbf E\Big [\frac{\alpha}{N|B_k|}\sum_{(x,y)\in B_k}\mathrm Q^N[f](x_k^n,y_k^n,\{W_k^i\}_{i=1,\ldots,N})   \, \big |\, \mathcal{F}_k^N\Big ]=  \langle f,D_k^N\rangle.$$
Therefore, using in addition that $\mathbf E [ \langle f,D_k^N\rangle   |\, \mathcal{F}_k^N  ]=\langle f,D_k^N\rangle$ (because $(W_k^1,\dots,W_k^N)$ is $\mathcal{F}_k^N$-measurable), 
we finally deduce~\eqref{eq.MkE=0}. 
\medskip

\noindent
\underline{Step 3}. We now end the proof of item~\ref{borne E[f,M^N]}.
If $j>k$, $\langle f,M_k^N\rangle$ is $\mathcal{F}_j^N$-measurable
(because $\langle f,M_k^N\rangle$ is
$\mathcal{F}_{k+1}^N$-measurable). Then, using also~\eqref{eq.MkE=0},
one obtains that for $j>k$:
\begin{align}\label{eq.Mk==}
\textbf{E}\left[\langle f,M_k^N\rangle\langle f,M_j^N\rangle \right]=\textbf{E}\left[\langle f,M_k^N\rangle\textbf{E}\left[\langle f,M_j^N\rangle |\mathcal{F}_j^N\right]\right]=\textbf{E}\left[\langle f,M_k^N\rangle \times 0\right]=0.
\end{align}
We then have (see~\eqref{eq.def-M}):
\begin{align}\label{22}
\textbf{E}\left[|\langle f,M^N_t\rangle| ^2\right]& =\sum_{k=0}^{\lfloor Nt\rfloor-1}\textbf{E}\left[|\langle f,M_k^N\rangle| ^2\right].
\end{align}  
Plugging~\eqref{eq.MTT} in~\eqref{22}  implies item~\ref{borne E[f,M^N]}.  \medskip

\noindent
\textbf{Proof of item~\ref{borne E[noise]}}.  Let $t\in[0,T]$.  By
Lemma~\ref{le.W}, $\nabla f(W_k^i)\cdot\ve_k^i$ is square-integrable
for all $k\in \mathbf N$ and $i\in \{1,\ldots,N\}$.  From the equality
$$
\textbf{E}\Big[\Big|\sum_{k=0}^{\lfloor Nt\rfloor-1}\sum_{i=1}^N \nabla f(W_k^i)\cdot\ve_k^i\Big|^2\Big]=\sum_{j,k=0}^{\lfloor Nt\rfloor-1} \sum_{i,\ell=1}^N\textbf{E}\big[ \nabla f(W_k^i)\cdot\ve_k^i\,  \nabla f(W_j^\ell)\cdot\ve_j^\ell \big].
$$
Recall that $ W_a^b$ is $\mathcal{F}_a^N$-measurable for all
$a\in \mathbf N$ and $b\in \{1,\ldots,N\}$, and that
$\ve_a^b\indep \mathcal F_a^N$ (see~\ref{as:noise}).  Let $e_q$
denotes the $q$-th element of the canonical basis of $\mathbf R^d$
($q\in \{1,\ldots,d\})$.  Assume that
$0\le j<k\le \lfloor Nt\rfloor-1$. Then, $\varepsilon_j^\ell$ is
$\mathcal{F}_k^N$-measurable, and it holds for all
$i, \ell\in \{1,\ldots,N\}$:
\begin{align}\label{eq.ekekj-0}
\textbf{E}\big[ \nabla f(W_k^i)\cdot\ve_k^i\,  \nabla f(W_j^\ell)\cdot\ve_j^\ell \big]=\sum_{n,m=1}^d  \textbf{E}\Big[ \partial_{e_n} f(W_k^i) \,  \partial_{e_m} f(W_j^\ell) \,    \ve_j^\ell\cdot e_m \Big]  \mathbf E \big [   \ve_k^i\cdot e_n     \big] =0.
\end{align} 
because $\ve_k^i\sim\mathcal{N}(0,I_{d})$ (see~\ref{as:noise}). On the
other hand, using~\ref{as:noise}, we have for all
$0\le k\le \lfloor Nt\rfloor-1$ and when
$i\neq \ell\in \{1,\ldots,N\}$:
\begin{align}\label{eq.ekekj}
\textbf{E}\big[ \nabla f(W_k^i)\cdot\ve_k^i\,  \nabla f(W_k^\ell)\cdot\ve_k^\ell \big]&=\sum_{n,m=1}^d\textbf{E}\Big[ \, \partial_{e_n} f(W_k^i) \,  \partial_{e_m} f(W_k^\ell)   \,\Big]   \mathbf E \big[ \ve_k^i\cdot e_n       \big] \mathbf E \big[     \ve_k^\ell\cdot e_m     \big]=0.
\end{align} 
 Consequently, we have:
\begin{align*}
\textbf{E}\Big[ \Big |\sum_{k=0}^{\lfloor Nt\rfloor-1}\sum_{i=1}^N\nabla f(W_k^i)\cdot\ve_k^i\Big|^2\Big]=\sum_{k=0}^{\lfloor Nt\rfloor-1}\sum_{i=1}^N\textbf{E}\big[|\nabla f(W_k^i)\cdot\ve_k^i|^2\big].
\end{align*}
Using the Cauchy-Schwarz inequality, we deduce, using also
Lemma~\ref{le.W} and~\ref{as:noise}, that:
\begin{align}\label{eq333}
\textbf{E}\Big[\Big|\sum_{k=0}^{\lfloor Nt\rfloor-1}\sum_{i=1}^N\nabla f(W_k^i)\cdot\ve_k^i\Big|^2\Big]&\leq\sum_{k=0}^{\lfloor Nt\rfloor-1}\sum_{i=1}^N\textbf{E}\left[|\nabla f(W_k^i)|^2\right]\textbf{E}\left[|\ve_k^i|^2\right]\nonumber\\
&\leq C\|f\|_{\mathcal C^{2,\gamma_*}}^2\sum_{k=0}^{\lfloor Nt\rfloor-1}\sum_{i=1}^N\textbf{E}\left[(1+|W_k^i|^{\gamma_*})^2\right] \leq C\|f\|_{\mathcal C^{2,\gamma_*}}^2N^2.
\end{align}
This proves~\ref{borne E[noise]}. The proof of
Lemma~\ref{lem:remainder terms} is complete.
\end{proof}


We now want to pass to the limit in~\eqref{eq: prelim mu}. To this
end, we first prove that $(\mu^N)_{N\ge 1}$ is relatively compact in
$\skodual$. This is the purpose of the following section.


\subsection{Relative compactness in $\mathcal D(\mathbf R_+, \mathcal P_{\gamma}(\mathbf R^d))$ and convergence to the limit equation \label{sec:rel compact}}
 In this section, we show that
 $(\mu^N)_{N\ge 1}$ is relatively compact in $\mathcal D(\mathbf R_+, \mathcal  P_{\gamma}(\mathbf R^d))$. Then, we prove 
that any limit point of
 $ (\mu^N)_{N\ge 1} $
 satisfies a.s.~\eqref{eq limite}.


\subsubsection{Relative compactness  in $\mathcal D(\mathbf R_+, \mathcal  P_{\gamma}(\mathbf R^d))$}
\label{sec.RC-LLN}

In this section we prove the following result.

{\prop{\label{lem:rel comp Sobolev}Let $\beta\ge 1/2$ and assume that the
    conditions~\ref{as:batch}-\ref{as:noise} hold. Then,
    $(\mu^N)_{N\ge 1}$ is relatively compact  in  in $\mathcal D(\mathbf R_+, \mathcal P_{\gamma}(\mathbf R^d))$}.    }

\medskip
 
 We first recall the following standard result.

\begin{prop}\label{prop.compact_wasserstein}
 Let $q>p\ge 1$ and $C>0$. The set $\mathscr K_C^{q}:=\{\mu\in \mathcal P_p(\mathbf R^{d}), \int_{\mathbf R^{d}}|x|^{q}\mu(\di x)\leq C\}$ is compact.
\end{prop}

 We have the following  result.
  
\begin{lem}
  \label{lem:cc sobo} 
   Let $\beta\ge 1/2$ and assume
 that~\ref{as:batch}-\ref{as:noise} hold. Then, for every
  $T>0$, there exists $C>0$ such that for all $f\in\mathcal C^{2,\gamma_*}(\mathbf R^d)$,
\begin{align}\label{cc eq obj}
\sup_{N\geq1}\mathbf{E}\Big[ \sup_{t\in[0,T]}\langle f,\mu_t^N\rangle ^2\Big]\le C \Vert f\Vert_{\mathcal C^{2,\gamma_*}}^2.
\end{align} 
\end{lem}

\begin{proof} Let $T>0$ and
  $f\in\mathcal C^{2,\gamma_*}(\mathbf R^d)$.  All along the proof,
  $C<\infty$ denotes a constant independent of $t\in [0,T]$, $N\ge 1$,
  $k\in \{0,\ldots, \lfloor Nt\rfloor\}$, $i\in \{1,\ldots,N\}$, and
  $f\in\mathcal C^{2,\gamma_*}(\mathbf R^d)$, which can change from
  one occurence to another.  From~\eqref{eq: prelim mu}, we have:
\begin{align}\label{cc eq depart}
\sup_{t\in[0,T]}\langle f,\mu_t^N\rangle ^2&\leq C\Big[\langle f,\mu_0^N\rangle ^2+\int_0^T\int_{\mathcal{X}\times\mathcal{Y}}(y-\langle\sigma_*(\cdot,x),\mu_s^N\rangle )^2\langle\nabla f\cdot\nabla\sigma_*(\cdot,x),\mu_s^N\rangle ^2\pi(\di x,\di y)\di s \nonumber\\
&\quad +\sup_{t\in[0,T]}\langle f,M^N_t\rangle ^2+\sup_{t\in[0,T]}|\langle f,V_t^N\rangle|^2+\sup_{t\in[0,T]}\Big|\sum_{k=0}^{\lfloor Nt\rfloor-1}\langle f,R_k^{N}\rangle\Big|^2\nonumber\\
&\quad +\frac{1}{N^{2+2\beta}}\sup_{t\in[0,T]}\Big|\sum_{k=0}^{\lfloor Nt\rfloor-1}\sum_{i=1}^N\nabla f(W_k^i)\cdot\ve_k^i\Big|^2\, \Big].
\end{align}
We now study each term of the right-hand side of~\eqref{cc eq
  depart}. Let us deal with the first term in the right-hand side
of~\eqref{cc eq depart}.  Using~\ref{as:moments initiaux}
and~\eqref{convexity inequality}, it holds:
\begin{align*}
\textbf{E}\left[\langle f,\mu_0^N\rangle ^2\right]=\textbf{E}\left[\langle f,\nu_0^N\rangle ^2\right]=\textbf{E}\Big[\Big|\frac{1}{N}\sum_{i=1}^N f(W_0^i)\Big|^2\Big] \leq \frac{1}{N}\sum_{i=1}^N\textbf{E}\left[|f(W_0^i)|^2\right] &\le  \frac{\|f\|^2_{\mathcal C^{2,\gamma_*}}}{N}\sum_{i=1}^N\textbf{E}\left[(1+|W_0^i|^{\gamma_*})^2\right]\\
&\leq C\|f\|^2_{\mathcal C^{2,\gamma_*}}.
\end{align*}
For the second term in the right-hand side of~\eqref{cc eq depart}, we
have since $\mathbf E[|y|^2]<+\infty$ (see (\ref{as:data})) and
using~\eqref{bound sigma,mu},~\eqref{bound nabla f nabla sigma, mu},
and Lemma~\ref{le.W}:
\begin{align*}
\textbf{E}\left[\int_0^T\int_{\mathcal{X}\times\mathcal{Y}}(y-\langle\sigma_*(\cdot,x),\mu_s^N\rangle )^2\langle\nabla f\cdot\nabla\sigma_*(\cdot,x),\mu_s^N\rangle ^2\pi(\di x,\di y)\di s\right]\leq C\|f\|^2_{\mathcal C^{2,\gamma_*}}.
\end{align*}
Let us deal with the third term in the right-hand side of~\eqref{cc eq
  depart}.  By~\eqref{eq.def-M} and~\eqref{convexity inequality}, we
have, for $t\in[0,T]$,
\begin{align*}
\sup_{t\in[0,T]}|\langle f,M_t^N\rangle |^2\leq\lfloor NT\rfloor\sum_{k=0}^{\lfloor NT\rfloor-1}\langle f,M_k^N\rangle ^2. 
\end{align*}
Hence, using~\eqref{eq.MTT}, we obtain that
$ \textbf{E}\big[\sup_{t\in[0,T]}\langle f,M_t^N\rangle ^2\big]\leq
C\|f\|_{\mathcal C^{2,\gamma_*}}^2$. Let us deal with the fourth term
in the right-hand side of~\eqref{cc eq depart}.  From~\eqref{convexity
  inequality} and~\eqref{borne v},
\begin{align*}
\sup_{t\in[0,T]}|\langle f,V_t^N\rangle|^2\leq \frac{C\|f\|_{\mathcal C^{2,\gamma_*}}^2}{N^3}\sum_{i=1}^N\max_{0\leq k\leq\lfloor NT\rfloor}(1+|W_k^i|^{\gamma_*})^2,
\end{align*}
which leads to 
\begin{align}\nonumber
\textbf{E}\Big[\sup_{t\in[0,T]}|\langle f,V_t^N\rangle|^2\Big]&\leq \frac{C\|f\|_{\mathcal C^{2,\gamma_*}}^2}{N^3}\sum_{i=1}^N\textbf{E}\left[\max_{0\leq k\leq\lfloor NT\rfloor}(1+|W_k^i|^{\gamma_*})^2\right]\\
\label{eq.supVt}
&\leq \frac{C\|f\|_{\mathcal C^{2,\gamma_*}}^2}{N^3}\sum_{i=1}^N\sqrt{\sum_{k=0}^{\lfloor NT\rfloor}\textbf{E}\left[(1+|W_k^i|^{\gamma_*})^4\right]} \leq \frac{C\|f\|_{\mathcal C^{2,\gamma_*}}^2}{N^{3/2}}.
\end{align}
Let us now consider the fifth term in the right-hand side of~\eqref{cc
  eq depart}. From~\eqref{borne r_k} and~\eqref{convexity inequality},
we have
\begin{align*}
|\langle f,R_k^{N}\rangle|^2&\leq \frac{C\|f\|^2_{\mathcal C^{2,\gamma_*}}}{N}\sum_{i=1}^N\Big[\frac{1}{N^4}(1+|W_k^i|^{\gamma_*}+|W_{k+1}^i|^{\gamma_*})^4+\frac{1}{N^4|B_k|}\sum_{(x,y)\in B_k}(|y|^4+C)^2\\
&\quad+\frac{1}{N^{4\beta}}\left[|\ve_k^i|^8+(1+|W_k^i|^{\gamma_*}+|W_{k+1}^i|^{\gamma_*})^4\right]\Big].
\end{align*}
Then, by~\ref{as:data} and~\ref{as:noise} together with
Lemma~\ref{le.W} and~\eqref{eq.conditio}, we obtain
\begin{align}\label{borneERk]}
\textbf{E}\left[|\langle f,R_k^{N}\rangle|^2\right]\leq C\|f\|^2_{\mathcal C^{2,\gamma_*}}\Big[\frac{1}{N^4}+\frac{1}{N^{4\beta}}\Big].
\end{align}
Therefore, using also~\eqref{convexity inequality}, it holds:
\begin{align}\label{eq.borneERk2}
\textbf{E}\Big[\sup_{t\in[0,T]}\Big|\sum_{k=0}^{\lfloor Nt\rfloor-1}\langle f,R_k^{N}\rangle\Big|^2\Big]&\leq \lfloor NT\rfloor\sum_{k=0}^{\lfloor NT\rfloor-1}\textbf{E}\left[|\langle f,R_k^{N}\rangle|^2\right] \leq C\|f\|^2_{\mathcal C^{2,\gamma_*}}\Big[\frac{1}{N^2}+\frac{N^2}{N^{4\beta}}\Big].
\end{align}
Let us deal with the last term in the right-hand side of~\eqref{cc eq
  depart}.  Using the same arguments leading to~\eqref{eq333} together
with~\eqref{convexity inequality} and~\eqref{eq.ekekj} we have
\begin{align}
\nonumber
\frac{1}{N^{2+2\beta}}\textbf{E}\Big[\sup_{t\in[0,T]}\Big|\sum_{k=0}^{\lfloor Nt\rfloor-1}\sum_{i=1}^N\nabla f(W_k^i)\cdot\ve_k^i\Big|^2\Big]&\leq \frac{C}{N^{2+2\beta}}\textbf{E}\Big[ \sup_{t\in[0,T]} \lfloor Nt\rfloor \sum_{k=0}^{\lfloor Nt\rfloor-1} \Big|\sum_{i=1}^N\nabla f(W_k^i)\cdot\ve_k^i\Big|^2\Big]\\
\nonumber
&\le \frac{C \lfloor NT\rfloor }{N^{2+2\beta}} \textbf{E}\Big[  \sum_{k=0}^{\lfloor NT\rfloor-1} \Big|\sum_{i=1}^N\nabla f(W_k^i)\cdot\ve_k^i\Big|^2\Big]\\
\nonumber
&\leq \frac{C}{N^{1+2\beta}}\sum_{k=0}^{\lfloor NT\rfloor-1}\textbf{E}\Big[ \Big|\sum_{i=1}^N\nabla f(W_k^i)\cdot\ve_k^i\Big|^2\Big]\\
\label{eq.boundRk1}
&\leq \frac{C}{N^{1+2\beta}}\sum_{k=0}^{\lfloor NT\rfloor-1}\sum_{i=1}^N\textbf{E}\big[|\nabla f(W_k^i)\cdot\ve_k^i|^2\big]\leq \frac{C\|f\|^2_{\mathcal C^{2,\gamma_*}}}{N^{2\beta-1}}.
\end{align}
Plugging all these previous bounds in~\eqref{cc eq depart}, we obtain
(recall that $\beta\ge 1/2$), for all
$f\in \mathcal C^{2,\gamma_*}(\mathbf R^d)$, $
\textbf{E} [\sup_{t\in[0,T]}\langle f,\mu_t^N\rangle ^2 ] \leq C\|f\|_{\mathcal C^{2,\gamma_*}}^2$. 
This
proves~\eqref{cc eq obj} and ends the proof of the lemma.
\end{proof}

 Lemma \ref{lem:cc sobo}   provides  the following compact containment for $(\mu^N)_{N\ge1}$ in $\mathcal D(\mathbf R_+,\mathcal P_{\gamma}(\mathbf R^d))$.
 
 \begin{cor}
 \label{lem_cc_mu^N}
Assume $\beta\ge 1/2$ and~\ref{as:batch}-\ref{as:noise}.  Let $0<\epsilon<1$. For every $T>0$, 
\begin{equation}
\sup_{N\ge1}\mathbf E\Big[\sup_{t\in[0,T]}\int_{\mathbf R^{d}}|w|^{\gamma+\epsilon}\mu_t^N(\di w) \Big] <+\infty.
\end{equation}
\end{cor}

\begin{proof}
 Recall $\gamma_*-\gamma=1$. Thus, it holds $f:w\mapsto(1-\chi(w))|w|^{\gamma +\epsilon}\in \mathcal C^{2,\gamma_*}(\mathbf R^{d})$ since $\gamma_*>\gamma +\epsilon$. The result follows from Lemma \ref{lem:cc sobo}. 
\end{proof}

 The following result will also be needed.

\begin{lem}
  \label{lem:rc sobo} Assume~\ref{as:batch}-\ref{as:noise} and
  $\beta \ge 1/2$. For all $T>0$, there exists $C>0$ such that for all
  $\delta>0$ and $0\leq r<t\leq T$ such that $t-r\leq\delta$, one has
  for all $N\ge 1$ and $f\in \mathcal C^{2,\gamma_*}(\mathbf R^d)$:
\begin{align}\label{rel comp eq7}
\mathbf{E}\left[|\langle f,\mu_t^N\rangle -\langle f,\mu_r^N\rangle |^2\right]\leq C\|f\|^2_{\mathcal C^{2,\gamma_*}}\left[\delta^2+\frac{1}{N}+\frac{1}{N^2}+(N\delta+1)\left[\frac{1}{N^4}+\frac{1}{N^{4\beta}}\right]+\frac{1}{N^{2\beta}}\right].
\end{align}
\end{lem}


\begin{proof}  
  Let $\delta>0$ and $0\leq r<t\leq T$ such that $t-r\leq\delta$.  Let
  $f\in \mathcal C^{2,\gamma_*}(\mathbf{R}^d)$.  In the following,
  $C>0$ is a constant independent of $t$, $r$, $\delta$, $N$, and $f$,
  which can change from one occurence to another.  From~\eqref{eq:
    prelim mu}, we have
\begin{align*}
\langle f,\mu_t^N\rangle -\langle f,\mu_r^N\rangle &=\int_r^t\int_{\mathcal{X}\times\mathcal{Y}}\alpha(y-\langle\sigma_*(\cdot,x),\mu_s^N\rangle )\langle\nabla f\cdot\nabla\sigma_*(\cdot,x),\mu_s^N\rangle \pi(\di x,\di y)\di s\\
&\quad +\langle f,M_t^N\rangle -\langle f,M_r^N\rangle +\langle f,V_t^N\rangle-\langle f,V_r^N\rangle +\sum_{k=\lfloor Nr\rfloor}^{\lfloor Nt\rfloor-1}\langle f,R_k^{N}\rangle+\frac{1}{N^{1+\beta}}\sum_{k=\lfloor Nr\rfloor}^{\lfloor Nt\rfloor-1}\sum_{i=1}^N\nabla f(W_k^i)\cdot\ve_k^i.
\end{align*}
Jensen's inequality provides 
\begin{align}\label{rel comp1}
|\langle f,\mu_t^N\rangle -\langle f,\mu_r^N\rangle |^2&\leq C\Big[(t-r)\int_r^t\int_{\mathcal{X}\times\mathcal{Y}}\left|(y-\langle\sigma_*(\cdot,x),\mu_s^N\rangle )\langle\nabla f\cdot\nabla\sigma_*(\cdot,x),\mu_s^N\rangle \right|^2\pi(\di x,\di y)\di s\nonumber\\
&\quad +\left|\langle f,M_t^N\rangle -\langle f,M_r^N\rangle \right|^2+\left|\langle f,V_t^N\rangle-\langle f,V_r^N\rangle\right|^2+\Big|\sum_{k=\lfloor Nr\rfloor}^{\lfloor Nt\rfloor-1}\langle f,R_k^{N}\rangle\Big|^2\nonumber\\
&\quad +\frac{1}{N^{2+2\beta}}\Big|\sum_{k=\lfloor Nr\rfloor}^{\lfloor Nt\rfloor-1}\sum_{i=1}^N\nabla f(W_k^i)\cdot\ve_k^i\Big|^2\Big].
\end{align}
We now study each term of the right-hand side of~\eqref{rel comp1}.
Let us consider the first term in the right-hand side of~\eqref{rel
  comp1}. From~\eqref{bound sigma,mu},~\eqref{bound nabla f nabla
  sigma, mu} and~\eqref{convexity inequality}, we have:
\begin{align*}
\textbf{E}\left[|y-\langle\sigma_*(\cdot,x),\mu_s^N\rangle |^2|\langle\nabla f\cdot\nabla\sigma_*(\cdot,x),\mu_s^N\rangle| ^2\right]&\leq \frac{C\|f\|_{\mathcal C^{2,\gamma_*}}^2}{N}\textbf{E}\Big[(|y|^2+C)\sum_{i=1}^N(1+|W_{\lfloor Ns\rfloor}^i|^{\gamma_*})^2\Big] \leq C\|f\|_{\mathcal C^{2,\gamma_*}}^2,
\end{align*}
where the last inequality follows from~\ref{as:data} and
Lemma~\ref{le.W}. We then have:
\begin{align}
\nonumber
\textbf{E}\left[(t-r)\int_r^t\int_{\mathcal{X}\times\mathcal{Y}}\left((y-\langle\sigma_*(\cdot,x),\mu_s^N\rangle )\langle\nabla f\cdot\nabla\sigma_*(\cdot,x),\mu_s^N\rangle \right)^2\pi(\di x,\di y)\di s\right] &\leq C(t-r)^2\|f\|_{\mathcal C^{2,\gamma_*}}^2\\
&\leq C\delta^2\|f\|_{\mathcal C^{2,\gamma_*}}^2.\label{rel comp eq2}
\end{align}
Let us consider the second term in the right-hand side of~\eqref{rel
  comp1}. From item~\ref{borne E[f,M^N]} of Lemma~\ref{lem:remainder
  terms}, we have
\begin{align}\label{rel comp eq3}
\textbf{E}\left[\left(\langle f,M_t^N\rangle -\langle f,M_r^N\rangle \right)^2\right]\leq 2\textbf{E}\left[\langle f,M_t^N\rangle ^2+\langle f,M_r^N\rangle ^2\right]\leq \frac{C\|f\|_{\mathcal C^{2,\gamma_*}}^2}{N}.
\end{align}
Let us consider the third term in the right-hand side of~\eqref{rel
  comp1}. From~\eqref{borne v} and~\eqref{convexity inequality}, we
have
\begin{align*} 
|\langle f,V_t^N\rangle|^2\leq \frac{C\|f\|_{\mathcal C^{2,\gamma_*}}^2}{N^3}\sum_{i=1}^N(1+|W_{\lfloor Nt\rfloor}^i|^{\gamma_*})^2.
\end{align*}
Therefore, by Lemma~\ref{le.W}, we obtain that:
\begin{align}\label{rel comp eq4}
\textbf{E}\big[\big|\langle f,V_t^N\rangle-\langle f,V_r^N\rangle\big|^2\big]\leq 2\textbf{E}\left[|\langle f,V_t^N\rangle|^2+|\langle f,V_r^N\rangle|^2\right]\leq \frac{C\|f\|_{\mathcal C^{2,\gamma_*}}^2}{N^2}.
\end{align}
Let us consider the fourth term in the right-hand side of~\eqref{rel
  comp1}.
By~\eqref{borneERk]},  
\begin{align}\label{rel comp eq5}
\textbf{E}\Big[\Big|\sum_{k=\lfloor Nr\rfloor}^{\lfloor Nt\rfloor-1}\langle f,R_k^{N}\rangle\Big|^2\Big]&\leq C\|f\|^2_{\mathcal C^{2,\gamma_*}}(\lfloor Nt\rfloor-\lfloor Nr\rfloor)\left[\frac{1}{N^4}+\frac{1}{N^{4\beta}}\right]\nonumber\\
&\leq C\|f\|^2_{\mathcal C^{2,\gamma_*}}(N\delta+1)\left[\frac{1}{N^4}+\frac{1}{N^{4\beta}}\right].
\end{align}
Let us consider the last term in the right-hand side of~\eqref{rel
  comp1}.  By item~\ref{borne E[noise]} in Lemma~\ref{lem:remainder
  terms},
\begin{align}
&\frac{1}{N^{2+2\beta}}\, \textbf{E}\Big[ \Big|\sum_{k=\lfloor Nr\rfloor}^{\lfloor Nt\rfloor-1}\sum_{i=1}^N\nabla f(W_k^i)\cdot\ve_k^i\Big|^2\Big]\nonumber\\
&\leq\frac{2}{N^{2+2\beta}}\, \textbf{E}\Big[\Big|\sum_{k=0}^{\lfloor Nt\rfloor-1}\sum_{i=1}^N\nabla f(W_k^i)\cdot\ve_k^i\Big|^2+\Big|\sum_{k=0}^{\lfloor Nr\rfloor-1}\sum_{i=1}^N\nabla f(W_k^i)\cdot\ve_k^i\Big|^2\Big] \leq \frac{C\|f\|^2_{\mathcal C^{2,\gamma_*}}}{N^{2\beta}}.\label{rel comp eq6}
\end{align}
Using~\eqref{rel comp eq2},~\eqref{rel comp eq3},~\eqref{rel comp
  eq4},~\eqref{rel comp eq5},~\eqref{rel comp eq6}, and ~\eqref{rel
  comp1}, we deduce \eqref{rel comp eq7}. 
%
%
%
\end{proof}

We now collect the results of the  previous lemmata to prove
Proposition~\ref{lem:rel comp Sobolev}.

\begin{proof}[Proof of Proposition~\ref{lem:rel comp Sobolev}]
To prove Proposition~\ref{lem:rel comp Sobolev}, we apply \cite[Theorem 4.6]{jakubowski1986skorokhod} with $E=  \mathcal P_{\gamma}(\mathbf R^{d})$  and $\mathbb F=\{\mathsf V_f, f\in \mathcal C^\infty_c(\mathbf R^{d})\}$ where 
$$\mathsf  V_f: \nu \in\mathcal P_{\gamma}(\mathbf R^{d})\mapsto \langle f, \nu \rangle.$$ 
The set  $\mathbb F$ on $\mathcal P_{\gamma}(\mathbf R^{d})$ satisfies Conditions~\cite[(3.1) and (3.2) in Theorem 3.1]{jakubowski1986skorokhod}. Condition (4.8) there is a consequence of Proposition \ref{prop.compact_wasserstein}, Corollary \ref{lem_cc_mu^N}, together with  Markov's inequality. 
We now prove that~\cite[Condition (4.9)]{jakubowski1986skorokhod} is verified, i.e. let us show that all $f\in \mathcal C^\infty_c(\mathbf R^{d})$,  the sequence $(\langle f,\mu^N\rangle)_{N\ge1}$ is relatively compact in $\mathcal D(\mathbf R_+,\mathbf R)$.  
 To do so, it suffices to use  Lemma \ref{lem:rc sobo}    and     Proposition~\ref{lem:note kurtz} below (with $\mathcal H_1=\mathcal H_2=\mathbf R$ there).   
 In conclusion, according to \cite[Theorem 4.6]{jakubowski1986skorokhod},  $(\mu^N)_{N\ge1}$ is relatively compact in $ \mathcal D(\mathbf R_+,\mathcal P_{\gamma}(\mathbf R^{d}))$.
\end{proof}


\subsubsection{Limit points in   $ \mathcal D(\mathbf R_+,\mathcal P_{\gamma}(\mathbf R^{d}))$ are continuous  in time }\label{subsec: continuity prop sob}

In this section we show that any limit point of $ (\mu^N)_{N\geq1}$   in   $ \mathcal D(\mathbf R_+,\mathcal P_{\gamma}(\mathbf R^{d}))$
belongs a.s. to
  $ \mathcal C(\mathbf R_+,\mathcal P_{1}(\mathbf R^{d}))$.

\begin{prop}\label{p-limit in P}
 Let $\beta>1/2$ and
  assume~\ref{as:batch}-\ref{as:noise}. Consider  $\mu^*\in \mathcal D(\mathbf R_+,\mathcal P_{\gamma}(\mathbf R^{d}))$   a limit point of $(\mu^N)_{N\ge 1}$ in $ \mathcal D(\mathbf R_+,\mathcal P_{\gamma}(\mathbf R^{d}))$. Then,   a.s.   $\mu^*\in\mathcal C(\mathbf R_+,\mathcal P_{1}(\mathbf R^{d}))$.
\end{prop}

\begin{proof}
Let $N'$ be a subsequence such that in distribution  $\mu^{N'} \to  \mu^*$   in  $\mathcal D(\mathbf R_+,\mathcal P_{\gamma}(\mathbf R^{d}))$. 
Because $\mathsf W_1\le \mathsf W_{\gamma}$,   $\mu^{N'} \to  \mu^*$ in distribution  also in $ \mathcal D(\mathbf R_+,\mathcal P_{1}(\mathbf R^{d}))$.
By \cite[Proposition 3.26 in Chapter VI]{jacod2003skorokhod},  $\mu^*\in\mathcal C(\mathbf R_+,\mathcal P_{1}(\mathbf R^{d}))$ a.s.  if  for all $T>0$, $\lim_{N\to +\infty} \mathbf E\big[ \sup_{t\in [0,T]} \mathsf W_1(\mu^N_{t_-},\mu^N_t)  \big]=0$. According to the duality formula \eqref{Kantorovitch Rubinstein}, this is equivalent to   
\begin{equation}\label{eq.L1}
\lim_{N\to +\infty} \mathbf E\Big[ \sup_{t\in [0,T]} \sup_{\Vert f\Vert_{\text{Lip}}\le 1}|\langle f,\mu^N_{t_-}\rangle-\langle f,\mu^N_t\rangle| \Big]=0.
\end{equation}
Let $T>0$ and consider a Lipschitz function $f:\mathbf R^{d}\to \mathbf R$ such that $\Vert f\Vert_{\text{Lip}}\le 1$. One has that  $\langle f,\mu_t^N\rangle=\langle f,\mu_0^N\rangle+ \sum_{k=0}^{\lfloor Nt\rfloor-1}\langle f,\nu_{k+1}^N\rangle-\langle f,\nu_k^N\rangle$ (with the convention $\sum_0^{-1}=0$). Therefore,  the discontinuity points of $t\in [0,T]\mapsto \langle f,\mu_t^N\rangle$ are exactly  $\{1/N, 2/N,\ldots, \lfloor NT\rfloor/N\}$ and for all $t\in [0,T]$, 
\begin{align}\label{eq.Bs}
|\langle f,\mu^N_{t_-}\rangle-\langle f,\mu^N_t\rangle|\le \max_{k=0,\ldots,\lfloor NT\rfloor-1}|\langle f,\nu_{k+1}^N\rangle-\langle f,\nu_{k}^N\rangle|.
\end{align}
Let  $k\in \{0,\ldots,\lfloor NT\rfloor-1\}$. We have using~\eqref{algorithm} and~\ref{as:sigma}: 
\begin{align}\label{eq.Bs2}
|\langle f,\nu_{k+1}^N\rangle-\langle f,\nu_{k}^N\rangle|&\le \frac 1N\sum_{i=1}^N  |W_{k+1}^i-W_k^i|  \le \frac CN \sum_{i=1}^N\Big[ \frac{1}{N  |B_k|}\sum_{(x,y)\in B_k}(|y|+1)+\frac{|\varepsilon_k^i|}{N^\beta}  \Big]=:\beta_k^N
\end{align}
Then, one deduces that 
\begin{align*}
|\beta_k^N|^2 \le   \frac CN \sum_{i=1}^N\Big[ \frac{1}{N^2  |B_k|}\sum_{(x,y)\in B_k}(|y|^2+1)+\frac{|\varepsilon_k^i|^2}{N^{2\beta}}  \Big],
\end{align*}
and hence that $\mathbf E[|\beta_k^N|^2]\le C(1/N^2 +1/N^{2\beta})$ where   $C>0$ is independent of $N\ge 1$ and $k=0,\ldots,\lfloor NT\rfloor-1$.
Then, using \eqref{eq.Bs} and\eqref{eq.Bs2}, 
\begin{align*}
\mathbf E\Big[ \sup_{t\in [0,T]} \sup_{\Vert f\Vert_{\text{Lip}}\le 1}|\langle f,\mu^N_{t_-}\rangle-\langle f,\mu^N_t\rangle| \Big]&\le \mathbf E\Big[ \sup_{\Vert f\Vert_{\text{Lip}}\le 1} \max_{k=0,\ldots,\lfloor NT\rfloor-1}|\langle f,\nu_{k+1}^N\rangle-\langle f,\nu_{k}^N\rangle| \Big]\\
&\le \mathbf E\Big[  \max_{k=0,\ldots,\lfloor NT\rfloor-1}\beta_k^N \Big] \\
&\le \mathbf E\Big[  \sqrt{\sum_{k=0}^{\lfloor NT\rfloor-1} |\beta_k^N|^2 }\Big]\le \sqrt{\mathbf E\Big[   \sum_{k=0}^{\lfloor NT\rfloor-1} |\beta_k^N|^2 \Big]}\le C\big[1/\sqrt N +  \sqrt{ N/N^{2\beta}}\big].
\end{align*}
This proves \eqref{eq.L1} since $\beta>1/2$. 
The proof of Proposition \ref{p-limit in P} is complete. 
\end{proof}

 We end this section with the following result which will be used later in the proof of Theorem \ref{thm:clt}.

\begin{lem}
  \label{lem:mu bar is continuous}
  Let $\beta\ge 1/2$ and
  assume~\ref{as:batch}-\ref{as:noise}.  Then,
  for all $T>0$ there exists $C>0$, 
\begin{align}\label{799}
\mathbf{E}\Big[\sup_{t\in[0,T]}\langle f,\mu_t^N-\mu_{t^-}^N\rangle^2\Big]\leq   C\|f\|^2_{\mathcal C^{2,\gamma_*}}\Big[\frac{1}{N^{3/2}}+\sqrt{\frac{1}{N^7}+\frac{1}{N^{8\beta-1}}}+\frac{\sqrt{N}}{N^{2\beta}}\Big], \ \forall f\in \mathcal C^{2,\gamma_*}(\mathbf{R}^d).
\end{align}
\end{lem}


\begin{proof} 
 The arguments used in the proof of Proposition \ref{p-limit in P} are not sufficient to prove  \eqref{799}. We will rather use~\eqref{eq: prelim mu}. 
Let $T>0$ and  $f\in \mathcal C^{2,\gamma_*}(\mathbf{R}^d)$. 
  In what follows, $C>0$ is a constant, independent of $N\geq1$,
  $k=0,\dots \lfloor NT\rfloor-1$, and $f$, which can change from one
  occurence to another.  Recall that $ t\in[0,T]\mapsto\langle f,\mu_t^N\rangle \in\mathbf{R}$ has
  $\lfloor NT\rfloor$ discontinuities, located at the points
  $\frac{1}{N},\frac{2}{N},\dots,\frac{\lfloor NT\rfloor}{N}$. In addition, from~\eqref{eq: prelim mu},~\eqref{eq.defV},
  and~\eqref{eq.def-M},  for
  $k\in\lbrace1,\dots,\lfloor NT\rfloor\rbrace$, its $k$-th
  discontinuity is equal to
\begin{align}
\nonumber
\mathbf {d}_k^N[f]:=&\langle f,M_{k-1}^N\rangle +\int_{\frac{k-1}{N}}^{\frac{k}{N}}\int_{\mathcal{X}\times\mathcal{Y}}\alpha(y-\langle\sigma_*(\cdot,x),\mu_s^N\rangle)\langle\nabla f\cdot\nabla\sigma_*(\cdot,x),\mu_s^N\rangle\pi(\di x,\di y)\di s\\
\label{eq.disc}
&+\langle f,R_{k-1}^N\rangle+\frac{1}{N^{1+\beta}}\sum_{i=1}^N\nabla f(W_{k-1}^i)\cdot\ve_{k-1}^i.
\end{align}
Thus, \begin{align}\label{cont eq1}
\sup_{t\in[0,T]}|\langle f,\mu_t^N-\mu_{t^-}^N\rangle|^2&\leq \max\big   \lbrace  |\mathbf{d}_{k+1}^N[f]|^2,   \ 0\leq k<\lfloor NT\rfloor  \big \rbrace.
\end{align}
By~\eqref{borne m_k^2} and~\eqref{convexity inequality}, it holds:
\begin{align*}
|\langle f,M_k^N\rangle|^4\leq\frac{C\|f\|_{\mathcal C^{2,\gamma_*}}^4}{N^5|B_k|}\sum_{(x,y)\in B_k}\sum_{i=1}^N(y^4+C)(1+|W_k^i|^{\gamma_*})^4.
\end{align*}
Then, using Lemma~\ref{le.W}, we have with the same computations as
the one made in~\eqref{eq.EcalculM}:
\begin{align}\label{2.40}
\textbf{E}\left[|\langle f,M_k^N\rangle|^4\right]&\leq \frac{C\|f\|^4_{\mathcal C^{2,\gamma_*}}}{N^5}\sum_{i=1}^N\textbf{E}\left[y^4+C\right]\textbf{E}\left[(1+|W_k^i|^{\gamma_*})^4\right] \leq\frac{C\|f\|^4_{\mathcal C^{2,\gamma_*}}}{N^4}. 
\end{align}
Consequently, one has:
 \begin{align}\label{E max f M_k}
\textbf{E}\left[\max_{0\leq k<\lfloor NT\rfloor}\langle f,M_k^N\rangle^2\right]&\leq\Big|\sum_{k=0}^{\lfloor NT\rfloor-1}\textbf{E}\left[\langle f,M_k^N\rangle^4\right]\Big|^{1/2} \leq \frac{C\|f\|^2_{\mathcal C^{2,\gamma_*}}}{N^{3/2}}.
\end{align}
By~\eqref{bound sigma,mu},~\eqref{bound nabla f nabla sigma, mu} and
since $\lfloor Ns\rfloor=k$ when $s\in [k/N,(k+1)/N]$, we have
\begin{align*}
&\Big|\int_{\frac{k}{N}}^{\frac{k+1}{N}}\int_{\mathcal{X}\times\mathcal{Y}}\alpha(y-\langle\sigma_*(\cdot,x),\mu_s^N\rangle)\langle\nabla f\cdot\nabla\sigma_*(\cdot,x),\mu_s^N\rangle\, \pi(\di x,\di y)\, \di s\Big|\\
&\le \frac{C\|f\|_{\mathcal C^{2,\gamma_*}}}{N} \int_{\frac{k}{N}}^{\frac{k+1}{N}}\int_{\mathcal{X}\times\mathcal{Y}}  (|y|+C)\sum_{i=1}^N(1+|W_k^i|^{\gamma_*})\,  \pi(\di x,\di y)\, \di s  \\
&= \frac{C\|f\|_{\mathcal C^{2,\gamma_*}}}{N^2}  \mathbf E[|y|+C] \sum_{i=1}^N(1+|W_k^i|^{\gamma_*})\\
&\le \frac{C\|f\|_{\mathcal C^{2,\gamma_*}}}{N^2}   \sum_{i=1}^N(1+|W_k^i|^{\gamma_*}).
\end{align*}
By~\eqref{convexity inequality} and Lemma~\ref{le.W}, it then holds:
\begin{align*}
&\textbf{E}\Big[ \Big|\int_{\frac{k}{N}}^{\frac{k+1}{N}}\int_{\mathcal{X}\times\mathcal{Y}}\alpha(y-\langle\sigma_*(\cdot,x),\mu_s^N\rangle)\langle\nabla f\cdot\nabla\sigma_*(\cdot,x),\mu_s^N\rangle\, \pi(\di x,\di y)\, \di s\Big|^4\Big]\\
&\le\frac{C\|f\|^4_{\mathcal C^{2,\gamma_*}}N^3 }{N^8}     \textbf{E}\Big[\sum_{i=1}^N(1+|W_k^i|^{\gamma_*})^4\Big]\le \frac{C\|f\|^4_{\mathcal C^{2,\gamma_*}}}{N^4}. 
\end{align*}
Thus, one has:
\begin{align*}
&\textbf{E}\Big[\max_{0\leq k<\lfloor NT\rfloor}\Big|\int_{\frac{k}{N}}^{\frac{k+1}{N}}\int_{\mathcal{X}\times\mathcal{Y}}\alpha(y-\langle\sigma_*(\cdot,x),\mu_s^N\rangle)\langle\nabla f\cdot\nabla\sigma_*(\cdot,x),\mu_s^N\rangle\pi(\di x,\di y)\di s\Big|^2\Big]\\
&\leq\Big |\sum_{k=0}^{\lfloor NT\rfloor-1}\textbf{E}\Big[\Big|\int_{\frac{k}{N}}^{\frac{k+1}{N}}\int_{\mathcal{X}\times\mathcal{Y}}\alpha(y-\langle\sigma_*(\cdot,x),\mu_s^N\rangle)\langle\nabla f\cdot\nabla\sigma_*(\cdot,x),\mu_s^N\rangle\pi(\di x,\di y)\di s\Big|^4\Big]\Big |^{1/2} \leq \frac{C\|f\|^2_{\mathcal C^{2,\gamma_*}}}{N^{3/2}}.
\end{align*}
On the other hand, from~\eqref{borne r_k} and~\eqref{convexity inequality}, we have
\begin{align*}
|\langle f,R_k^{N}\rangle|^4&\leq \frac{C\|f\|^4_{\mathcal C^{2,\gamma_*}}}{N}\sum_{i=1}^N\Bigg[\frac{1}{N^8}(1+|W_k^i|^{\gamma_*}+|W_{k+1}^i|^{\gamma_*})^8+\frac{1}{N^8|B_k|}\sum_{(x,y)\in B_k}(|y|^4+C)^4\\
&\quad +\frac{C}{N^{8\beta}}\left[|\ve_k^i|^{16}+(1+|W_k^i|^{\gamma}+|W_{k+1}^i|^{\gamma_*})^8\right]\Bigg].
\end{align*}
Using Lemma~\ref{le.W},~\ref{as:data}, and the same computations as
those made in~\eqref{eq.conditio}, we deduce that:
$$
\textbf{E}\left[|\langle f,R_k^{N}\rangle|^4\right]\leq C\|f\|^4_{\mathcal C^{2,\gamma}}\big( {1}/{N^8}+ {1}/{N^{8\beta}}\big).$$
Then, it holds:
\begin{align*}
\textbf{E}\Big[\max_{0\leq k<\lfloor NT\rfloor}|\langle f,R_k^{N}\rangle|^2\Big]&\leq\Big |\sum_{k=0}^{\lfloor NT\rfloor-1}\textbf{E}\left[|\langle f,R_k^{N}\rangle|^4\right]\Big |^{1/2} \leq C\|f\|_{\mathcal C^{2,\gamma_*}}^2\Big[\frac{1}{N^7}+\frac{1}{N^{8\beta-1}}\Big]^{1/2}. 
\end{align*}
By~\eqref{convexity inequality},~\ref{as:noise}, and Lemma~\ref{le.W},
\begin{align*}
\textbf{E}\Big[\Big|\sum_{i=1}^N\nabla f(W_k^i)\cdot\ve_k^i\Big|^4\Big]&\leq N^3\sum_{i=1}^N\textbf{E}\left[|\nabla f(W_k^i)\cdot\ve_k^i|^4\right]\\
&\leq N^3\|f\|^4_{\mathcal C^{2,\gamma_*}}\sum_{i=1}^N\textbf{E}\left[(1+|W_k^i|^{\gamma_*})^4\right]\textbf{E}\left[|\ve_k^i|^4\right]\leq C\|f\|^4_{\mathcal C^{2,\gamma_*}}N^4.
\end{align*}
Thus, one deduces that 
\begin{align*}
\textbf{E}\Big[\max_{0\leq k<\lfloor NT\rfloor}\Big|\frac{1}{N^{1+\beta}}\sum_{i=1}^N\nabla f(W_k^i)\cdot\ve_k^i\Big|^2\Big]&\leq\Big |\sum_{k=0}^{\lfloor NT\rfloor-1}\textbf{E}\Big[\Big|\frac{1}{N^{1+\beta}}\sum_{i=1}^N\nabla f(W_k^i)\cdot\ve_k^i\Big|^4\Big]\Big |^{1/2}\\
&\leq\Big |\frac{1}{N^{4+4\beta}}C\|f\|^4_{\mathcal C^{2,\gamma_*}}N^5\Big |^{1/2} \leq C\|f\|_{\mathcal C^{2,\gamma_*}}^2\frac{\sqrt{N}}{N^{2\beta}}.
\end{align*}
Plugging all these previous bounds in~\eqref{cont eq1} implies \eqref{799}.  
\end{proof}
 

\subsubsection{Convergence to the limit equation (\ref{eq limite})}
\label{sec:conv to lim eq}

This section is devoted to prove Proposition~\ref{lem:convergence to
  lim eq} where we show that any limit point of
$(\mu^N)_{N\geq1}$ in   $ \mathcal D(\mathbf R_+,\mathcal P_{\gamma}(\mathbf R^{d}))$  satisfies a.s.~\eqref{eq limite}.

 For $t\in\mathbf{R}_+$ and
$f\in \mathcal C^{1,\gamma}(\mathbf{R}^d)$, we introduce the function
$\boldsymbol{\Lambda}_t[f]:  \mathcal D(\mathbf R_+,\mathcal P_{\gamma}(\mathbf R^{d})) \rightarrow\mathbf{R}$ defined by
\begin{align*} 
\boldsymbol{\Lambda}_t[f]:m\mapsto\left| \langle f,m_t\rangle-\langle f,\mu_0\rangle-\int_0^t\int_{\mathcal{X}\times\mathcal{Y}}\alpha(y-\langle\sigma_*(\cdot,x),m_s\rangle)\langle\nabla f\cdot\nabla\sigma_*(\cdot,x),m_s\rangle\pi(\di x,\di y)\di s\right|.
\end{align*}
To prove that any limit point of the sequence $(\mu^N)_{N\ge 1}$ in
the space $\mathcal D(\mathbf R_+,\mathcal P_{\gamma}(\mathbf R^{d}))$ satisfies~\eqref{eq limite}, we study the
continuity of the function $\boldsymbol{\Lambda}_t[f]$. This is the
purpose of Lemma~\ref{le.Fc}. 

\begin{lem}
  \label{le.Fc}
 For any $t\in\mathbf{R}_+$ and
$f\in \mathcal C^{1,\gamma}(\mathbf{R}^d)$,  the function
  $\boldsymbol{\Lambda}_t[f]$ is well defined. In  addition, let $(m^N)_{N\ge 1}$ be such that  $m^N\to m$ in $\mathcal D(\mathbf R_+,\mathcal P_{\gamma}(\mathbf R^{d}))$. Then, for all
 continuity points $t\in\mathbf{R}_+$ of $m$, 
 $\boldsymbol{\Lambda}_t[f](m^N)\rightarrow
\boldsymbol{\Lambda}_t[f](m)$ as $N\rightarrow+\infty$.
\end{lem}

\begin{proof}  
  In the following $C>0$ is a constant
  independent of $f\in \mathcal C^{1,\gamma}(\mathbf{R}^d)$, $s\in [0,t]$,    $x\in \mathcal X$, and $y\in \mathcal Y$,
  which can change from one occurence to another.  By~\ref{as:moments
    initiaux} 
\begin{equation}\label{eq.mu00}
|\langle f,\mu_0\rangle |=\big|\int_{\mathbf{R}^d}\frac{f(w)}{1+|w|^\gamma}(1+|w|^\gamma)\mu_0(\di w)\big|\leq (1+\mathbf{E}[|W_0^1|^{\gamma}])\|f\|_{\mathcal C^{0,\gamma}} .
\end{equation}
The following result~\cite[Theorem 6.9]{villani2009optimal} will be used many times in the following: 
\begin{equation}
\label{eq.Vil}
\mu_n\to  \mu \text{ in  $\mathcal P_{\gamma}(\mathbf{R}^d)$ iff } \langle g,\mu_n\rangle\to \langle g,\mu\rangle \text{ for all $g: \mathbf{R}^d\to \mathbf R$ continuous s.t.  $\frac{g}{1+|\cdot |^{\gamma}}$ is bounded.}
\end{equation} In particular $u\ge 0\mapsto\langle 1+|\cdot |^{\gamma},m_u\rangle \in \mathcal D(\mathbf R_+, \mathbf R)$ and thus $\sup_{u\in [0,t]} |\langle 1+|\cdot |^{\gamma},m_u\rangle|<+\infty$ for all $t\ge 0$.    
Define the function $\phi_s^{x,y}(m)= \alpha(y-\langle\sigma_*(\cdot,x),m_s\rangle)\langle\nabla f\cdot\nabla\sigma_*(\cdot,x),m_s\rangle $.  
Using~\ref{as:sigma}, one has for $s\in [0,t]$:
\begin{equation} \label{eq.phi-b}
|\phi_s^{x,y}(m)  |\le C(1+|y|) \sup_{u\in [0,t]} |\langle 1+|\cdot |^{ \gamma}, m_u\rangle| \|f\|_{\mathcal C^{1,\gamma}}.
\end{equation} 
Using also~\ref{as:data}, this proves that   $\boldsymbol{\Lambda}_t[f]$ is well defined. 

Let us now consider $(m^N)_{N\ge 1}$   such that  $m^N\to m$ in $\mathcal D(\mathbf R_+,\mathcal P_{\gamma}(\mathbf R^{d}))$. 
Denote by 
  $\mathcal C(m)\subset \mathbf R_+$ the set of continuity points of
  $m$.    From \cite[Proposition 5.2 in Chapter
  3]{ethier2009markov}, we have that for all $t\in \mathcal C(m)$,
  $m_t^N\rightarrow m_t$ in $\mathcal P_{\gamma}(\mathbf{R}^d)$,
  and thus, for all $t\in \mathcal C(m)$, according to \eqref{eq.Vil},
$$\langle f,m_t^N\rangle\underset{N\rightarrow\infty}{\longrightarrow}\langle f,m_t\rangle.$$
For the same reasons, 
for all $s\in[0,t]\cap \mathcal C(m)$ and
$x\in\mathcal{X}$,
$$\langle\sigma_*(\cdot,x),m_s^N\rangle\underset{N\rightarrow\infty}{\longrightarrow}
\langle\sigma_*(\cdot,x),m_s\rangle \quad \text{and}\quad\langle\nabla
f\cdot\nabla\sigma_*(\cdot,x),m_s^N\rangle\underset{N\rightarrow\infty}{\longrightarrow}\langle\nabla
f\cdot\nabla\sigma_*(\cdot,x),m_s\rangle.$$
Since
$\mathbf{R}_+\backslash \mathcal C(m)$ is at most countable (see
\cite[Lemma 5.1 in Chapter 3]{ethier2009markov}), it holds a.e. on $[0,t]\times \mathcal X\times \mathcal Y$, $\phi_s^{x,y}(m^N)\to   \phi_s^{x,y}(m)$. 
Note that   using \cite[Item (b) in Proposition 5.3 in Chapter
  3]{ethier2009markov} together with the triangular inequality: 
  $$\langle |\cdot |^{ \gamma}, m_{u}^N\rangle=\mathsf W_\gamma(\delta_0,m_u^N)\le [\mathsf W_\gamma(\delta_0,m_{\lambda^N_u})+\mathsf W_\gamma(m_{\lambda^N_u},m_u^N)]^\gamma, \, \lambda^N:\mathbf R_+\to \mathbf R_+, \,  u\ge 0, $$  
  one deduces that
   there exists $C>0$, for all $N\ge 1$ and $s\in [0,t]$, $|\langle 1+|\cdot |^{ \gamma}, m_s^N\rangle |\le C$. Together with   \eqref{eq.phi-b}, one has using 
  the dominated
convergence theorem,   $\int_0^t\int_{\mathcal{X}\times\mathcal{Y}}\phi_s^{x,y}(m^N) \pi(\di x,\di y)\di s\to \int_0^t\int_{\mathcal{X}\times\mathcal{Y}}\phi_s^{x,y}(m) \pi(\di x,\di y)\di s$. This proves the desired result. 
\end{proof}

We are now in position to prove that any limit point of the sequence
$(\mu^N)_{N\ge 1}$  in the space $\mathcal D(\mathbf R_+,\mathcal P_{\gamma}(\mathbf R^{d}))$  satisfies~\eqref{eq
  limite}.

\begin{prop}
 \label{lem:convergence to lim eq}Let $\beta>1/2$ and
  assume~\ref{as:batch}-\ref{as:noise}.    Let
  $\mu^*$ be a limit point
  of $(\mu^N)_{N\ge 1}$ in
  $\mathcal D(\mathbf R_+,\mathcal P_{\gamma}(\mathbf R^{d}))$. Then, a.s., $ \mu^*$ satisfies~\eqref{eq limite}.
\end{prop}


\begin{proof} 
  Up to extracting a subsequence, we assume that in distribution  $\mu^N\to \mu^*$  in $\mathcal D(\mathbf R_+,\mathcal P_{\gamma}(\mathbf R^{d}))$. 
    Let $t\in\mathbf{R}_+$ and $f\in\mathcal C_c^\infty(\mathbf R^d)$.  By~\eqref{eq:
    prelim mu} and Lemma~\ref{lem:remainder terms},  we have:
\begin{align*}
 &\textbf{E}\left[\boldsymbol{\Lambda}_t[f](\mu^{N})\right]\\
&=\textbf{E}\Big[\, \Big |\langle f,\mu_0^{N}\rangle-\langle f,\mu_0\rangle+\langle f,M_t^{N}\rangle+\langle f,V_t^N\rangle+\sum_{k=0}^{\lfloor Nt\rfloor-1}\langle f,R_k^{N}\rangle+\frac{1}{N^{1+\beta}}\sum_{k=0}^{\lfloor Nt\rfloor-1}\sum_{i=1}^N\nabla f(W_k^i)\cdot\ve_k^i\Big|\,  \Big ]\\
&\leq \textbf{E}[\, |\langle f,\mu_0^N\rangle-\langle f,\mu_0\rangle|\, ]+ \textbf{E}[|\langle f,V_t^N\rangle|]+\sqrt{\textbf{E}[\langle f,M_t^N\rangle^2]}+\textbf{E}\Big[\Big|\sum_{k=0}^{\lfloor Nt\rfloor -1} \langle f,R_k^{N}\rangle\Big|\Big]\\
&\quad+\sqrt{\textbf{E}\Big[\Big|\frac{1}{N^{1+\beta}}\sum_{k=0}^{\lfloor Nt\rfloor -1}\sum_{i=1}^N \nabla f(W_k^i)\cdot\ve_k^i\Big|^2\Big]} \leq C\|f\|_{\mathcal C^{2,\gamma_*}}\Big[\frac{1}{\sqrt{N}}+  \frac{1}{N}+\frac{1}{N^{2\beta-1}} +\frac{1}{N^{\beta}} \Big],
\end{align*}  
where the bound
$\textbf{E}[|\langle f,\mu_0^N\rangle-\langle f,\mu_0\rangle|]\leq
C\|f\|_{\mathcal C^{2,\gamma_*}}/\sqrt{N}$  follows
from~\eqref{eq.mu00} and the fact that
the initial coefficients are i.i.d. (see~\ref{as:moments
  initiaux}). Therefore, since $\beta>1/2$, for all $t\in\mathbf{R}_+$ and
$f\in\mathcal C_c^\infty(\mathbf R^d)$,
 \begin{equation}\label{eq44} \lim_{N\to
    +\infty}\textbf{E}\left[\boldsymbol{\Lambda}_t[f](\mu^{N})\right]=0.
\end{equation}
By~\cite[Lemma 7.7 in Chapter 3]{ethier2009markov}, the complementary of the set 
$$\mathcal C(\mu^*)=\{t\ge 0, \, \mathbf P(\mu^*_{t^-}=\mu^*_t)=1\}$$
is at most countable. Let $t_*\in \mathcal C(\mu^*)$. 
Denoting by $ \mathsf D({\boldsymbol{\Lambda}_{t_*}[f]})$ the set of
discontinuity points of $\boldsymbol{\Lambda}_{t_*}[f]$, we recall that  from
Lemma~\ref{le.Fc}, $m\notin \mathsf D({\boldsymbol{\Lambda}_{t_*}[f]})$ if
$ m$ is continuous at $t_*$.
Then,   we have:
\begin{align*}\label{eq.CF-P}
\textbf{P}\big(\mu^*\in \mathsf D({\boldsymbol{\Lambda}_{t_*}[f]})\big)=0.
\end{align*}
By~\cite[Theorem
2.7]{billingsley2013convergence}, it then holds:
\begin{equation}
\lim_{N'\to +\infty}\boldsymbol{\Lambda}_{t_*}[f](\mu^{N})=\boldsymbol{\Lambda}_{t_*}[f](\mu^*) \text{  in distribution}, \ \forall t_*\in \mathcal C(\mu^*).\label{eq45}
\end{equation}
By uniqueness of the limit in distribution,~\eqref{eq44}
and~\eqref{eq45} imply that for all $t_*\in \mathcal C(\mu^*)$ and
$f\in\mathcal C_c^\infty(\mathbf R^d)$, a.s.  $\boldsymbol{\Lambda}_{t_*}[f](\mu^*)=0$. 
It then remains to show that a.s. for all $t\in\mathbf{R}_+$ and $f \in\mathcal C_c^\infty(\mathbf R^d)$, 
 $\boldsymbol{\Lambda}_t[f](\mu^*)=0$.  To do so we use a standard continuity argument. 
 
 First of all, for all $m\in \mathcal D(\mathbf R_+,\mathcal P_{\gamma}(\mathbf R^{d}))$ and $f \in\mathcal C_c^\infty(\mathbf R^d)$, the function
  $t\in\mathbf{R}_+\mapsto \boldsymbol{\Lambda}_t[f](m)$ is
  right-continuous. Moreover, there exists a countable subset $\mathcal R_{\mu^*}$ of 
$\mathcal C(\mu^*)$ such that for all $t\ge 0$ and $\epsilon>0$, there exists $s\in \mathcal R_{\mu^*}$,  $s\in [t,t+\ve]$. Thus,  for all $f \in\mathcal C_c^\infty(\mathbf R^d)$, it holds a.s.  for all $t\in\mathbf{R}_+$
 $\boldsymbol{\Lambda}_t[f](\mu^*)=0$. 
 
 Secondly, for all $m\in \mathcal D(\mathbf R_+,\mathcal P_{\gamma}(\mathbf R^{d}))$ and $t\ge 0$, using the dominated convergence theorem, the function $f\in \mathcal H^{L,\gamma}(\mathbf R^{d})\mapsto \boldsymbol \Lambda_{t}[f]( m)$ 
 is continuous (because $\mathcal H^{L,\gamma}(\mathbf R^{d})\hookrightarrow \mathcal C^{1,\gamma}(\mathbf R^{d})$, by \eqref{eq.SE1}). Furthermore, $\mathcal H^{L,\gamma}(\mathbf R^{d})$    admits a dense and countable subset of elements in $\mathcal C_c^\infty(\mathbf R^{d})$. Thus,  a.s. for all $f \in\mathcal H^{L,\gamma}(\mathbf R^{d})$ and  all $t\in\mathbf{R}_+$
 $\boldsymbol{\Lambda}_t[f](\mu^*)=0$. 
 
 Note also that $\mathcal C^\infty_b(\mathbf R^{d})\subset \mathcal H^{L,\gamma}(\mathbf R^{d})$ since $2\gamma>d$. This ends the proof of the proposition. 
\end{proof}

  Note that we have not used Proposition \ref{p-limit in P} in the proof of Proposition \ref{lem:convergence to lim eq}. 

\subsection{Uniqueness \label{sec: uniqueness} of the limit equation
  in $\mathcal C(\mathbf R_+,\mathcal P_1(\mathbf R^d))$ and proof of
  Theorem~\ref{thm:lln}}

\subsubsection{Uniqueness \label{sec: uniqueness} of the limit
  equation in $\mathcal C(\mathbf R_+,\mathcal P_1(\mathbf R^d))$}

\begin{prop}
  \label{prop: existence uniqueness eq transport} There exists a
  unique solution to~\eqref{eq limite}  in the space
  $\mathcal C(\mathbf{R}_+,\mathcal{P}_1(\mathbf{R}^d))$.
\end{prop}

\begin{proof}
  We have already proved the existence.  Let us now prove that there is at most one solution to~\eqref{eq limite} in   
  $\mathcal C(\mathbf{R}_+,\mathcal{P}_1(\mathbf{R}^d))$. 
The proof of the uniqueness of~\eqref{eq limite}  relies on
arguments developed in~\cite{piccoli2016properties,piccoli2015control}
and is divided into several steps.  \medskip

\noindent
\textbf{Step 1.}  Preliminary considerations. 
\medskip

\noindent
If $\mu^\star$ is solution to~\eqref{eq limite}, then for all
$f\in \mathcal C^\infty_b(\mathbf{R}^d)$,
$s\ge0\mapsto \int_{\mathcal{X}\times\mathcal{Y}} \alpha(y-\langle
\sigma_*(\cdot, x),\mu^\star_s\rangle)\langle\nabla
f\cdot\nabla\sigma_*(\cdot,x),\mu^\star_s\rangle\pi(\di x,\di y) $ is
continuous (by the dominated convergence theorem). This implies that
for all $f\in \mathcal C^\infty_b(\mathbf{R}^d)$ and
$t\in\mathbf{R}_+$,
\begin{equation*}\label{unicite: eq lim2}
\frac{\di}{\di t} \langle f,\mu^\star_t\rangle= \int_{\mathbf R^d} \nabla f(w)\cdot \mathbf V[ \mu^\star_t] (w) \mu^\star_t(\di w),
\end{equation*}
where $\mathbf V:\mu\in\mathcal{P}(\mathbf{R}^d)\mapsto \mathbf \mathbf \mathbf V[\mu]$
is defined by 
\begin{equation*} 
\mathbf V[\mu]:w\in\mathbf{R}^d\mapsto\int_{\mathcal{X}\times\mathcal{Y}}\alpha(y-\langle\sigma_*(\cdot,x),\mu\rangle)\nabla\sigma_*(w,x)\pi(\di x,\di y)\in \mathbf{R}^d.
\end{equation*}  
Adopting the terminology of \cite[Section
4.1.2]{santambrogio2015optimal},
$\mu^\star$ is thus a \textit{weak solution}\footnote{We mention that
  according to \cite[Proposition 4.2]{santambrogio2015optimal}, the
  two notions of solutions of~\eqref{eq.measure} (namely the weak
  solution and the \textit{distributional} solution) are equivalent.}
of the measure-valued equation
\begin{align}\label{eq.measure}
\begin{cases} \partial_t \mu^\star_t=\text{div}\left(\mathbf V[ \mu^\star_t] \mu^\star_t\right)\\
 \mu^\star_0=\mu_0.
 \end{cases}
\end{align}
Therefore, to prove the uniqueness result in Proposition~\ref{prop:
  existence uniqueness eq transport}, it is enough to show
that~\eqref{eq.measure} has a unique weak solution in $\mathcal
C(\mathbf{R}_+,\mathcal{P}_1(\mathbf{R}^d))$.  To this end, we
consider two solutions $\mu^1=\lbrace t\mapsto\mu^1_t, t\ge 0
\rbrace$ and $\mu^2=\lbrace t\mapsto\mu^2_t, t\ge 0
\rbrace$ of~\eqref{eq.measure} in $\mathcal C(\mathbf
R_+,\mathcal{P}_1(\mathbf{R}^d))$,
%
%
%
%
and we introduce the following mappings
$$\mathbf v^1: (t,x)\in \mathbf R_+\times\mathbf{R}^d \mapsto \mathbf V[\mu^1_t](x)\quad\text{and}\quad \mathbf v^2: (t,x) \in \mathbf R_+\times\mathbf{R}^d \mapsto \mathbf V[\mu^2_t](x).$$

\noindent
\textbf{Step 2.}  In this step, we prove some basic regularity
properties of $\mathbf V$, $\mathbf v^1$, and $\mathbf v^2$.  \medskip

\noindent
Let us first prove that the velocity fields $\mathbf v^1$ and
$\mathbf v^2$ are globally Lipschitz continuous over
$\mathbf R_+\times\mathbf{R}^d $.  Let $\mu \in \{\mu^1, \mu^2\}$ and
set $v(t,x)=\mathbf V[\mu_t](x)$.  For $0\leq s\le t$ and
$ w_1,w_2\in\mathbf{R}^d$, we have
$$|v(t,w_1)-v(s,w_2)|\leq |v(t,w_1)-v(t,w_2)|+|v(t,w_2)-v(s,w_2)|.$$
By~\ref{as:sigma}, the function $w\mapsto \mathbf V[\mu](w)$ is smooth and
$|\nabla \mathbf V[\mu](w)|\le C$ for some $C>0$ independent of $\mu$
and $w$. Thus, it holds
\begin{align*} 
|v(t,w_1)-v(t,w_2)|=|\mathbf V[\mu_t](w_1)-\mathbf V[\mu_t](w_2)|\leq C|w_1-w_2|,
\end{align*}
for some $C>0$ independent of $t$, $w_1$, and $w_2$.  Secondly, for
any $x\in\mathcal{X}$, considering~\eqref{eq limite}  with
$f=\sigma_*(\cdot,x)$, we obtain
\begin{align*}
|\langle\sigma_*(\cdot,x),\mu_s-\mu_t\rangle|&\leq \int_s^t\int_{\mathcal{X}\times\mathcal{Y}}\left|\alpha(y-\langle\sigma_*(\cdot,x'),\mu_r\rangle)\langle\nabla\sigma_*(\cdot,x)\cdot\nabla\sigma_*(\cdot,x'),\mu_r\rangle\right|\pi(\di x',\di y)\di r\\
&\leq C|t-s|,
\end{align*}
leading to 
\begin{align*}
|v(t,w_2)-v(s,w_2)|=\left|\int_{\mathcal{X}\times\mathcal{Y}}\alpha\langle\sigma_*(\cdot,x),\mu_s-\mu_t\rangle\nabla\sigma_*(w_2,x)\pi(\di x,\di y)\right|
\leq C|t-s|.
\end{align*}
Thus, there exists $C>0$ such that for $0\le s\le t$ and
$w_1,w_2\in\mathbf{R}^d,$
$|v(t,w_1)-v(s,w_2)|\leq C(|t-s|+|w_1-w_2|)$, which proves that $v$ is
globally Lipschitz. Now we claim that there exists $L'>0$ such that
for every $\mu,\nu\in\mathcal{P}_1(\mathbf{R}^d)$,
\begin{align}\label{eq.VLip}
 \|\mathbf V[\mu]-\mathbf V[\nu]\|_\infty:=\sup_{w\in \mathbf R^d} |\mathbf V[\mu](w)-\mathbf V[\nu](w)| \leq L'\mathsf W_1(\mu,\nu).
\end{align} 
By~\ref{as:sigma}, there exists $C>0$ such that for all
$\mu,\nu\in\mathcal{P}_1(\mathbf{R}^d)$ and all $w\in\mathbf{R}^d$,
\begin{align*}
|\mathbf V[\mu](w)-\mathbf V[\nu](w)|&=\left|\int_{\mathcal{X}\times\mathcal{Y}}\alpha(\langle\sigma_*(\cdot,x),\nu\rangle-\langle\sigma_*(\cdot,x),\mu\rangle)\nabla_W\sigma_*(w,x)\pi(\di x,\di y)\right|\\
&\leq C\int_{\mathcal{X}\times\mathcal{Y}}\left|\langle  \sigma_*(\cdot,x),\nu\rangle-\langle\sigma_*(\cdot,x),\mu\rangle\right|\pi(\di x,\di y) \le C\mathsf W_1(\mu,\nu),
\end{align*}  
where the last inequality is obtained by the Lipschitz continuity of
$\sigma_*(\cdot,x)$ (which is uniform in $x\in \mathcal X$).  \medskip

\noindent
\textbf{Step 3.}  End of the proof of Proposition~\ref{prop: existence
  uniqueness eq transport}.  \medskip

\noindent
Since $v$ is globally Lipschitz, we can introduce the flows
$(\phi^1_t)_{t\in[0,T]}$ and $(\phi^2_t)_{t\in[0,T]}$ with respect to
$\mu^1$ and $\mu^2$. By \cite[Theorem 5.34]{villani2021topics}, one
has
\begin{equation}\label{unicite: forme des solutions}
\mu^1_t=\phi^1_t\#\mu_0,\quad \mu^2_t=\phi^2_t\#\mu_0,\quad \forall t\ge 0.
\end{equation} 
The symbol $\#$ stands for the pushforward of a measure.  Let $L>0$ be
a constant such that $|v_t^i(w_1)-v_t^i(w_2)|\leq L|w_1-w_2|$ for all
$i=1,2$, $t\in \mathbf R_+$ and $w_1,w_2\in\mathbf{R}^d$ (which exists
by the previous step). Then by \cite[Proposition~4]{piccoli2016properties}, it holds for all
$\mu,\nu\in\mathcal{P}_1(\mathbf{R}^d)$, \begin{align}\label{u2}
  \mathsf W_1(\phi^1_t\#\mu,\phi^2_t\#\nu)\leq e^{Lt}\mathsf
  W_1(\mu,\nu)+\frac{e^{Lt}-1}{L}\ \sup_{0\leq s\leq t} \
  \|v_s^1-v_s^2\|_{\infty}.
\end{align}
We are now in position to prove that $\mu^1=\mu^2$. We use the
techniques introduced in \cite{piccoli2015control}. Let us now
consider $T>0$, and introduce
$$t_0:=\inf\lbrace t\in[0,T], \, \mathsf W_1(\mu^1_t,\mu^2_t)\neq
0\rbrace.$$ We shall prove that $t_0=T$. Assume that $t_0<T$.
By~\eqref{unicite: forme des solutions} and~\eqref{u2}, we have, for
$0\leq s\leq T-t_0$,
$$\mathsf W_1(\mu^1_{t_0+s},\mu^2_{t_0+s})\leq e^{Ls}\mathsf W_1(\mu^1_{t_0},\mu^2_{t_0})+\frac{e^{Ls}-1}{L}\ \sup_{t_0\leq\tau\leq t_0+s}\ \|v_\tau^1-v_\tau^{ 2}\|_\infty$$
By continuity, $\mathsf W_1(\mu^1_{t_0},\mu^2_{t_0})=0$.  For $s$
small enough such that $e^{Ls}-1<2Ls$, we obtain,
using~\eqref{eq.VLip},
\begin{align*}
\mathsf W_1(\mu^1_{t_0+s},\mu^2_{t_0+s})\leq 2sL' \sup_{t_0\leq\tau\leq t_0+s}\ \mathsf W_1(\mu^1_\tau,\mu^2_\tau).
\end{align*}
Then, for $0\leq s'\leq s<\min(1/2L',T-t_0)$, applying the last
inequality for $s'$ gives
\begin{align*}
\mathsf W_1(\mu^1_{t_0+s'},\mu^2_{t_0+s'})<\sup_{t_0\leq\tau\leq t_0+s}\ \mathsf W_1(\mu^1_\tau,\mu^2_\tau), 
\end{align*} 
which is not possible. Hence, $t_0=T$, and again, by continuity, we
conclude that $ \mathsf W_1(\mu^1_t,\mu_t^2)=0$, $ \forall t\in[0,T]$.
Therefore, $\mu^1=\mu^2$. We have thus proved that~\eqref{eq limite}  admits a unique solution in
$\mathcal C(\mathbf R_+,\mathcal{P}_1(\mathbf{R}^d))$, which is the
desired result.
\end{proof}


\subsubsection{End of the proof of Theorem~\ref{thm:lln}\label{sec:
    proof thm 1}}

We can now prove  Theorem~\ref{thm:lln}. 
\begin{proof}[Proof of Theorem~\ref{thm:lln}]
  By Proposition~\ref{lem:rel comp Sobolev}, the sequence
  $(\mu^N)_{N\ge 1}$ is
  relatively compact in $\mathcal D(\rp,\mathcal P_{\gamma}(\mathbf{R}^d))$. Let $\bar \mu^1$ and $\bar \mu^2$ be two limit
  points of $(\mu^N)_{N\ge 1}$ in
  $\mathcal D(\rp,\mathcal P_{\gamma}(\mathbf{R}^d))$.  Let
  $j\in \{1,2\}$.  By Lemma~\ref{p-limit in P}, a.s.
  $\bar \mu^j\in \mathcal C(\rp,\mathcal P_{1}(\mathbf{R}^d))$. According to Proposition~\ref{lem:convergence
    to lim eq}, $\bar \mu^j$ satisfies a.s.~\eqref{eq limite}.   
Let
$ \mu^\star\in \mathcal C(\mathbf{R}_+,\mathcal{P}_1(\mathbf{R}^d))$
be the unique solution of~\eqref{eq limite} (see
Proposition~\ref{prop: existence uniqueness eq transport}).  Therefore, one has that   a.s.  $\bar \mu^j= \mu^\star$ in
$\mathcal C(\mathbf{R}_+,\mathcal{P}_1(\mathbf{R}^d))$.  
Note that this implies also that 
$ \mu^\star\in \mathcal D(\mathbf{R}_+,\mathcal{P}_{\gamma}(\mathbf{R}^d))$ and then that   a.s. $\bar \mu^j= \mu^\star$   in
$\mathcal D(\mathbf{R}_+,\mathcal{P}_{\gamma}(\mathbf{R}^d))$.
Therefore,  $(\mu^N)_{N\geq1}$   converges in distribution  to $ {\mu^\star}$ in
$\mathcal D(\mathbf{R}_+,\mathcal P_{\gamma}(\mathbf{R}^d))$ (and
then the convergence holds in probability).  This ends the proof of
Theorem~\ref{thm:lln}.
\end{proof}


\section{Proof of Theorem~\ref{thm:clt}}
\label{sec.TH2}
In this section, we prove Theorem~\ref{thm:clt}. Recall that
$\bar \mu \in \mathcal C(\mathbf R_+, \mathcal H^{-L,\gamma}(\mathbf
R^d))$ is given by Corollary~\ref{co.e} and that the fluctuation
process is defined by (see~\eqref{eq.etaN}):
$$
\eta^N=\sqrt{N}(\mu^N -\bar \mu), \ N\ge 1. 
$$
Throughout this section, we assume that~\ref{as:batch}-\ref{as:batch limite} hold. 

\subsection{Relative compactness of $(\eta^N)_{N\ge 1}$ in
$\mathcal D(\mathbf R_+,\mathcal H^{-J_0+1,j_0}(\mathbf R^d))$}

To prove the relative compactness of $(\eta^N)_{N\ge 1}$ in
$\mathcal D(\mathbf R_+,\mathcal H^{-J_0+1,j_0}(\mathbf R^d))$, we
will use Proposition~\ref{lem:note kurtz} with
$\mathcal H_1=\mathcal H^{J_0-1,j_0}(\mathbf R^d)$ and
$\mathcal H_2=\mathcal H^{J_2,j_2}(\mathbf R^d)$.  Mimicking the proof
of \cite[Theorem 1.1]{sznitman_topics_1991}, there exists a unique,
trajectorial and in law, solution of
\begin{equation}\begin{cases}\label{particle system_diff}
\di X_t=\alpha\int_{\mathcal{X}\times\mathcal{Y}}(y-\langle\sigma_*(\cdot,x),\hat \mu_t\rangle)\nabla_W\sigma_*(X_t,x)\pi(\di x,\di y)\di t, \\
X_0\sim\mu_0,\ \ \  \hat\mu_t=\text{Law}(X_t).
\end{cases}\end{equation} 
Denote by
$\hat\mu\in\mathcal P(\mathcal C(\mathbf R_+,\mathbf R^d))$ this
solution.
The mapping $t\ge 0\mapsto\hat\mu_t$ satisfies Equation~\eqref{eq
  limite}. In addition, using \ref{as:sigma}, it is straightforward
to show that the function $t\mapsto\widetilde\mu_t$ lies in
$\mathcal C(\mathbf R_+,\mathcal P_1(\mathbf R^d))$. Since $\bar\mu$
is the unique solution of~\eqref{eq limite} (see
Proposition~\ref{prop: existence uniqueness eq transport}), $\hat \mu=\bar\mu$.  Therefore, we introduce, as it is
customary, the particle system defined as follows: for $N\ge 1$, let
$\bar{X}^i=\lbrace t\mapsto\bar{X}^i_t, t\in\mathbf{R}_+\rbrace$
($i\in \{1,\ldots, N\}$) be the $N$ independent processes satisfying:
\begin{equation}\begin{cases}\label{particle system}
\bar{X}^i_t=W_0^i+\int_0^t\alpha\int_{\mathcal{X}\times\mathcal{Y}}(y-\langle\sigma_*(\cdot,x),\bar \mu_s\rangle)\nabla_W\sigma_*(\bar{X}_s^i,x)\pi(\di x,\di y)\di s, \ t\in \mathbf{R}_+ \\
\bar \mu_t=\text{Law}(\bar{X}_t^i).
\end{cases}\end{equation} 
We then introduce its empirical measure: 
\begin{equation}\label{eq.mubarparticule}
\bar \mu_t^N=\frac{1}{N}\sum_{i=1}^{N}\delta_{\bar{X}_t^i}, \ N\ge 1, \ t\in \mathbf R_+.
\end{equation}
By~\ref{as:sigma}, there exists $C_0>0$ such that a.s. for all
$0\le s\le t$, it holds for all $i\in \{1,\ldots,N\}$:
 \begin{equation}\label{eq.particule-Lip}
 |\bar{X}_t^i-\bar{X}_s^i|\le C_0(t-s).
 \end{equation}
 In particular, $t\in \mathbf R_+ \mapsto \bar{X}_t^i\in \mathbf R^d$
 is a.s. continuous for all $i\in \{1,\ldots,N\}$. Because $\mu_0$ is
 compactly supported (see indeed~\ref{as: mu0compact}),  one deduces that 
 there exists $C>0$ such that a.s. for all $T>0$ and  for all $i\in \{1,\ldots,N\}$:
\begin{equation}\label{tilde w borne}
\sup_{t\in[0,T]}|\bar{X}_t^i|\leq C(1+T).
\end{equation}
Thus, for any $\beta\ge 0$, a.s.
$ \bar \mu^N\in \mathcal C(\mathbf{R}_+,\mathcal
C^{1,\beta}(\mathbf{R}^d)')$. We now define, for $N\geq1$,
\begin{equation}\label{eq.dec-eta}
\Upsilon^N:=\sqrt{N}(\mu^N-\bar \mu^N) \ \text{ and } \ \Theta^N:=\sqrt{N}(\bar \mu^N-\bar \mu).
\end{equation}
It then holds:
\begin{equation}\label{eq.dec-eta2}
\eta^N=  \Upsilon^N +  \Theta^N.
\end{equation}
 For all
$N\ge 1$ and for any $\beta\ge 0$, since a.s.
$\mu^N\in \mathcal D(\mathbf{R}_+,\mathcal
C^{0,\beta}(\mathbf{R}^d)')$, one has
  $$\text{a.s. }  \Upsilon^N\in \mathcal D(\mathbf{R}_+,\mathcal C^{1,\beta}(\mathbf{R}^d)').$$
  In particular a.s.
  $\Upsilon^N\in \mathcal D(\mathbf{R}_+,\mathcal
  H^{1-J_1,j_1}(\mathbf{R}^d))$ because
  $\mathcal H^{J_1-1,j_1}(\mathbf R^d)\hookrightarrow \mathcal
  C^{1,j_1}(\mathbf R^d)$ (according to~\eqref{eq.SEC} and since $J_1-1>d/2+1$).  On the other hand,  since
  $\bar \mu \in \mathcal C(\mathbf R_+, \mathcal H^{-L,\gamma}(\mathbf
  R^d))$ (see Corollary \ref{co.e}) and since a.s.
  $\bar \mu^N \in \mathcal C(\mathbf R_+, \mathcal
  H^{-L,\gamma}(\mathbf R^d))$, it holds for all
  $N\ge 1$:
 \begin{equation}
 \label{eq.t-c}
 \text{a.s. } \Theta^N\in  \mathcal C(\mathbf{R}_+,\mathcal H^{-L,\gamma}(\mathbf R^d)).
  \end{equation}
We start with  the following lemma.

\begin{lem}\label{lem : unif bound xi and eta}
  Let $\beta\ge 3/4$ and assume~\ref{as:batch}-\ref{as:batch
    limite}. Then, for all $T>0$, we have
$$ 
\sup_{N\geq1} \sup_{t\in[0,T]} \mathbf{E}\big [\|\Theta_t^N\|_{\mathcal H^{-J_1,j_1}}^2 + \|\Upsilon_t^N\|_{\mathcal H^{-J_1,j_1}}^2\big ]<+\infty.
$$
In particular, $\sup_{N\geq1} \sup_{t\in[0,T]} \mathbf{E}[\|\eta_t^N\|_{\mathcal H^{-J_1,j_1}}^2]<+\infty$.
\end{lem}

\begin{proof}
  Let $T>0$.  In all this proof,
  $f\in \mathcal H^{J_1,j_1}(\mathbf{R}^d)$ and
  $\lbrace f_a\rbrace_{a\geq 1}$ is an orthonormal basis of
  $\mathcal H^{J_1,j_1}(\mathbf{R}^d)$.  In the following, $C$ denotes
  a constant independent of $N\geq1$, $t\in[0,T]$,
  $(x,y)\in \mathcal X\times \mathcal Y$, and the test function $f$,
  which can change from one occurence to another.  The proof is
  divided into two steps.  \medskip

\noindent
\textbf{Step 1.} Upper bound on
$\textbf{E}\left[\|\Theta_t^N\|^2_{\mathcal H^{-J_1,j_1}}\right]$.
\medskip

\noindent
Since $\bar{X}_s^1,\dots,\bar{X}_s^N$ are i.i.d. with law $\bar \mu_s$
(see~\eqref{particle system}),   using~\eqref{tilde w borne},  the fact that  $\bar \mu \in \mathcal C(\mathbf R_+, \mathcal H^{-L,\gamma}(\mathbf
  R^d))$ (see Corollary~\ref{co.e}), and $\mathcal H^{L,\gamma}(\mathbf{R}^d)\hookrightarrow \mathcal C^{0,\gamma}(\mathbf{R}^d)$, it holds for all $t\in [0,T]$: 
\begin{align*}
\textbf{E}\left[\langle f,\Theta_t^N\rangle^2\right]= \textbf{E}\Big[\Big\langle f,\sqrt{N}(\bar \mu_t^N-\bar \mu_t)\Big\rangle^2\Big]&=\textbf{E}\Big[\Big|\frac{1}{\sqrt{N}}\sum_{i=1}^N\big [f(\bar{X}_t^i)-\langle f,\bar \mu_t\rangle\big ]\Big|^2\Big]\nonumber\\
&=\frac{1}{N}\sum_{i=1}^N\textbf{E}\Big[\Big|f(\bar{X}_t^i)-\langle f,\bar \mu_t\rangle\Big|^2\Big]\nonumber\\
&\leq\frac{2}{N}\sum_{i=1}^N\textbf{E}\big[|f(\bar{X}_t^i)|^2\big]+|\langle f,\bar \mu_t\rangle|^2 
 \leq C\|f\|_{\mathcal H^{L,\gamma}}^2. 
\end{align*}
Taking $f=f_a$ in the previous inequality, summing over
$a\in \mathbf N^*$, and using the fact that
$\mathcal H^{J_1,j_1}(\mathbf{R}^d)\hookrightarrow_{\text{H.S.}}
\mathcal H^{L,\gamma}(\mathbf{R}^d)$, one deduces that:
\begin{align}\label{borne Z}
\sup_{N\geq1} \sup_{t\in[0,T]}\textbf{E}\left[\|\Theta_t^N\|^2_{\mathcal H^{-J_1,j_1}}\right]\leq C.
\end{align}
  
\noindent
\textbf{Step 2.} Upper bound on
$\textbf{E}\left[\|\Upsilon_t^N\|^2_{\mathcal H^{-J_1,j_1}}\right]$.
\medskip

\noindent
Recall that a.s.
$\Upsilon^N\in \mathcal D(\mathbf{R}_+,\mathcal
H^{1-J_1,j_1}(\mathbf{R}^d))$. Thus, a.s. $\Upsilon^N\in\mathcal D(\mathbf{R}_+,\mathcal
H^{-J_1,j_1}(\mathbf{R}^d))$. We then have for all $t\in \mathbf R_+$,
\begin{align}\label{eq.formulefa}
\|\Upsilon_t^N\|_{\mathcal H^{-J_1,j_1}}^2=\sum_{a= 1}^{+\infty}\langle f_a,\Upsilon_t^N\rangle^2.
\end{align}
 For all $i\in \{1,\ldots,N\}$, a.s. $t\in \mathbf R_+\mapsto f(\bar X^i_t) \in \mathcal C^1(\mathbf R_+)$. Indeed, $f\in\mathcal C^1(\mathbf R^d)$ and a.s. $t\in  \mathbf R_+\mapsto  \bar X^i_t$ (see~\eqref{particle system}) is $\mathcal C^1$ (because  a.s. $s\in \mathbf R_+\mapsto  \alpha\int_{\mathcal{X}\times\mathcal{Y}}(y-\langle\sigma_*(\cdot,x),\bar \mu_s\rangle)\nabla_W\sigma_*(\bar{X}_s^i,x)\pi(\di x,\di y)$ is continuous by the dominated convergence theorem).
Therefore, it holds $\langle f,\bar \mu_t^N\rangle=\langle f,\bar \mu_0^N\rangle+\int_0^t\frac{\di }{\di s}\langle f,\bar \mu_s^N\rangle\di s$ and therefore, 
\begin{align}\label{formule f mu tilde}
\langle f,\bar \mu_t^N\rangle =\langle f,\bar \mu_0^N\rangle+\int_0^t\int_{\mathcal{X}\times\mathcal{Y}}\alpha(y-\langle\sigma_*(\cdot,x),\bar \mu_s\rangle)\langle\nabla f\cdot\nabla\sigma_*(\cdot,x),\bar \mu_s^N\rangle\pi(\di x,\di y)\di s.
\end{align}
Thus, by definition of $\Upsilon_t^N$ (see~\eqref{eq.dec-eta}) and
using also~\eqref{eq: prelim mu}, we have:
\begin{align}
\langle f,\Upsilon_t^N\rangle  &=  \sqrt N \underbrace{ \langle f,(\mu_0^N-\bar\mu_0^N) \rangle}_{=0}+\sqrt{N}\int_0^t\int_{\mathcal{X}\times\mathcal{Y}}\alpha(y-\langle\sigma_*(\cdot,x),\mu_s^N\rangle)\langle\nabla f\cdot\nabla\sigma_*(\cdot,x),\mu_s^N\rangle\pi(\di x,\di y)\di s\nonumber\\
&\quad -\sqrt{N}\int_0^t\int_{\mathcal{X}\times\mathcal{Y}}\alpha(y-\langle\sigma_*(\cdot,x),\bar \mu_s\rangle)\langle\nabla f\cdot\nabla\sigma_*(\cdot,x),\bar \mu_s^N\rangle\pi(\di x,\di y)\di s\nonumber\\
&\quad +\sqrt{N}\langle f,M_t^N\rangle+\sqrt{N}\langle f,V_t^N\rangle+\sqrt{N} \sum_{k=0}^{\lfloor Nt\rfloor-1}\langle f,R_k^{N}\rangle+\frac{\sqrt{N}}{N^{1+\beta}}\sum_{k=0}^{\lfloor Nt\rfloor -1}\sum_{i=1}^N\nabla f(W_k^i)\cdot\ve_k^i\label{f contre xi0 }.
\end{align}
Furthermore, it holds,
\begin{align}\label{trick decompo}
&\sqrt{N}(y-\langle\sigma_*(\cdot,x),\mu_s^N\rangle)\langle\nabla f\cdot\nabla\sigma_*(\cdot,x),\mu_s^N\rangle\nonumber\\
&=(y-\langle\sigma_*(\cdot,x),\mu_s^N\rangle)\langle\nabla f\cdot\nabla\sigma_*(\cdot,x),\Upsilon_s^N\rangle\nonumber-\langle\sigma_*(\cdot,x),\Upsilon_s^N\rangle\langle\nabla f\cdot\nabla\sigma_*(\cdot,x),\bar \mu_s^N\rangle\nonumber\\
&\quad+\sqrt{N}(y-\langle\sigma_*(\cdot,x),\bar \mu_s\rangle-\langle\sigma_*(\cdot,x),\bar \mu_s^N-\bar \mu_s\rangle)\langle\nabla f\cdot\nabla\sigma_*(\cdot,x),\bar \mu_s^N\rangle.
\end{align}
Consequently, plugging~\eqref{trick decompo} in~\eqref{f contre xi0 },
we obtain for $t\in \mathbf R_+$:
\begin{align}
\langle f,\Upsilon_t^N\rangle&=\int_0^t\int_{\mathcal{X}\times\mathcal{Y}}\alpha (y-\langle\sigma_*(\cdot,x),\mu_s^N\rangle)\langle\nabla f\cdot\nabla\sigma_*(\cdot,x),\Upsilon_s^N\rangle\pi(\di x,\di y)\di s\nonumber\\
&\quad-\int_0^t\int_{\mathcal{X}\times\mathcal{Y}}\alpha \langle\sigma_*(\cdot,x),\Upsilon_s^N)\rangle\langle\nabla f\cdot\nabla\sigma_*(\cdot,x),\bar \mu_s^N\rangle\pi(\di x,\di y)\di s\nonumber\\
&\quad-\int_0^t\int_{\mathcal{X}\times\mathcal{Y}}\alpha \langle\sigma_*(\cdot,x),\sqrt{N}(\bar \mu_s^N-\bar \mu_s)\rangle\langle\nabla f\cdot\nabla\sigma_*(\cdot,x),\bar \mu_s^N\rangle\pi(\di x,\di y)\di s\nonumber\\&\quad +\sqrt{N}\langle f,M_t^N\rangle+\sqrt{N}\langle f,V_t^N\rangle+\sqrt{N} \sum_{k=0}^{\lfloor Nt\rfloor-1}\langle f,R_k^{N}\rangle+\frac{\sqrt{N}}{N^{1+\beta}}\sum_{k=0}^{\lfloor Nt\rfloor -1}\sum_{i=1}^N\nabla f(W_k^i)\cdot\ve_k^i\label{f contre xi }.
\end{align}
By Lemma~\ref{le.CR} (in Appendix~\ref{app:B}, see also Remark~\ref{re.CR}), one then has for
all $t\in \mathbf R_+$:
\begin{align}\label{eq : ito + decompo}
\langle f,\Upsilon_t^N\rangle^2&\le  \textbf A^N_t[f]+ \textbf B^N_t[f],
\end{align}
where 
\begin{align*}
\textbf A^N_t[f]&=2\int_0^t\int_{\mathcal{X}\times\mathcal{Y}}\alpha\langle f,\Upsilon_s^N\rangle(y-\langle\sigma_*(\cdot,x),\mu_s^N\rangle)\langle\nabla f\cdot\nabla\sigma_*(\cdot,x),\Upsilon_s^N\rangle\pi(\di x,\di y)\di s\nonumber\\
&\quad-2\int_0^t\int_{\mathcal{X}\times\mathcal{Y}}\alpha\langle f,\Upsilon_s^N\rangle\langle\sigma_*(\cdot,x),\Upsilon_s^N)\rangle\langle\nabla f\cdot\nabla\sigma_*(\cdot,x),\bar \mu_s^N\rangle\pi(\di x,\di y)\di s\nonumber\\
&\quad-2\int_0^t\int_{\mathcal{X}\times\mathcal{Y}}\alpha\langle f,\Upsilon_s^N\rangle\langle\sigma_*(\cdot,x),\sqrt{N}(\bar \mu_s^N-\bar \mu_s)\rangle\langle\nabla f\cdot\nabla\sigma_*(\cdot,x),\bar \mu_s^N\rangle\pi(\di x,\di y)\di s\\
&  =:\textbf I_t^N[f]+ \textbf J_t^N[f]+\textbf K_t^N[f],
\end{align*}
and
\begin{align*}
\textbf B^N_t[f]&= \sum_{k=0}^{\lfloor Nt\rfloor-1}\Big[2\langle f,\Upsilon_{\frac{k+1}{N}^-}^N\rangle\sqrt{N}\langle f,M_k^N\rangle+4 N\langle f,M_k^N\rangle^2\Big] +\sum_{k=0}^{\lfloor Nt\rfloor-1}\Big[2\langle f,\Upsilon_{\frac{k+1}{N}^-}^N\rangle\sqrt{N}\langle f,R_k^{N}\rangle+4 N|\langle f,R_k^{N}\rangle|^2\Big]\nonumber\\
&\quad+\sum_{k=0}^{\lfloor Nt\rfloor-1}\Big[\frac{2\sqrt{N}}{N^{1+\beta}}\langle f,\Upsilon_{\frac{k+1}{N}^-}^N\rangle\sum_{i=1}^N\nabla f(W_k^i)\cdot\ve_k^i+\frac{4 }{N^{1+2\beta}}\big|\sum_{i=1}^N\nabla f(W_k^i)\cdot\ve_k^i\big|^2\Big]\nonumber\\
&\quad+\sum_{k=0}^{\lfloor Nt\rfloor-1}\Big[2\langle f,\Upsilon_{\frac{k+1}{N}^-}^N\rangle \mathbf a_k^N[f]+4 |\mathbf a_k^N[f]|^2\Big]-2\sqrt{N}\int_0^t\langle f,\Upsilon_s^N\rangle \mathbf L_s^N[f]\di s,
\end{align*}
with 
\begin{equation}\label{eq.akN}
\mathbf L_s^N[f]=\int_{\mathcal{X}\times\mathcal{Y}}\!\!\!\!\!  \alpha(y-\langle\sigma_*(\cdot,x),\mu_s^N\rangle)\langle\nabla f\cdot\nabla\sigma_*(\cdot,x),\mu_s^N\rangle\pi(\di x,\di y) \,  \text{ and } \,  \mathbf a_k^N[f]=\sqrt{N}\int_{\frac kN}^{\frac{k+1}{N}}\!\! \mathbf L_s^N[f]\di s.
\end{equation}
\noindent
Using~\eqref{eq : ito + decompo} and~\eqref{eq.formulefa},
\begin{equation}\label{eq.Xi-fa}
\|\Upsilon_t^N\|_{\mathcal H^{-J_1,j_1}}^2=\sum_{a=1}^{+\infty} \textbf A^N_t[f_a]+\sum_{a=1}^{+\infty} \textbf B^N_t[f_a].
\end{equation}
By Lemma~\ref{lem : additionnal terms}, detailed in Appendix~\ref{app:B}, and since $\beta \ge 3/4$, one
has for all $f\in \mathcal H^{J_1,j_1}(\mathbf{R}^d)$ and
$t\in \mathbf R_+$,
$$\mathbf{E}\big[\textbf B^N_t[f]\big]\le C\big (\| f\|_{\mathcal H^{L,\gamma}}^2+ \mathbf{E}\left[\int_0^t\langle f,\Upsilon_s^N\rangle^2\di s\right]\big).$$
Thus, recalling $\mathcal H^{J_1,j_1}(\mathbf{R}^d)\hookrightarrow_{\text{H.S.}}\mathcal H^{L,\gamma}(\mathbf{R}^d)$ and~\eqref{eq.formulefa}, we obtain,
\begin{equation}\label{eq.Bfa} \sum_{a=1}^{+\infty}\mathbf E[ \textbf
  B^N_t[f_a]]\le C+ C\int_0^t\textbf{E}\big[ \|\Upsilon_s^N
  \|_{\mathcal H^{-J_1,j_1}}^2\big] \di s.
\end{equation}
Let us now provide a similar upper bound on
$ \sum_{a=1}^{+\infty} \textbf A^N_t[f_a]$.  By~\eqref{tilde w borne}
and because
$\mathcal H^{L,\gamma}(\mathbf{R}^d)\hookrightarrow \mathcal
C^{2,\gamma_*}(\mathbf{R}^d)$ (see~\eqref{eq.SE1}),
\begin{align}
\label{31}
|\langle\nabla f\cdot\nabla\sigma_*(\cdot,x),\bar \mu_s^N\rangle| =\Big|\frac{1}{N}\sum_{i=1}^N\nabla f(\bar{X}_s^i)\cdot\nabla\sigma_*(\bar{X}_s^i,x)\Big|&\leq\frac{C\|f\|_{\mathcal C^{2,\gamma_*}}}{N}\sum_{i=1}^N(1+|\bar{X}_s^i|^{\gamma_*}) \leq C\|f\|_{\mathcal H^{L,\gamma}}.
\end{align}
Then, using~\ref{as:sigma} and since $j_1>d/2$, it holds:
we have using~\eqref{31}:
\begin{align*}
&-2\int_0^t\int_{\mathcal{X}\times\mathcal{Y}}\alpha\langle f,\Upsilon_s^N\rangle\langle\sigma_*(\cdot,x),\Upsilon_s^N)\rangle\langle\nabla f\cdot\nabla\sigma_*(\cdot,x),\bar \mu_s^N\rangle\pi(\di x,\di y)\di s\nonumber\\
&\leq 4\alpha \int_0^t\int_{\mathcal{X}\times\mathcal{Y}}\left(\langle f,\Upsilon_s^N\rangle^2+C\|f\|_{\mathcal H^{L,\gamma}}^2\|\sigma_*(W,x)\|_{\mathcal H^{J_1,j_1}}^2\|\Upsilon_s^N\|_{\mathcal H^{-J_1,j_1}}^2\right)\pi(\di x,\di y)\di s\nonumber\\
&\leq C\int_0^t \left(\langle f,\Upsilon_s^N\rangle^2+C\|f\|_{\mathcal H^{L,\gamma}}^2\|\Upsilon_s^N\|_{\mathcal H^{-J_1,j_1}}^2\right)\di s.
\end{align*}
Thus,
$ \sum_{a=1}^{+\infty} \textbf{E}[\textbf J^N_t[f_a]]\le
C\int_0^t\textbf{E}\big[ \|\Upsilon_s^N \|_{\mathcal
  H^{-J_1,j_1}}^2\big] \di s$.  Let us now study
$\sum_{a=1}^{+\infty} \mathbf K^N_t[f_a]$. Since
$\bar{X}_s^1,\ldots,\bar{X}_s^N$ are i.i.d. with law $\bar \mu_s$
(see~\eqref{particle system}) and because $\sigma_*$ is bounded
(see~\ref{as:sigma}), we have:
\begin{align}\label{eq : tilde(mu)-bar(mu)}
\textbf{E}\left[\left\langle\sigma_*(\cdot,x),\bar \mu_s^N-\bar \mu_s\right\rangle^2\right]&=\textbf{E}\Big[\Big|\frac{1}{N}\sum_{i=1}^N\sigma_*(\bar{X}_s^i,x)-\langle\sigma_*(\cdot,x),\bar \mu_s\rangle\Big|^2\Big]\nonumber\\
&=\frac{1}{N^2}\sum_{i=1}^N\textbf{E}\left[(\sigma_*(\bar{X}_s^i,x)-\langle\sigma_*(\cdot,x),\bar \mu_s\rangle)^2\right]\leq \frac{C}{N}.
\end{align}
 Hence, $\textbf{E}[\langle\sigma_*(\cdot,x),\sqrt{N}(\bar \mu_s^N-\bar \mu_s) \rangle^2]\leq C$.
%
 Thus, using in addition~\eqref{31}, one deduces that
\begin{align*}
&\textbf{E}\left[-2\int_0^t\int_{\mathcal{X}\times\mathcal{Y}}\alpha\langle f,\Upsilon_s^N\rangle\langle\sigma_*(\cdot,x),\sqrt{N}(\bar \mu_s^N-\bar \mu_s)\rangle\langle\nabla f\cdot\nabla\sigma_*(\cdot,x),\bar \mu_s^N\rangle\pi(\di x,\di y)\di s\right]\\
&\leq C\int_0^t\int_{\mathcal{X}\times\mathcal{Y}}\left\lbrace\textbf{E}[\langle f,\Upsilon_s^N\rangle^2]+\textbf{E}\left[\langle\sigma_*(\cdot,x),\sqrt{N}(\bar \mu_s^N-\bar \mu_s)\rangle^2\langle\nabla f\cdot\nabla\sigma_*(\cdot,x),\bar \mu_s^N\rangle^2\right]\right\rbrace\pi(\di x,\di y)\di s\\
&\leq C\int_0^t\int_{\mathcal{X}\times\mathcal{Y}}\left\lbrace\textbf{E}[\langle f,\Upsilon_s^N\rangle^2]+\textbf{E}\left[\langle\sigma_*(\cdot,x),\sqrt{N}(\bar \mu_s^N-\bar \mu_s)\rangle^2\right]\|f\|_{\mathcal H^{L,\gamma}}^2\right\rbrace  \pi(\di x,\di y)\di s\\
&\leq C\int_0^t\left\lbrace\textbf{E}[\langle f,\Upsilon_s^N\rangle^2]+\|f\|_{\mathcal H^{L,\gamma}}^2\right\rbrace\di s.
\end{align*}
Therefore,
$ \sum_{a=1}^{+\infty}\textbf{E}[ \mathbf K^N_t[f_a]]\le
C+C\int_0^t\textbf{E}\big[ \|\Upsilon_s^N \|_{\mathcal
  H^{-J_1,j_1}}^2\big] \di s$.  It remains to study
$\sum_{a=1}^{+\infty} \textbf I^N_t[f_a]$. To this end, for
$x\in \mathcal X$, introduce the bounded linear operator
\begin{equation}\label{eq.Tx}
\mathsf T_x: f\in \mathcal H^{J_1,j_1}(\mathbf{R}^d)\mapsto \nabla f\cdot \nabla\sigma_*(\cdot,x)\in \mathcal H^{J_1-1,j_1}(\mathbf{R}^d).
\end{equation}
Then, one has: 
\begin{align*}
\sum_{a=1}^{+\infty} \textbf I_t^N[f_a]= &\sum_{a=1}^{+\infty} 2\int_0^t\int_{\mathcal{X}\times\mathcal{Y}}\alpha\langle f_a,\Upsilon_s^N\rangle(y-\langle\sigma_*(\cdot,x),\mu_s^N\rangle)\langle\mathsf T_xf_a,\Upsilon_s^N\rangle\pi(\di x,\di y)\di s\\
&=2\int_0^t\int_{\mathcal{X}\times\mathcal{Y}}\alpha(y-\langle\sigma_*(\cdot,x),\mu_s^N\rangle)\sum_{a=1}^{+\infty} \langle f_a,\Upsilon_s^N\rangle\langle\mathsf T_xf_a,\Upsilon_s^N\rangle\pi(\di x,\di y)\di s\\
&=2\int_0^t\int_{\mathcal{X}\times\mathcal{Y}}\alpha(y-\langle\sigma_*(\cdot,x),\mu_s^N\rangle)\left\langle \Upsilon_s^N,\mathsf T_x^*\Upsilon_s^N\right\rangle_{-J_1,j_1}\pi(\di x,\di y)\di s.
\end{align*}
Since $\sigma_*$ and $\mathcal Y$ are bounded, this implies that:
\begin{align*}
  \sum_{a=1}^{+\infty} \textbf{E}[\textbf I^N_t[f_a]]&\le C\int_0^t \textbf{E}\Big[ \int_{\mathcal{X}\times\mathcal{Y}}  |y-\langle\sigma_*(\cdot,x),\mu_s^N\rangle| \,   \big | \langle\Upsilon_s^N,\mathsf T_x^*\Upsilon_s^N \rangle_{-J_1,j_1} \big | \pi(\di x,\di y) \Big] \di s\\
                                                     &\le C\int_0^t \textbf{E}\Big[ \int_{\mathcal{X}\times\mathcal{Y}}(  |y|+C)    \, | \langle\Upsilon_s^N,\mathsf T_x^*\Upsilon_s^N \rangle_{-J_1,j_1}  | \pi(\di x,\di y) \Big] \di s
\end{align*}
By Lemma~\ref{lem : <xi,gxi>}, detailed in Appendix~\ref{app:B}, and
since a.s.
$\Upsilon^N\in \mathcal D(\mathbf{R}_+,\mathcal
H^{1-J_1,j_1}(\mathbf{R}^d))$, there exists $C>0$ such that for all
$x\in\mathcal{X}$,
$|\langle\Upsilon_s^N,\mathsf T_x^*\Upsilon_s^N\rangle_{\mathcal
  H^{-J_1,j_1}}|\leq C\|\Upsilon_s^N\|_{\mathcal
  H^{-J_1,j_1}}^2$. Hence, since $\mathbf E[|y|]<+\infty$, we deduce
that:
$$\sum_{a=1}^{+\infty} \textbf{E}[\textbf I^N_t[f_a]]\le C\int_0^t\textbf{E}\left[\left\|\Upsilon_s^N\right\|_{\mathcal H^{-J_1,j_1}}^2\right]\di s.$$
We have thus proved that
$\sum_{a=1}^{+\infty} \textbf{E}[\textbf A^N_t[f_a]]\le
C+C\int_0^t\textbf{E} [ \|\Upsilon_s^N \|_{\mathcal H^{-J_1,j_1}}^2
]\di s$. In conclusion, using also~\eqref{eq.Bfa}
and~\eqref{eq.Xi-fa}, we have
$\textbf{E}[\|\Upsilon_t^N\|_{\mathcal H^{-J_1,j_1}}^2]\le
C+C\int_0^t\textbf{E}[\|\Upsilon_s^N \|_{\mathcal H^{-J_1,j_1}}^2 ]\di
s$.
%
%
By Gronwall's Lemma, we get:
$$
\sup_{N\geq1}\sup_{t\in[0,T]} \mathbf{E}[\|\Upsilon_t^N\|_{\mathcal H^{-J_1,j_1}}^2]<+\infty.
$$
Together with the first step, this ends the proof of the Lemma (recall
the decomposition $\eta^N=\Upsilon^N+\Theta^N$,
see~\eqref{eq.dec-eta2}).
\end{proof}


\begin{lem} 
\label{le.martingale}
Assume~\ref{as:batch}-\ref{as:batch limite}.  Introduce the following
$\boldsymbol{\sigma}$-algebra (see~\eqref{eq.filtration}):
\begin{equation*}
\mathfrak{F}_t^N:=\mathcal{F}_{\lfloor Nt\rfloor}^N, \  t\in \mathbf R_+.
\end{equation*}
Then, for all $f\in \mathcal C^{2,\gamma_*}(\mathbf R^d)$, the two processes   
\begin{equation}\label{eq.proc}
 \Big \{t \mapsto \sum_{k=0}^{\lfloor Nt\rfloor-1}\sum_{i=1}^N\nabla f(W_k^i)\cdot\ve_k^i,\ t\in \mathbf R_+\Big \} \text{ and } \big\{t \mapsto \langle f, M_t^N\rangle,t\in \mathbf R_+\big\} \text{ are $\mathfrak{F}^N_t$-martingale.}
\end{equation}
\end{lem}


\begin{proof}  
  Recall that by Lemma~\ref{le.W}, the first process
  in~\eqref{eq.proc} is integrable.  By~\eqref{eq.MTT}
  and~\eqref{eq.def-M}, the second process in~\eqref{eq.proc} is
  integrable.  For $0\le s<t$, we write
\begin{align}\label{eq 1670}
\textbf{E}\Big[\sum_{k=0}^{\lfloor Nt\rfloor-1}\sum_{i=1}^N\nabla f(W_k^i)\cdot\ve_k^i|\mathfrak{F}_s^N\Big]=\sum_{k=0}^{\lfloor Ns\rfloor-1}\sum_{i=1}^N\textbf{E}\big[\nabla f(W_k^i)\cdot\ve_k^i|\mathcal{F}_{\lfloor Ns\rfloor}^N\big]+\mathfrak R^N_{t,s}[f],
\end{align}
where
 $$\text{if $\lfloor Nt\rfloor=\lfloor Ns\rfloor$: $\, \mathfrak R^N_{t,s}[f]=0$,    
   and if $\lfloor Nt\rfloor>\lfloor Ns\rfloor$: } \, \mathfrak
 R^N_{t,s}[f]=\sum_{k=\lfloor Ns\rfloor}^{\lfloor
   Nt\rfloor-1}\sum_{i=1}^N\textbf{E}[\nabla
 f(W_k^i)\cdot\ve_k^i|\mathcal{F}_{\lfloor Ns\rfloor}^N].
 $$
 When $\lfloor Nt\rfloor>\lfloor Ns\rfloor$, since the $W_k^i$'s are
 $\mathcal{F}_k^N$-measurable (see~\eqref{algorithm}
 and~\eqref{eq.filtration}) and the $\ve_k^i$'s are centered and
 independent of $\mathcal{F}_k^N$ (see~\ref{as:noise}), we have, for
 $ k\ge \lfloor Ns\rfloor$ and $i\in \{1,\dots,N\}$:
$$\textbf{E}\big [\nabla f(W_k^i)\cdot\ve_k^i|\mathcal{F}_{\lfloor Ns\rfloor}^N\big ]=\textbf{E}\Big [\nabla f(W_k^i)\cdot \textbf{E}[ \ve_k^i|\mathcal{F}_k^N]\Big |\mathcal{F}_{\lfloor Ns\rfloor}^N\Big ]=0.$$
Hence, for all $0\leq s<t$, $\mathfrak R^N_{t,s}[f]=0$.  Furthermore,
for $0\le k\le \lfloor Ns\rfloor-1$ and $i\in \{1,\dots,N\}$,
$\nabla f(W_k^i)\cdot\ve _k^i$ is $\mathcal{F}_{k+1}^N$-measurable and
thus is $\mathcal{F}_{\lfloor
  Ns\rfloor}^N$-measurable. Therefore,~\eqref{eq 1670} reduces to
$$\textbf{E}\Big [\sum_{k=0}^{\lfloor Nt\rfloor-1}\sum_{i=1}^N\nabla f(W_k^i)\cdot\ve_k^i \Big |\mathfrak{F}_s^N\Big ]=\sum_{k=0}^{\lfloor Ns\rfloor-1}\sum_{i=1}^N\nabla f(W_k^i)\cdot\ve_k^i.$$
This proves that
$\{t \mapsto \sum_{k=0}^{\lfloor Nt\rfloor-1}\sum_{i=1}^N\nabla
f(W_k^i)\cdot\ve_k^i,\ t\in \mathbf R_+\}$ is a
$\mathfrak{F}^N_t$-martingale. Let us now prove that the process
$\big\{t \mapsto \langle f, M_t^N\rangle,t\in \mathbf R_+\big\}$ is a
$\mathfrak{F}^N_t$-martingale (see~\eqref{eq.def-M}).
We have, for $0\leq s<t$, 
\begin{align*} 
\textbf{E}\big [\langle f, M_t^N\rangle|\mathfrak{F}_s^N\big ]&=\sum_{k=0}^{\lfloor Ns\rfloor-1} \textbf{E}\big[  \langle f,M_k^N\rangle\big|\mathcal{F}_{\lfloor Ns\rfloor}^N\big] +\mathfrak E^N_{t,s}[f],
\end{align*}
where
 $$\text{if $\lfloor Nt\rfloor=\lfloor Ns\rfloor$: $\, \mathfrak E^N_{t,s}[f]=0$,    
   and if $\lfloor Nt\rfloor>\lfloor Ns\rfloor$: } \, \mathfrak
 E^N_{t,s}[f]=\sum_{k=\lfloor Ns\rfloor}^{\lfloor Nt\rfloor-1}
 \textbf{E}\big [ \langle f,M_k^N\rangle|\mathcal{F}_{\lfloor
   Ns\rfloor}^N\big ].
 $$
 When $\lfloor Nt\rfloor>\lfloor Ns\rfloor$, we have for
 $k\ge \lfloor Ns\rfloor$, by~\eqref{eq.MkE=0}:
$$\textbf{E}\big [ \langle f,M_k^N\rangle|\mathcal{F}_{\lfloor Ns\rfloor}^N\big ]=\textbf{E}\Big[ \textbf{E}\big [ \langle f,M_k^N\rangle |\mathcal{F}_k^N\big ] \Big |\mathcal{F}_{\lfloor Ns\rfloor}^N\Big]=0.$$
Hence, for all $0\leq s<t$, $\mathfrak E^N_{t,s}[f]=0$. In addition,
for $k\le \lfloor Ns\rfloor-1$, $\langle f,M_k^N\rangle$ is
$\mathcal{F}_{k+1}^N$-measurable and thus is
$\mathcal{F}_{\lfloor Ns\rfloor}^N$-measurable. In conclusion,
$\big\{t \mapsto \langle f, M_t^N\rangle,t\in \mathbf R_+\big\}$ is a
$\mathfrak{F}^N_t$-martingale. This ends the proof of
Lemma~\ref{le.martingale}. \end{proof}

The following lemma provides the compact containment condition needed
to prove that $(\eta^N)_{N\ge 1}$ is relatively compact in
$\mathcal D(\mathbf R_+,\mathcal H^{-J_0+1,j_0}(\mathbf R^d))$.

\begin{lem}\label{clt lem CC of eta^N}
Let $\beta\ge 3/4$ and assume~\ref{as:batch}-\ref{as:batch limite}.  Then, for all $T>0$, 
\begin{equation}\label{equation de lemme cc eta^N}
 \sup_{N\geq1} \mathbf{E}\Big[ \sup_{t\in[0,T]} \|\eta_t^N \|_{\mathcal H^{-J_2,j_2}}^2\Big]<+\infty.
\end{equation}
\end{lem}
\begin{proof} Let $T>0$. All along the proof, $C>0$ denotes a constant
  indendent of $t\in[0,T]$ and $N\geq1$, which can change from one
  occurence to another.  Recall that $\eta^N=\Upsilon^N+\Theta^N$,
  see~\eqref{eq.dec-eta2}.  By~\eqref{f contre xi } and Jensen's
  inequality, it holds for $f\in \mathcal H^{J_2,j_2}(\mathbf{R}^d)$
  (recall
  $\mathcal H^{J_2,j_2}(\mathbf{R}^d)\hookrightarrow \mathcal
  H^{J_1,j_1}(\mathbf{R}^d)$, see~\eqref{eq.SE2}),
\begin{equation}\label{217}\begin{split}
\sup_{t\in[0,T]}\langle f,\Upsilon_t^N\rangle^2&\leq C\Bigg[\int_0^T\int_{\mathcal{X}\times\mathcal{Y}}\left|(y-\langle\sigma_*(\cdot,x),\mu_s^N\rangle)\langle\nabla f\cdot\nabla\sigma_*(\cdot,x),\Upsilon_s^N\rangle\right|^2\pi(\di s,\di y)\di s\\
&\quad+\int_0^T\int_{\mathcal{X}\times\mathcal{Y}}\left|\langle\sigma_*(\cdot,x),\Upsilon_s^N\rangle\langle\nabla f\cdot\nabla\sigma_*(\cdot,x),\bar \mu_s^N\rangle\right|^2\pi(\di x,\di y)\di s\\
&\quad+\int_0^T\int_{\mathcal{X}\times\mathcal{Y}}\left|\langle\sigma_*(\cdot,x),\sqrt{N}(\bar \mu_s^N-\bar \mu_s)\rangle\langle\nabla f\cdot\nabla\sigma_*(\cdot,x),\bar \mu_s^N\rangle\right|^2\pi(\di x,\di y)\di s\\
&\quad+  {N\sup_{t\in[0,T]}|\langle f,M_t^N\rangle|^2}  + N\sup_{t\in[0,T]}\left|\langle f, V_t ^N\rangle\right|^2+N \sup_{t\in[0,T]}\Big|\sum_{k=0}^{\lfloor Nt\rfloor-1}\langle f,R_k^{N}\rangle\Big|^2\\
&\quad +\frac{1}{N^{1+2\beta}}\sup_{t\in[0,T]} \Big|\sum_{k=0}^{\lfloor Nt\rfloor -1}\sum_{i=1}^N\nabla f(W_k^i)\cdot\ve_k^i\Big|^2\Bigg].
\end{split}\end{equation}
We now consider successively each term in the right-hand-side
of~\eqref{217}. By Lemma~\ref{lem : unif bound xi and eta}, for
$0\leq s\leq T$, we have using also~\ref{as:sigma} and
$\mathbf E[|y|^2]<+\infty$ (see~\ref{as:data}):
\begin{align}
\nonumber
&\textbf{E} \Big [\int_{\mathcal{X}\times\mathcal{Y}}  |y-\langle\sigma_*(\cdot,x),\mu_s^N\rangle|^2\langle\nabla f\cdot\nabla\sigma_*(\cdot,x),\Upsilon_s^N\rangle^2 \pi(\di x,\di y)\Big ]\\
&\leq C\textbf{E} \Big [\int_{\mathcal{X}\times\mathcal{Y}} (|y|^2+1)\|\nabla f\cdot\nabla\sigma_*(\cdot,x)\|_{\mathcal H^{J_1,j_1}}\|\Upsilon_s^N\|_{\mathcal H^{-J_1,j_1}}^2 \pi(\di x,\di y)\Big ]\nonumber\\
\label{eq 1st term}
&\leq C\|f\|_{\mathcal H^{J_1+1,j_1}}^2\textbf{E}[\|\Upsilon_s^N\|_{\mathcal H^{-J_1,j_1}}^2]\leq C\|f\|_{\mathcal H^{J_1+1,j_1}}^2.
\end{align}
Let us now study the second term in~\eqref{217}. By~\eqref{31} and
since $ \sup_{x\in \mathcal X}\Vert \sigma_*(\cdot,x)\Vert<+\infty$
(because $j_1>d/2$ and
$\sigma_*\in \mathcal C^\infty_b(\mathbf{R}^d\times \mathcal{X})$
by~\ref{as:sigma}), for $0\leq s\leq T$,
\begin{align}\label{eq 2nd term}
\textbf{E}\left[\left|\langle\sigma_*(\cdot,x),\Upsilon_s^N\rangle\langle\nabla f\cdot\nabla\sigma_*(\cdot,x),\bar \mu_s^N\rangle\right|^2\right]\leq C\|f\|_{\mathcal H^{L,\gamma}}^2\textbf{E}\left[\|\Upsilon_s^N\|_{\mathcal H^{-J_1,j_1}}^2\right]\leq C\|f\|_{\mathcal H^{L,\gamma}}^2.
\end{align}
Let us now consider the third term in~\eqref{217}. By~\eqref{31} and
~\eqref{eq : tilde(mu)-bar(mu)}, we have for $0\leq s\leq T$,
\begin{align}\label{eq 3rd term}
  \textbf{E} \Big [\langle\sigma_*(\cdot,x),\sqrt{N}(\bar \mu_s^N-\bar \mu_s)\rangle^2\langle\nabla f\cdot\nabla\sigma_*(\cdot,x),\bar \mu_s^N\rangle^2 \Big ]\leq C\|f\|_{\mathcal H^{L,\gamma}}^2.
\end{align} 
Let us now deal with the fourth term in~\eqref{217}.  By
Lemma~\ref{le.martingale}, we have using Doob's maximal inequality,
$\textbf{E} [\sup_{t\in[0,T]}\langle f, M_t^N\rangle^2 ]\leq
C\textbf{E} [\langle f, M_T^N\rangle^2 ]$.  Then by
Lemma~\ref{lem:remainder terms} and since
$\mathcal H^{L,\gamma} (\mathbf R^d)\hookrightarrow \mathcal
C^{2,\gamma_*}(\mathbf R^d)$ (recall indeed~\eqref{eq.SE1}), we obtain
\begin{align}\label{eq finale cc M}
N\textbf{E} \Big [\sup_{t\in[0,T]}\langle f, M_t^N\rangle^2 \Big ]\leq C\|f\|^2_{\mathcal C^{2,\gamma_*}} \le C \|f\|^2_{\mathcal H^{L,\gamma}} .
\end{align} 
Using~\eqref{eq.supVt} and again~\eqref{eq.SE1},   the fifth term in~\eqref{217} satisfies:
\begin{equation}\label{eq 4th term}
N\textbf{E} \Big [\sup_ {t\in[0,T]} |\langle f,V_t^N\rangle|^2  \Big ]\leq CN^{-1/2}\|f\|_{\mathcal H^{L,\gamma}}^2.
\end{equation}
Using~\eqref{eq.borneERk2}, the sixth term in~\eqref{217} satisfies: 
\begin{align*}
\textbf{E}\Big[N {\sup}_{t\in[0,T]}\Big|\sum_{k=0}^{\lfloor Nt\rfloor-1}\langle f,R_k^{N}\rangle\Big|^2\Big]\leq C\|f\|^2_{\mathcal H^{L,\gamma}}(1/N+N^3/N^{4\beta}).
\end{align*}
Let us deal with the last term in the right-hand side of~\eqref{217}
for which we need a more accurate upper bound
than~\eqref{eq.boundRk1}.  By Lemma~\ref{le.martingale} and Doob's
maximal inequality, we have using~\eqref{eq333} and
$\mathcal H^{L,\gamma} (\mathbf R^d)\hookrightarrow \mathcal
C^{2,\gamma_*}(\mathbf R^d)$,
\begin{align*}
\nonumber
\frac{1}{N^{1+2\beta}}\textbf{E}\Big[\sup_{t\in[0,T]}\! \Big|\sum_{k=0}^{\lfloor Nt\rfloor -1}\!\sum_{i=1}^N\nabla f(W_k^i)\cdot\ve_k^i\Big|^2\Big]&\le\!\! \frac{C}{N^{1+2\beta}}\textbf{E}\Big[\Big|\sum_{k=0}^{\lfloor NT\rfloor -1}\!\sum_{i=1}^N\nabla f(W_k^i)\cdot\ve_k^i\Big|^2\Big]\\
\label{eq.Eps-k}
&\leq C\|f\|^2_{\mathcal H^{L,\gamma}}N^{1-2\beta}.
\end{align*}
Collecting these bounds, we obtain, from~\eqref{217}, for $f\in \mathcal H^{J_2,j_2}(\mathbf{R}^d)$,
\begin{equation}\begin{split}\label{32}
    \textbf{E}\Big[\sup_{t\in[0,T]}\langle
    f,\Upsilon_t^N\rangle^2\Big]&\leq C\left(\|f\|_{\mathcal
        H^{J_1+1,j_1}}^2 +\|f\|_{\mathcal H^{L,\gamma}}^2\right).
\end{split}\end{equation}
We now turn to the study of
$\textbf{E} [\sup_{t\in[0,T]}\langle f,\Theta^N_t\rangle^2]$ for
$f\in \mathcal H^{J_2,j_2}(\mathbf{R}^d)$.   By~\eqref{formule f mu
  tilde} and  Corollary \ref{co.e}, we have, for all
$t\in[0,T]$ (recall that
$\Theta_t^N=\sqrt{N}(\bar \mu_t^N-\bar \mu_t)$,
see~\eqref{eq.dec-eta}),
\begin{align}\label{eq.Theta_tf}
\langle f,\Theta_t^N\rangle= \langle f,\Theta_0^N\rangle+ \int_0^t\int_{\mathcal{X}\times\mathcal{Y}}\alpha(y-\langle\sigma_*(\cdot,x),\bar \mu_s\rangle)\langle\nabla f\cdot\nabla\sigma_*(\cdot,x),\Theta_s^N\rangle\pi(\di x,\di y)\di s.
\end{align}
By Jensen's inequality together with~\ref{as:sigma} and~\eqref{borne Z},
one has:
\begin{align} 
\nonumber
\textbf{E}\Big[\sup_{t\in[0,T]} \langle f,\Theta_t^N\rangle^2\Big]&\leq C \textbf{E}\big [|\langle f,\Theta_0^N\rangle|^2\big ]+C\int_0^T\int_{\mathcal{X}\times\mathcal{Y}} (|y|^2+1)\textbf{E}\left[\langle\nabla f\cdot\nabla\sigma_*(\cdot,x),\Theta_s^N\rangle^2\right]\pi(\di x,\di y)\di s\\
\nonumber 
&\leq C \|f\|_{\mathcal H^{J_1,j_1}}^2 \textbf{E}\big [\Vert \Theta_0^N\Vert_{\mathcal H^{-J_1,j_1}}^2\big ]\\
\nonumber
&\quad+C\int_0^T\int_{\mathcal{X}\times\mathcal{Y}} (|y|^2+1)\, \|\nabla f\cdot\nabla\sigma_*(\cdot,x)\|_{\mathcal H^{J_1,j_1}}^2\,  \textbf{E}\left[\|\Theta_t^N\|^2_{\mathcal H^{-J_1,j_1}}\right]\pi(\di x,\di y)\di s\\ 
\label{34}
&\leq C \|f\|_{\mathcal H^{J_1+1,j_1}}^2 .
\end{align}
Let $\{f_a\}_{a\geq1}$ be an orthonormal basis of
$\mathcal H^{J_2,j_2}(\mathbf{R}^d)$. 
 Let us recall that (see~\eqref{eq.SE2}) 
$\mathcal H^{J_2,j_2}(\mathbf{R}^d)\hookrightarrow_{\text{H.S.}}
\mathcal H^{J_1+1,j_1}(\mathbf{R}^d)$ and $\mathcal H^{J_1+1,j_1}(\mathbf{R}^d)\hookrightarrow \mathcal H^{L,\gamma}(\mathbf{R}^d)$.  Then, by~\eqref{32}
and~\eqref{34}, we obtain, since $\beta\ge 3/4$,
\begin{align*}
\textbf{E}\Big[\sup_{t\in[0,T]} \|\eta_t^N\|^2_{\mathcal H^{-J_2,j_2}}\Big] \!\!&=\textbf{E}\Big[\sup_{t\in[0,T]} \sum_{a\geq1}\langle f_a,\eta_t^N\rangle^2\Big]\leq \sum_{a\geq1}\textbf{E}\Big[\sup_{t\in[0,T]} \langle f_a,\eta_t^N\rangle^2\Big]  \\
&\leq 2\sum_{a\geq 1}\Big(\textbf{E}\Big[\sup_{t\in[0,T]}\langle f_a,\Upsilon_t^N\rangle^2\Big]+\textbf{E}\Big[\sup_{t\in[0,T]}\langle f_a,\Theta_t^N\rangle^2\Big]\Big)\\
&\leq C\sum_{a\geq1}\left(\|f_a\|_{\mathcal H^{J_1+1,j_1}}^2+\|f_a\|_{\mathcal H^{L,\gamma}}^2 \right) \leq C.
\end{align*}
This concludes the proof of the lemma. 
\end{proof}

The following result provides the regularity condition needed to prove
that $(\eta^N)_{N\ge 1}$ is relatively compact in
$\mathcal D(\mathbf R_+,\mathcal H^{-J_0+1,j_0}(\mathbf R^d))$.

\begin{lem}\label{clt lem reg cond eta}
  Let $\beta\ge 3/4$ and assume~\ref{as:batch}-\ref{as:batch limite}.
  Let $T>0$. Then, there exists $C>0$ such that for all $\delta>0$ and
  $0\leq r<t\leq T$ such that $t-r\leq \delta$, one
  has 
\begin{align}\label{eq reg eta HJ2}
\mathbf{E}\left[\left\|\eta_t^N-\eta_r^N\right\|_{\mathcal H^{-J_2,j_2}}^2\right]\leq C\left[\delta^2+\frac{N\delta+1}{N}+\frac{1}{\sqrt N}+ (N\delta+1)^2 \Big(\frac 1{N^3}+\frac 1{N^{4\beta-1}}\Big)+ \frac{N\delta+1}{N^{2\beta}}\right].
\end{align}
\end{lem}
\begin{proof} 
 Let $T>0$.  In the following, $C>0$
  is a constant independent of $\delta>0$, $0\leq r<t\leq T$, $N\ge1$, and $f\in \mathcal H^{J_2,j_2}(\mathbf{R}^d)$ which can
  change from one occurence to another. In what follows $t-r\le \delta$. Recall  $\eta^N=\Upsilon^N+\Theta^N$,
  see~\eqref{eq.dec-eta2}. Using~\eqref{f contre xi } and
the Jensen's inequality, one has:
\begin{align} 
\nonumber
\left|\langle f,\Upsilon_t^N\rangle-\langle f,\Upsilon_r^N\rangle\right|^2&\leq C\Big[(t-r)\int_r^t\int_{\mathcal{X}\times\mathcal{Y}}\left|(y-\langle\sigma_*(\cdot,x),\mu_s^N\rangle)\langle\nabla f\cdot\nabla\sigma_*(\cdot,x),\Upsilon_s^N\rangle\right|^2\pi(\di s,\di y)\di s \\
\nonumber
&+(t-r)\int_r^t\int_{\mathcal{X}\times\mathcal{Y}}\left|\langle\sigma_*(\cdot,x),\Upsilon_s^N\rangle\langle\nabla f\cdot\nabla\sigma_*(\cdot,x),\bar \mu_s^N\rangle\right|^2\pi(\di x,\di y)\di s\\
\nonumber
&+(t-r)\int_r^t\int_{\mathcal{X}\times\mathcal{Y}}\left|\langle\sigma_*(\cdot,x),\sqrt{N}(\bar \mu_s^N-\bar \mu_s)\rangle\langle\nabla f\cdot\nabla\sigma_*(\cdot,x),\bar \mu_s^N\rangle\right|^2\pi(\di x,\di y)\di s\\
\nonumber
&+ {N|\langle f,M_t^N-M_r^N\rangle|^2} + N\left|\langle f,V_t^N\rangle-\langle f,V_r^N\rangle\right|^2+N \Big|\sum_{k=\lfloor Nr\rfloor}^{\lfloor Nt\rfloor-1}\langle f,R_k^{N}\rangle\Big|^2\\
\label{197}
&\quad +\frac{1}{N^{1+2\beta}}\Big|\sum_{k=\lfloor Nr\rfloor}^{\lfloor Nt\rfloor -1}\sum_{i=1}^N\nabla f(W_k^i)\cdot\ve_k^i\Big|^2\, \Big].
 \end{align}
We now study each term of the right-hand side of~\eqref{197}.
%
By~\eqref{eq 1st term}, we bound the first term in~\eqref{197} as follows:
\begin{align*}
\textbf{E}\Big[(t-r)\int_r^t\int_{\mathcal{X}\times\mathcal{Y}}\left|(y-\langle\sigma_*(\cdot,x),\mu_s^N\rangle)\langle\nabla f\cdot\nabla\sigma_*(\cdot,x),\Upsilon_s^N\rangle\right|^2\pi(\di s,\di y)\di s\Big]\leq C\delta^2\|f\|_{\mathcal H^{J_1+1,j_1}}^2.
\end{align*}
Using~\eqref{eq 2nd term}, we bound the second term of~\eqref{197} as follows:
\begin{align*}
\textbf{E}\Big[(t-r)\int_r^t\int_{\mathcal{X}\times\mathcal{Y}}\left|\langle\sigma_*(\cdot,x),\Upsilon_s^N\rangle\langle\nabla f\cdot\nabla\sigma_*(\cdot,x),\bar \mu_s^N\rangle\right|^2\pi(\di x,\di y)\di s\Big]\leq C\delta^2\|f\|_{\mathcal H^{L,\gamma}}^2.
\end{align*}
Using~\eqref{eq 3rd term}, we have the following bound on the third term of~\eqref{197}:
\begin{align*}
\textbf{E}\Big[(t-r)\int_r^t\int_{\mathcal{X}\times\mathcal{Y}}\left|\langle\sigma_*(\cdot,x),\sqrt{N}(\bar \mu_s^N-\bar \mu_s)\rangle\langle\nabla f\cdot\nabla\sigma_*(\cdot,x),\bar \mu_s^N\rangle\right|^2\pi(\di x,\di y)\di s\Big]\leq C\delta^2\|f\|_{\mathcal H^{L,\gamma}}^2.
\end{align*}
In addition, we have, using~\eqref{eq.Mk==},~\eqref{eq.MTT}, and
$\mathcal H^{L,\gamma} (\mathbf R^d)\hookrightarrow \mathcal
C^{2,\gamma_*}(\mathbf R^d)$ (see~\eqref{eq.SE1}),
\begin{align}
\nonumber
N\textbf{E}[|\langle f,M_t^N-M_r^N\rangle|^2]&=N\textbf{E}\Big[\Big|\sum_{k=\lfloor Nr\rfloor}^{\lfloor Nt\rfloor-1}\langle f,M_k^N\rangle\Big|^2\Big]=N\sum_{k=\lfloor Nr\rfloor}^{\lfloor Nt\rfloor-1}\textbf{E}\Big[\langle f,M_k^N\rangle^2\Big]\\
\label{eq.Mk-rt}
&\leq CN(\lfloor Nt\rfloor-\lfloor Nr\rfloor)\|f\|_{\mathcal{C}^{2,\gamma_*}}^2/N^2\leq C(N\delta+1)\|f\|_{\mathcal H^{L,\gamma}}^2/N. 
\end{align}
The fifth term of~\eqref{197} is bounded as follows using~\eqref{eq
  4th term}:
$$\textbf{E}\left[N|\langle f,V_t^N\rangle-\langle f,V_r^N\rangle|^2\right]\leq 2N\textbf{E}[|\langle f,V_t^N\rangle|^2]+2N\textbf{E}[|V_r^N[f]|^2]\leq {C\|f\|_{\mathcal H^{L,\gamma}}^2}/{\sqrt N}.$$
Let us consider the sixth term in the right-hand side of~\eqref{197}.
By~\eqref{convexity inequality},~\eqref{borneERk]}, and because
$\mathcal H^{L,\gamma} (\mathbf R^d)\hookrightarrow \mathcal
C^{2,\gamma_*}(\mathbf R^d)$, we have that:
\begin{align*}
\textbf{E}\Big[N \Big|\sum_{k=\lfloor Nr\rfloor}^{\lfloor Nt\rfloor-1}\langle f,R_k^{N}\rangle\Big|^2\Big]\leq N(\lfloor Nt\rfloor-\lfloor Nr\rfloor)\!\!\sum_{k=\lfloor Nr\rfloor}^{\lfloor Nt\rfloor-1}\!\!\!\!\textbf{E}\left[|\langle f,R_k^{N}\rangle|^2\right]\leq CN(N\delta+1)^2 \|f\|^2_{\mathcal H^{L,\gamma}}(1/N^4+1/N^{4\beta}).
\end{align*}
Let us consider the last term in the right-hand side
of~\eqref{197}. Using~\eqref{eq.ekekj-0} and~\eqref{eq.ekekj}, we
have:
\begin{align*}
\textbf{E}\Big[\frac{1}{N^{1+2\beta}}\Big|\sum_{k=\lfloor Nr\rfloor}^{\lfloor Nt\rfloor -1}\sum_{i=1}^N\nabla f(W_k^i)\cdot\ve_k^i\Big|^2\Big]=\frac{1}{N^{1+2\beta}}\sum_{k=\lfloor Nr\rfloor}^{\lfloor Nt\rfloor -1}\sum_{i=1}^N\textbf{E}\left[|\nabla f(W_k^i)\cdot\ve_k^i|^2\right]\\
\leq\frac{1}{N^{1+2\beta}}\times C\|f\|^2_{\mathcal H^{L,\gamma}}N(N\delta+1)=\frac{C\|f\|^2_{\mathcal H^{L,\gamma}}}{N^{2\beta}}(N\delta+1).
\end{align*}
Let $\{f_a\}_{a\geq1}$ be an orthonormal basis of
$\mathcal H^{J_2,j_2}(\mathbf{R}^d)$. Gathering the previous bounds,
we obtain, using also~\eqref{eq.SE2},
\begin{align}\label{eq reg xi}
\textbf{E}\big [\|\Upsilon_t^N-\Upsilon_r^N\|_{\mathcal H^{-J_2,j_2}}^2\big ]&=\sum_{a=1}^{+\infty}\textbf{E}\left[\left|\langle f_a,\Upsilon_t^N\rangle-\langle f_a,\Upsilon_r^N\rangle\right|^2\right]\nonumber\\
&\leq C\left[\delta^2+\frac{N\delta+1}{N}+\frac{1}{\sqrt N}+ (N\delta+1)^2 \Big(\frac 1{N^3}+\frac 1{N^{4\beta-1}}\Big)+ \frac{N\delta+1}{N^{2\beta}}\right].
\end{align}
By~\eqref{eq.Theta_tf} and using the same arguments leading
to~\eqref{34}, we obtain for
$f\in \mathcal H^{J_2,j_2}(\mathbf{R}^d)$,
\begin{align*}
\textbf{E}\left[|\langle f,\Theta_t^N\rangle-\langle f,\Theta_r^N\rangle|^2\right]\leq C\delta\int_r^t\textbf{E}\left[\|\nabla f\|_{\mathcal H^{J_1,j_1}}^2\|\Theta_t^N\|^2_{\mathcal H^{-J_1,j_1}}\right] \di s\leq C\delta^2 \|f\|^2_{\mathcal H^{J_1+1,j_1}}.
\end{align*}
Considering an orthonormal basis of
$\mathcal H^{J_2,j_2}(\mathbf{R}^d)$, and using the fact that
$\mathcal H^{J_2,j_2}(\mathbf{R}^d)\hookrightarrow_{\text{H.S.}}
\mathcal H^{J_1+1,j_1}(\mathbf{R}^d)$, we obtain
$\textbf{E}[\|\Theta_t^N-\Theta_r^N\|_{\mathcal H^{-J_2,j_2}}^2]\leq
C\delta^2$. Combining this result with~\eqref{eq reg xi}, we
obtain~\eqref{eq reg eta HJ2}. This concludes the proof of the lemma.
\end{proof}
Now, we collect the results of Lemmata~\ref{clt lem CC of eta^N}
and~\ref{clt lem reg cond eta} to prove the following result.

\begin{prop}\label{prop relative compactness}
  Let $\beta\ge 3/4$ and assume~\ref{as:batch}-\ref{as:batch
    limite}. Then, the sequence $(\eta^N)_{N\geq1}$ is relatively
  compact
in $\mathcal D(\mathbf{R}_+,\mathcal H^{-J_0+1,j_0}(\mathbf{R}^d))$. 
\end{prop}
\begin{proof} 
Recall
  $\mathcal H^{J_0-1,j_0}(\mathbf{R}^d)\hookrightarrow_{\text{H.S.}}
  \mathcal H^{J_2,j_2} (\mathbf{R}^d)$ (by~\eqref{eq.SE2}).  Using
  Markov's inequality, Lemma~\ref{clt lem CC of eta^N} implies item~1
  in Proposition~\ref{lem:note kurtz}. In addition, according to
  Lemma~\ref{clt lem reg cond eta}, item~2 in
  Proposition~\ref{lem:note kurtz} is satisfied. Consequently,
  $(\eta^N)_{N\geq1}$ is relatively compact in
  $\mathcal D(\mathbf{R}_+,\mathcal H^{-J_0+1,j_0}(\mathbf{R}^d))$. The result follows from Proposition~\ref{lem:note kurtz}.   
\end{proof}

 To prove Proposition \ref{prop relative compactness}, we mention that one could also have used  \cite[Theorem 4.6]{jakubowski1986skorokhod} with $E= \mathcal H^{-J_0+1,j_0}(\mathbf{R}^d)$  and $\mathbb F=\{\mathsf L_f, f\in \mathcal C^\infty_c(\mathbf R^{d})\}$ where $\mathsf  L_f: \Phi \in\mathcal H^{-J_0+1,j_0}(\mathbf{R}^d)\mapsto \langle f, \Phi \rangle_{J_0-1,j_0}$.

\subsection{Relative compactness of $(\sqrt{N}M^N)_{N\ge 1}$ in 
  $\mathcal D(\mathbf{R}_+,\mathcal H^{-J_0,j_0}(\mathbf{R}^d))$}


We begin this section with the compact containment condition on the
sequence
$\lbrace t\mapsto\sqrt NM_t^N,t\in\mathbf{R}_+\rbrace_{N\geq1}$
(see~\eqref{def Mk} and~\eqref{eq.def-M}).

\begin{lem}
\label{cc of M_t^N}
Let $\beta\ge 3/4$ and assume~\ref{as:batch}-\ref{as:batch
  limite}. Then, for all $T>0$,
\begin{equation*}
\sup_{N\geq1}\mathbf{E}\Big[\sup_{t\in[0,T]}\big\|\sqrt{N}M_t^N\big\|^2_{\mathcal H^{-J_1,j_1}}\Big]<+\infty.
\end{equation*}
\end{lem}
\begin{proof} Let $T>0$ and $f\in \mathcal H^{J_1,j_1}(\mathbf{R}^d)$.
Then, according to~\eqref{eq finale cc M}, we have: 
$$
\textbf{E} \Big [\sup_{t\in[0,T]}\langle f,\sqrt{N}M_t^N\rangle^2 \Big
]\leq C\|f\|^2_{\mathcal H^{L,\gamma}}.
$$
The proof of the lemma is complete considering an orthonormal basis
$\{f_a\}_{a\geq1}$ of
$\mathcal
H^{J_1,j_1}(\mathbf{R}^d)\hookrightarrow_{\text{H.S.}}\mathcal
H^{L,\gamma}(\mathbf{R}^d)$ (see~\eqref{eq.SE2}).
%
\end{proof}

Let us now turn to the regularity condition on the process
$\lbrace t\mapsto\sqrt NM_t^N,t\in\mathbf{R}_+\rbrace_{N\geq1}$.

\begin{lem}\label{lem: reg of M^N}Let $\beta\ge 3/4$ and
  assume~\ref{as:batch}-\ref{as:batch limite}. Fix $T>0$. Then, there
  exists $C>0$ such that for all $\delta>0$ and $0\leq r<t\leq T$ such
  that $t-r\leq\delta$, one has
$$
\mathbf{E}\Big[\Big\|\sqrt{N}M_t^N-\sqrt{N}M_r^N\Big\|_{\mathcal H^{-J_1,j_1}}^2\Big]\leq C\frac{N\delta+1}{N}.
$$

\end{lem}
\begin{proof} 
  This lemma is a direct consequence of~\eqref{eq.Mk-rt} (which also
  holds for $f\in \mathcal H^{J_1,j_1}(\mathbf{R}^d)$) together with
  the embedding
  $\mathcal
  H^{J_1,j_1}(\mathbf{R}^d)\hookrightarrow_{\text{H.S.}}\mathcal
  H^{L,\gamma}(\mathbf{R}^d)$ (see~\eqref{eq.SE2}).
\end{proof}


\begin{prop}\label{prop relative compactness M}
  Let $\beta\ge 3/4$ and assume~\ref{as:batch}-\ref{as:batch
    limite}. Then, the sequence
  $\lbrace t\mapsto\sqrt NM_t^N,t\in\mathbf{R}_+\rbrace_{N\geq1}$ is
  relatively compact in
  $\mathcal D(\mathbf{R}_+,\mathcal H^{-J_0,j_0}(\mathbf{R}^d))$.
\end{prop}
\begin{proof} It is a direct consequence of Proposition~\ref{lem:note
    kurtz}, Lemmata~\ref{cc of M_t^N} and~\ref{lem: reg of M^N},
  together with the embedding
  $\mathcal H^{J_0,j_0}(\mathbf{R}^d)\hookrightarrow_{\text{H.S.}}
  \mathcal H^{J_1,j_1}(\mathbf{R}^d)$ (see~\eqref{eq.SE2}).
%
%
\end{proof}

\subsection{Regularity of the limit points}

\begin{lem}
\label{lem: continuity prop eta}
Let $\beta>3/4$ and assume~\ref{as:batch}-\ref{as:batch limite}. Then, for all $T>0$, 
\begin{align}\label{eq continuity prop}
 \lim_{N\to +\infty} \mathbf{E}\Big[\sup_{t\in [0,T]}\big\|\eta_t^N-\eta_{t^-}^N\big\|_{\mathcal H^{-J_0+1,j_0}}^2\Big]=0.
\end{align}
In particular, any limit point of $(\eta^N)_{N\geq1}$ in
$\mathcal D(\mathbf R_+,\mathcal H^{-J_0+1,j_0}(\mathbf R^d))$ belongs
a.s. to
$ \mathcal C(\mathbf R_+,\mathcal H^{-J_0+1,j_0}(\mathbf{R}^d))$.
\end{lem}

\begin{proof} 
  Let $T>0$ and $f\in \mathcal H^{J_0-1,j_0}(\mathbf{R}^d)$.
%
%
  We have (see~\eqref{eq.dec-eta2}):
$$\sup_{t\in [0,T]}\|\eta_t^N-\eta_{t^-}^N\|_{\mathcal H^{-J_0+1,j_0}}^2\leq 2\ \sup_{t\in [0,T]}\|\Upsilon_t^N-\Upsilon_{t^-}^N\|_{\mathcal H^{-J_0+1,j_0}}^2+2\ \sup_{t\in [0,T]}\|\Theta_t^N-\Theta_{t^-}^N  \|_{\mathcal H^{-J_0+1,j_0}}^2.$$
 According to \eqref{eq.t-c} and since $\mathcal H^{J_0-1,j_0}(\mathbf{R}^d)\hookrightarrow  \mathcal H^{L,\gamma}(\mathbf{R}^d)$ (see \eqref{eq.SE2}), one has 
a.s. for all $t\in \mathbf R_+$ and $N\ge 1$,
$$\|\Theta_t^N-\Theta_{t^-}^N\|_{\mathcal H^{-J_0+1,j_0}}=0.$$ 
Since a.s. $\bar \mu^N \in \mathcal C(\mathbf R_+, \mathcal
  H^{-L,\gamma}(\mathbf R^d))$, by definition of $\Upsilon^N$ (see \eqref{eq.dec-eta}), it follows that a.s. for all $N\ge 1$, 
\begin{align*}
\sup_{t\in [0,T]}\langle f,\Upsilon_t^N-\Upsilon_{t^-}^N\rangle^2&= N \sup_{t\in [0,T]}\langle f,\mu_t^N-\mu_{t^-}^N\rangle^2
\end{align*}
From~\eqref{799}
and   the fact that
$\mathcal H^{L,\gamma} (\mathbf R^d)\hookrightarrow \mathcal
C^{2,\gamma_*}(\mathbf R^d)$, we obtain
\begin{align*}
\textbf{E}\Big[\sup_{t\in[0,T]}\langle f,\Upsilon_t^N-\Upsilon_{t^-}^N\rangle^2\Big]\leq C\|f\|^2_{\mathcal H^{L,\gamma}}\Big[\frac{1}{\sqrt{N}}+\sqrt{\frac{1}{N^5}+\frac{1}{N^{8\beta-3}}}+N^{\frac{3}{2}-2\beta}\Big].
\end{align*}
Using
$\mathcal H^{J_0-1,j_0}(\mathbf{R}^d)\hookrightarrow_{\text{H.S.}}
\mathcal H^{L,\gamma}(\mathbf{R}^d)$ (see~\eqref{eq.SE2}), we deduce
that
\begin{align*}
\textbf{E}\Big[\sup_{t\in[0,T]}\|\Upsilon_t^N-\Upsilon_{t^-}^N\|^2_{\mathcal H^{-J_0+1,j_0}}\Big] \leq C\Big[\frac{1}{\sqrt{N}}+\sqrt{\frac{1}{N^5}+\frac{1}{N^{8\beta-3}}}+N^{\frac{3}{2}-2\beta}\Big].
\end{align*}
Because $\beta >3/4$, this ends the proof of~\eqref{eq continuity
  prop}.
The second statement in Lemma~\ref{lem: continuity prop eta} follows
from Proposition~\ref{prop relative compactness},~\eqref{eq continuity
  prop}, and~\cite[Condition~3.28 in Proposition
3.26]{jacod2003skorokhod}. The proof of Lemma~\ref{lem: continuity
  prop eta} is complete.
\end{proof}

\begin{lem}\label{le.conVMM}
Let $\beta>3/4$ and assume~\ref{as:batch}-\ref{as:batch limite}. Then, for all $T>0$: 
\begin{align}\label{eq reg limit of NM}
\lim_{N\to +\infty} \mathbf{E}\Big[\sup_{t\in[0,T]}\|\sqrt{N}M_t^N-\sqrt{N}M_{t^-}^N\|_{\mathcal H^{-J_0,j_0}}^2\Big]=0. 
\end{align}
In particular, any limit point of 
$(\sqrt{N}M^N)_{N\ge 1}$ in
$\mathcal D(\mathbf R_+,\mathcal H^{-J_0,j_0}(\mathbf R^d))$ belongs
a.s. to $\mathcal C(\mathbf R_+,\mathcal H^{-J_0,j_0}(\mathbf R^d))$.
\end{lem}

\begin{proof} Let $T>0$ and $f\in \mathcal H^{J_0,j_0}(\mathbf{R}^d)$.
  The function
  $t\in[0,T]\mapsto\langle f,\sqrt{N}M_t^N\rangle\in\mathbf{R}$ has
  $\lfloor NT\rfloor$ discontinuities, located at the times
  $\frac{1}{N},\frac{2}{N},\dots,\frac{\lfloor NT\rfloor}{N}$. For
  $k\in\{1,\dots,\lfloor NT\rfloor\}$, its $k$-th discontinuity is
  equal to $\sqrt{N}\langle f,M_{k-1}^N\rangle$. Thus,
 $$\sup_{t\in[0,T]}\langle f,\sqrt{N}M_{t^-}^N-\sqrt{N}M_t^N\rangle^2=N \max_{0\leq k<\lfloor NT\rfloor}\langle f,M_k^N\rangle^2.$$
 Then, using~\eqref{E max f M_k} and because
 $\mathcal H^{L,\gamma} (\mathbf R^d)\hookrightarrow \mathcal
 C^{2,\gamma_*}(\mathbf R^d)$ (see indeed~\eqref{eq.SE1}),
\begin{equation}\label{eq.Mk-convD}
\textbf{E}[\sup_{t\in[0,T]}\langle f,\sqrt{N}M_{t^-}^N-\sqrt{N}M_t^N\rangle^2]\leq C\|f\|^2_{\mathcal C^{2,\gamma_*}}/\sqrt{N}\le C\|f\|^2_{\mathcal H^{L,\gamma}}/\sqrt{N}.
\end{equation}
Considering an orthonormal basis of
$\mathcal H^{J_0,j_0}(\mathbf{R}^d)$ and
$\mathcal H^{J_0,j_0}(\mathbf{R}^d)\hookrightarrow_{\text{H.S.}}
\mathcal H^{L,\gamma}(\mathbf{R}^d)$ (by~\eqref{eq.SE2}), we obtain
$$\mathbf{E}\Big[\sup_{t\in[0,T]}\|\sqrt{N}M_t^N-\sqrt{N}M_{t^-}^N\|_{\mathcal H^{-J_0,j_0}}^2\Big]\leq C/\sqrt N,$$
for some $C>0$ independent of $N\geq1$ and $f$. This proves~\eqref{eq
  reg limit of NM}. The second statement in Lemma~\ref{le.conVMM} is a
consequence of Proposition~\ref{prop relative compactness
  M},~\eqref{eq reg limit of NM}, and condition 3.28 of
\cite[Proposition 3.26]{jacod2003skorokhod}. The proof of
Lemma~\ref{le.conVMM} is complete.
\end{proof}

\subsection{Convergence of $(\sqrt{N}M^N)_{N\ge 1}$ to a G-process}

The aim of this section is to prove Proposition~\ref{prop.CVMN} below
which states that $\{ t\mapsto\sqrt NM_t^N,t\in\mathbf{R}_+\}_{N\ge1}$
(see~\eqref{def Mk} and~\eqref{eq.def-M}) converges towards a
G-process (see Definition~\ref{de.gaussian}). To this end, we first
show the convergence of this process against test functions.

\begin{prop}\label{prop:convergence to martingale0}
  Let $\beta>3/4$ and assume~\ref{as:batch}-\ref{as:batch
    limite}. Then, for every $f\in\mathcal C^{2,\gamma} (\mathbf R^d)$ the sequence
  $\{ t\mapsto \sqrt{N}\langle f,M_t^N\rangle,
  t\in\mathbf{R}_+\}_{N\geq1}$ (see~\eqref{eq.def-M}) converges in
  distribution in $\mathcal D(\mathbf{R}_+,\mathbf{R})$ towards a
  process $X^f \in \mathcal C(\mathbf{R}_+,\mathbf{R})$ that has
  independent Gaussian increments. Moreover, for all
  $t\in\mathbf{R}_+$,
$$\mathbf{E}[{X}^f_t]=0 \ \text{ and } \ \Var(X^f _t )=\alpha^2\mathbf{E}\left[\frac{1}{|B_{\infty}|}\right] \int_0^t \Var (\mathrm Q_s[f](x,y))\di s,$$
where we recall
$\mathrm Q_s[f](x,y)=(y-\langle \sigma_*(\cdot, x),\bar \mu_s
\rangle)\langle\nabla f\cdot\nabla\sigma_*(\cdot,x),\bar \mu_s\rangle$
(see Definition~\ref{de.gaussian}).
\end{prop}


\begin{proof}   
 Let   $f\in\mathcal C^{2,\gamma} (\mathbf R^d)$. Set for ease of notation
$$\mathfrak{m}_t^N[f]=\sqrt{N}\langle f,M_t^N\rangle.$$
To prove Proposition~\ref{prop:convergence to martingale} we apply the
martingale central limit theorem \cite[Theorem
7.1.4]{ethier2009markov} to the sequence
$\{ t\mapsto \mathfrak{m}_t^N[f] , t\in\mathbf{R}_+\}_{N\geq1}$. To
this end, let $T>0$. Let us first show that Condition (a) in
\cite[Theorem 7.1.4]{ethier2009markov} holds. First of all, by
\cite[Remark 7.1.5]{ethier2009markov} and~\eqref{eq.def-M}, the
covariation matrix of $ \mathfrak{m}_t^N[f]$ is
$\mathfrak a^N_t[f]=N\sum_{k=0}^{\lfloor Nt\rfloor -1}\langle
f,M_k^N\rangle^2$ and therefore
$\mathfrak a^N_t[f]-\mathfrak a^N_s[f]\ge 0$ if $t\ge s$. On the other
hand, by~\eqref{eq.Mk-convD} and \eqref{eq.SE1}, we have:
\begin{equation}\label{1.14 EK}
\lim_{N\to\infty}\textbf{E}\Big[\sup_{t\in[0,T]}\big |\mathfrak{m}_t^N[f]-\mathfrak{m}_{t^-}^N[f] \big |\Big]=0.
\end{equation}
Thus Condition (a) in \cite[Theorem 7.1.4]{ethier2009markov} holds.
%
Let us now prove the last required condition in \cite[Theorem
7.1.4]{ethier2009markov}, namely that for all $t\in\mathbf{R}_+$,
$\mathfrak a^N_t[f]\overset{p}{\to} \mathfrak c_t[f]$ where
$\mathfrak c$ satisfies the assumptions of \cite[Theorem
7.1.1]{ethier2009markov}, i.e.,
$t\in \mathbf R_+\mapsto \mathfrak c_t[f]$ continuous,
$\mathfrak c_0[f]=0$, and $\mathfrak c_t[f]-\mathfrak c_s[f]\ge 0$ if
$t\ge s$. Recall the definition of the $\boldsymbol{\sigma}$-algebra
$\mathcal{F}_k^N$ in~\eqref{eq.filtration}. For $t\in\mathbf{R}_+$,
\begin{align}\label{q variation}
\mathfrak a^N_t[f] 
&=N\sum_{k=0}^{\lfloor Nt\rfloor -1}\textbf{E}[\langle f,M_k^N\rangle^2|\mathcal{F}_k^N]+N\sum_{k=0}^{\lfloor Nt\rfloor -1}\left(\langle f,M_k^N\rangle^2-\textbf{E}[\langle f,M_k^N\rangle^2|\mathcal{F}_k^N]\right).
\end{align} 
We start by studying the first term in the right-hand side of~\eqref{q
  variation}. Recall that (see~\eqref{eq.QQ-d})
$$\mathrm Q^N[f](x,y,\{W_k^i\}_i)= (y-\langle\sigma_*(\cdot,x),\nu_k^N\rangle)\langle\nabla f\cdot \nabla\sigma_*(\cdot,x),\nu_k^N\rangle,$$   and set
 (see~\eqref{def Dk})
$$ \bar{\mathrm Q} ^N[f]( \{W_k^i\}_i) : =\int_{\mathcal{X}\times\mathcal{Y}} \mathrm Q^N[f](x,y,\{W_k^i\}_i)\pi(\di x,\di y) =\frac N \alpha\langle f,D_k^N\rangle.$$  
Using that $\big(|B_k|, ((x_k^n,y_k^n))_{n\ge 1}\big)\indep \mathcal  {F}_k^N$,    $|B_k|\indep ((x_k^n,y_k^n))_{n\ge 1}$ (see~\ref{as:batch}), and   $(W_k^1,\dots,W_k^N)$ is $\mathcal{F}_k^N$-measurable, it holds:  
\begin{align*}
 &\textbf{E}\Big[\frac{1}{|B_k|^2}\!\!\!\!\sum_{1\leq n<m\leq |B_k|}\!\!\!\!\!\!\Big( \mathrm Q^N[f](x_k^n,y_k^n,\{W_k^i\}_i)-\bar{\mathrm Q} ^N[f]( \{W_k^i\}_i) \Big)\Big( \mathrm Q^N[f](x_k^m,y_k^m,\{W_k^i\}_i)-\bar{\mathrm Q} ^N[f]( \{W_k^i\}_i) \Big)\Big| {\mathcal F}_k^N\Big]\\
 &= \sum_{q\ge 1} \frac{1}{q^2} \sum_{1\leq n<m\leq q} \textbf{E}\Big[  \mathbf 1_{|B_k|=q} \Big( \mathrm Q_k^N[f](x_k^n,y_k^n,\{W_k^i\}_i)-\bar{\mathrm Q} ^N[f]( \{W_k^i\}_i) \Big)\\
 &\quad \times \Big(\mathrm Q^N[f](x_k^m,y_k^m,\{W_k^i\}_i)-\bar{\mathrm Q} ^N[f]( \{W_k^i\}_i) \Big)\Big|  {\mathcal F}_k^N\Big]\\
 &=  \sum_{q\ge 1} \frac{1}{q^2}   \sum_{1\leq n<m\leq q} \textbf{E}\Big[  \mathbf 1_{|B_k|=q}  \Big( \mathrm Q^N[f](x_k^n,y_k^n,\{w_k^i\}_i)-\bar{\mathrm Q} ^N[f]( \{w_k^i\}_i) \Big)\\
 &\quad \times\Big( \mathrm Q_k^N[f](x_k^m,y_k^m,\{w_k^i\}_i)-\bar{\mathrm Q} ^N[f]( \{w_k^i\}_i) \Big) \Big]\Big|_{\{w_k^i\}_i=\{W_k^i\}_i}\\
 &= \sum_{q\ge 1} \frac{1}{q^2}    \sum_{1\leq n<m\leq q}  \textbf{E} [  \mathbf 1_{|B_k|=q}] \, \\
 &\quad \times \textbf{E}\Big[   \Big( \mathrm Q^N[f](x_k^n,y_k^n,\{w_k^i\}_i)-\bar{\mathrm Q} ^N[f]( \{w_k^i\}_i) \Big)\Big( \mathrm Q^N[f](x_k^m,y_k^m,\{w_k^i\}_i)-\bar{\mathrm Q} ^N[f]( \{w_k^i\}_i) \Big) \Big]\Big|_{\{w_k^i\}_i=\{W_k^i\}_i}\\
 &=\sum_{q\ge 1} \frac{1}{q^2}    \sum_{1\leq n<m\leq q}  \textbf{E} [  \mathbf 1_{|B_k|=q}]\\
 &\quad \times \textbf{E} \Big[   \mathrm Q^N[f](x_k^n,y_k^n,\{w_k^i\}_i)-\bar{\mathrm Q} ^N[f]( \{w_k^i\}_i) \Big]\Big|_{\{w_k^i\}_i=\{W_k^i\}_i} \textbf{E} \Big[   \mathrm Q^N[f](x_k^m,y_k^m,\{w_k^i\}_i)-\bar{\mathrm Q} ^N[f]( \{w_k^i\}_i) \Big]\Big|_{\{w_k^i\}_i=\{W_k^i\}_i}\\
 &=0.
 \end{align*}
 where we have used~\ref{as:data} at the two last equalities. We also have with the same arguments: 
 \begin{align*}
&\textbf{E}  \Big[\frac{1}{|B_k|^2}\sum_{n=1}^{|B_k|} \big| \mathrm Q_k^N[f](x_k^n,y_k^n)-\bar{\mathrm Q} ^N[f]( \{W_k^i\}_i) \big|^2\Big |\mathcal{F}_k^N \Big]\\
&=\sum_{q\ge 1} \frac{1}{q^2}    \sum_{n=1}^q  \textbf{E} [  \mathbf 1_{|B_k|=q}] \,  \textbf{E} \Big[ \big|  \mathrm Q^N[f](x_k^n,y_k^n,\{w_k^i\}_i)-\bar{\mathrm Q} ^N[f]( \{w_k^i\}_i) \big|^2\Big]\Big|_{\{w_k^i\}_i=\{W_k^i\}_i} \\
&=\sum_{q\ge 1} \frac{1}{q^2}    \sum_{n=1}^q  \textbf{E} [  \mathbf 1_{|B_k|=q}] \textbf{E} \Big[ \big|  \mathrm Q^N[f](x_1^1,y_1^1,\{w_k^i\}_i)-\bar{\mathrm Q} ^N[f]( \{w_k^i\}_i) \big|^2\Big]\Big|_{\{w_k^i\}_i=\{W_k^i\}_i}\\
&= \textbf{E}\left[\frac{1}{|B_k|}\right]\Var_\pi\left( \mathrm Q^N[f](x,y,\{W_k^i\}_i) \right).
 \end{align*} 
 The notation $\Cov_\pi$ means that we consider the expectation only
 w.r.t.  $(x,y)\sim \pi$ (see~\ref{as:data}) .  We then have, for
 $k\ge 0$ (see~\eqref{def Mk}),
\begin{align*}
\textbf{E}[\langle f,M_k^N\rangle^2|\mathcal{F}_k^N]&=\frac{\alpha^2}{N^2}\textbf{E} \Big[ \Big|\frac{1}{|B_k|}\sum_{(x,y)\in B_k} \big[\mathrm Q_k^N[f](x,y)-\bar{\mathrm Q} ^N[f]( \{W_k^i\}_i) \big] \Big|^2 \Big|\mathcal{F}_k^N \Big]\\
&=\frac{\alpha^2}{N^2}\textbf{E}  \Big[\frac{1}{|B_k|^2}\sum_{n=1}^{|B_k|} \big| \mathrm Q_k^N[f](x_k^n,y_k^n)-\bar{\mathrm Q} ^N[f]( \{W_k^i\}_i) \big|^2\Big |\mathcal{F}_k^N \Big]\\
&=\frac{\alpha^2}{N^2}\textbf{E}\left[\frac{1}{|B_k|}\right]\Var_\pi\left( \mathrm Q_k^N[f](x,y,\{W_k^i\}_i) ) \right).
\end{align*} 
Then, one has:   
\begin{align} 
\nonumber
N\sum_{k=0}^{\lfloor Nt\rfloor-1}\textbf{E}[\langle f,M_k^N\rangle^2|\mathcal{F}_k^N]&=\alpha^2\sum_{k=0}^{\lfloor Nt\rfloor-1}\int_{\frac{k}{N}}^{\frac{k+1}{N}}\textbf{E}\left[\frac{1}{|B_{\lfloor Ns\rfloor}|}\right]\Var_\pi\big( \mathrm Q^N[f](x,y,\{W_{\lfloor Ns\rfloor}^i\}_i) \big)\di s\\
\nonumber
&=\alpha^2\int_0^t\textbf{E}\left[\frac{1}{|B_{\lfloor Ns\rfloor}|}\right]\Var_\pi\big( \mathrm Q ^N[f](x,y,\{W_{\lfloor Ns\rfloor}^i\}_i) \big)\di s\\
\label{79}
&\quad -\alpha^2\int_{\frac{\lfloor Nt\rfloor}{N}}^t\textbf{E}\left[\frac{1}{|B_{\lfloor Ns\rfloor}|}\right]\Var_\pi\big( \mathrm Q ^N[f](x,y,\{W_{\lfloor Ns\rfloor}^i\}_i)\big)\di s.
\end{align} 
Using~\ref{as:batch limite}, a dominated convergence
theorem, and the same arguments as those used in the proof of
Lemma~\ref{le.Fc},  we prove that if $m^N\to m$ in  $ \mathcal
D(\mathbf{R}_+,\mathcal P_{\gamma}(\mathbf{R}^d))$, we have for all
$f\in\mathcal C^{2,\gamma}(\mathbf R^d)$ and $t\in \mathbf R_+$, as $N\to +\infty$,
\begin{align*}
&\int_0^t\int_{\mathcal{X}\times\mathcal{Y}}\textbf{E}\left[\frac{1}{|B_{\lfloor Ns\rfloor}|}\right] \Big[(y-\langle \sigma_*(\cdot, x),m_s^N\rangle)\langle\nabla f\cdot\nabla\sigma_*(\cdot,x),m_s^N\rangle  \\
&\quad-\int_{\mathcal{X}\times\mathcal{Y}}(y'-\langle \sigma_*(\cdot, x'),m_s^N\rangle)\langle\nabla f\cdot\nabla\sigma_*(\cdot,x'),m_s^N\rangle\pi(\di x',\di y')\Big]^2\pi(\di x,\di y)\di s \\
&\to \int_0^t\int_{\mathcal{X}\times\mathcal{Y}}\textbf{E}\left[\frac{1}{|B_{\infty}|}\right]\Big[(y-\langle \sigma_*(\cdot, x),m_s\rangle)\langle\nabla f\cdot\nabla\sigma_*(\cdot,x), m_s\rangle\\
&\quad-\int_{\mathcal{X}\times\mathcal{Y}}(y'-\langle \sigma_*(\cdot , x'),  m_s\rangle)\langle\nabla f\cdot\nabla\sigma_*(\cdot,x'),m_s\rangle\pi(\di x',\di y')\Big]^2\pi(\di x,\di y) \di s.
\end{align*}
 Recall that by Theorem~\ref{thm:lln}, $\mu^N\overset{p}{\to}\bar \mu$ in  $ \mathcal
D(\mathbf{R}_+,\mathcal P_{\gamma}(\mathbf{R}^d))$. 
mapping theorem, we have for all $t\in \mathbf R_+$ and  
$f\in\mathcal C^{2,\gamma}(\mathbf R^d)$:
\begin{align*}
\alpha^2\int_0^t\textbf{E}\left[\frac{1}{|B_{\lfloor Ns\rfloor}|}\right]\Var_\pi( \mathrm Q ^N[f](x,y,\{W_{\lfloor Ns\rfloor}^i\}_i ))\di s\ \overset{p}{\to}\ \mathfrak c_t[f]:=\alpha^2\int_0^t\textbf{E}\left[\frac{1}{|B_{\infty}|}\right]\Var (\mathrm Q_s[f](x,y))\di s.
\end{align*}  
Note that $t\in \mathbf R_+\mapsto \mathfrak c_t[f]$ is locally
Lipschitz continuous  since for all $s\in [0,t]$,
$\Var (\mathrm Q_s[f](x,y))\le C \int_{\mathcal X\times Y}
(|y|^2+1)\pi(\di x,\di y) \Vert f\Vert^2 _{\mathcal C^{1,\gamma}}
\sup_{s\in [0,t]} |\langle (1+|\cdot |^{ \gamma}), \bar \mu_s\rangle|^2$.
Let us now consider the second term in the right-hand side
of~\eqref{79}. Using \eqref{ineg <sigma, nu_k>} and Lemma~\ref{le.W},
  $\mathbf{E}[|\mathrm Q^N[f](x,y, \{W_k^i\}_i) |^2]\leq C\|f\|_{\mathcal C^{2,\gamma}}^2$. Consequently, it holds:
$$\textbf{E}\Big[\Big|\alpha^2\int_{\frac{\lfloor Nt\rfloor}{N}}^t\textbf{E}\Big[\frac{1}{|B_{\lfloor Ns\rfloor}|}\Big]\Var_\pi\big( \mathrm Q ^N[f](x,y,\{W_{\lfloor Ns\rfloor}^i\}_i)\big)\di s\Big|\Big]\xrightarrow[N\to\infty]{} 0.$$
We have thus shown   that 
\begin{align*}
N\sum_{k=0}^{\lfloor Nt\rfloor-1}\textbf{E}[\langle f,M_k^N\rangle^2|\mathcal{F}_k^N]\xrightarrow[N\to\infty]{p} \alpha^2\int_0^t\textbf{E}\left[\frac{1}{|B_{\infty}|}\right]\Var (\mathrm Q_s[f](x,y))\di s.
\end{align*}
At this point, the study of the first term in the right-hand side
of~\eqref{q variation} is complete. It remains to study the second
term in the right-hand side of~\eqref{q
  variation}. Using~\eqref{2.40}, we obtain:
\begin{align*}
N^2\textbf{E}\Big[\Big|\sum_{k=0}^{\lfloor Nt\rfloor -1}\langle f,M_k^N\rangle^2-\textbf{E}[\langle f,M_k^N\rangle^2|\mathcal{F}_k^N]\Big|^2\Big]= N^2\sum_{k=0}^{\lfloor Nt\rfloor -1}\textbf{E}\Big[\Big|\langle f,M_k^N\rangle^2-\textbf{E}[\langle f,M_k^N\rangle^2|\mathcal{F}_k^N]\Big|^2\Big]\\
\leq CN^2\sum_{k=0}^{\lfloor Nt\rfloor -1}\textbf{E}[\langle f,M_k^N\rangle^4]\leq CN^2\|f\|^4_{\mathcal C^{2,\gamma_*}}/N^3 \to 0.
\end{align*}
In conclusion, we have proved that for every $t\in\mathbf{R}_+$,
  \begin{equation}\label{eq.convAt}
 \text{$\mathfrak a^N_t[f] \overset{p}{\to} \mathfrak c_t[f]$ as $N\to +\infty$. }
   \end{equation}
By \cite[Theorem 7.1.4]{ethier2009markov}, the proof of    Proposition~\ref{prop:convergence to martingale0} is complete. 
 
\end{proof}


\begin{prop}\label{prop:convergence to martingale}
  Let $\beta>3/4$ and assume~\ref{as:batch}-\ref{as:batch limite}.
  Consider a  family $\mathscr F=\{f_a\}_{a\ge1}$ of elements of
  $\mathcal C^{2,\gamma}(\mathbf{R}^d)$.  Then, for $k\ge 1$, the
  sequence
$$\lbrace t\mapsto (\sqrt N\langle f_1,M_t^N\rangle,\dots,\sqrt N\langle f_{k},M_t^N\rangle)^T,t\in\mathbf R_+\}_{N\ge1}$$
converges in distribution in $\mathcal D(\mathbf R_+,\mathbf R^{k})$
towards a process
$Y_k^{\mathscr F}=\{t\mapsto ({Y}_t^1,\dots,{Y}_t^k)^T
,t\in\mathbf{R}_+\}\in \mathcal C(\mathbf{R}_+,\mathbf{R}^{k})$ with
zero-mean and independent Gaussian increments (which is thus a
martingale). In addition, for all $0\le s\le t$,
\begin{equation}\label{eq.cov}
\Cov({Y}_t^i,{Y}_s^j)=\alpha^2\mathbf{E}\left[\frac{1}{|B_{\infty}|}\right] \int_0^s \Cov (\mathrm Q_v[f_i](x,y),\mathrm Q_v[f_j ](x,y))\di v,\quad 1\leq i,j\leq {k}.
\end{equation}
\end{prop}

Notice that~\eqref{eq.cov} is exactly~\eqref{eq.cov2}. 
\begin{proof}
  Set for ease of notation,
  $ \mathscr M^N_t=(\sqrt N\langle f_1,M_t^N\rangle,\dots,\sqrt
  N\langle f_{k},M_t^N\rangle)^T$, $t\in \mathbf R_+$. We have
  $ \mathscr M^N_t = \sum_{q=0}^{\lfloor Nt\rfloor -1} \xi_q^N$, where
  $ \xi_q^N =(\sqrt N\langle f_1,M_q^N\rangle,\dots,\sqrt N\langle
  f_{k},M_q^N\rangle)^T$ (see indeed~\eqref{eq.def-M}).  From
  \cite[Remark 7.1.5]{ethier2009markov}, the covariation matrix of
  $ \mathscr M^N_t $ is
$$  \mathscr A^N_t [f_1,\ldots,f_k]:=N\sum_{q=0}^{\lfloor Nt\rfloor -1} \xi_q^N(\xi_q^N)^T =N\sum_{q=0}^{\lfloor Nt\rfloor -1} (\langle f_i,M_q^N\rangle\langle f_j,M_q^N\rangle)_{i,j=1,\ldots,k} .$$ 
If $t\ge s$, we have
$ \mathscr A^N_t[f_1,\ldots,f_k]- \mathscr
A^N_s[f_1,\ldots,f_k]=N\sum_{q=\lfloor Ns\rfloor}^{\lfloor Nt\rfloor
  -1} \xi_q^T\xi_q \ge 0$.  By~\eqref{1.14 EK}, Condition (a) in
\cite[Theorem 7.1.4]{ethier2009markov} is satisfied for
$\mathscr M^N$.  Secondly, condition (1.19) in \cite[Theorem
7.1.4]{ethier2009markov} is satisfied, using the decomposition
\begin{align*}
\mathscr  A^N_t[f_1,\ldots,f_k]_{i,j}=  N\sum_{q=0}^{\lfloor Nt\rfloor -1} \langle f_i,M_q^N\rangle\langle f_j,M_q^N\rangle&=N\sum_{q=0}^{\lfloor Nt\rfloor -1}\frac{1}{2}(\langle f_i+f_j,M_q^N\rangle^2-\langle f_i,M_q^N\rangle^2-\langle f_j,M_q^N\rangle^2)\\
&=\frac 12 \big(\mathfrak a^N_t[f_i+f_j]-\mathfrak a^N_t[f_i]-\mathfrak a^N_t[f_j]\big)\\
&\overset{p}{\to} \frac 12 \big(\mathfrak c_t[f_i+f_j]-\mathfrak c_t[f_i]-\mathfrak c_t[f_j]\big)  \text{ (by~\eqref{eq.convAt})}\\
&=\alpha^2\int_0^t\mathbf{E}\left[\frac{1}{|B_{\infty}|}\right]\Cov (\mathrm Q_v[f_i](x,y),\mathrm Q_v[f_j ](x,y))\di v:= (\mathfrak C_t)_{i,j}.
\end{align*}
It remains to check that $\mathfrak C$ satisfies the assumptions of
\cite[Theorem 7.1.1]{ethier2009markov}.  Clearly $\mathfrak C(0)=0$
and $t\in \mathbf R_+\mapsto \mathfrak C_t$ is continuous. In
addition, if $0\le s\le t$,
\begin{align*}
\mathfrak C_t-\mathfrak C_s&=\alpha^2\Big(\int_s^t\mathbf{E}\left[\frac{1}{|B_{\infty}|}\right]\Cov (\mathrm Q_v[f_i](x,y),\mathrm Q_v[f_j ](x,y))\di v\Big)_{i,j=1,\ldots,k}\\
&=\alpha^2 \int_s^t\mathbf{E}\left[\frac{1}{|B_{\infty}|}\right]\mathbf E [\Xi_v (x,y) \Xi_v^T(x,y)]\di v,
\end{align*}
where
$\Xi_v(x,y)_i=\mathrm Q_v[f_i ](x,y) - \mathbf E [\mathrm Q_v[f_i
](x,y)]$, $i\in \{1,\ldots,k\}$. Thus, $\mathfrak C_t-\mathfrak C_s$
is symmetric and non negative definite.  The proof of
Proposition~\ref{prop:convergence to martingale} is complete.
\end{proof}


\begin{prop}\label{prop.CVMN}
\begin{sloppypar}
  Let $\beta>3/4$ and assume that~\ref{as:batch}-\ref{as:batch limite}
  hold. Then, the sequence
  $(\sqrt{N}M^N)_{N\ge1}$ converges in
  distribution in
  $\mathcal D(\mathbf R_+,\mathcal H^{-J_0,j_0}(\mathbf R^d))$ to a
  {\rm G}-process
  $\mathscr G\in \mathcal
  C(\mathbf R_+,\mathcal H^{-J_0,j_0}(\mathbf{R}^d))$ (see
  Definition~\ref{de.gaussian}).
\end{sloppypar}
\end{prop}

\noindent
To prove Proposition~\ref{prop.CVMN}, we will prove that there is a
unique limit point of the sequence
$(\sqrt{N}M^N)_{N\ge1}$ in
$\mathcal D(\mathbf R_+,\mathcal H^{-J_0,j_0}(\mathbf R^d))$ (recall
that this sequence is relatively compact in this space, see
Proposition~\ref{prop relative compactness M}), so that the whole
sequence converges in distribution.  Proposition~\ref{prop:convergence
  to martingale} will then imply that this unique limit point is a
G-process. Before, we need to introduce some definitions.  For a
family $\mathscr F=\{f_a\}_{a\ge1}$ of elements of
$\mathcal H^{J_0,j_0}(\mathbf{R}^d)$, we define, for $k\ge1$, the
projection
$$
\pi_k^{\mathscr F}:\mathcal D(\mathbf{R}_+,\mathcal H^{-J_0,j_0}(\mathbf{R}^d))\to \mathcal D(\mathbf{R}_+,\mathbf{R})^k,\quad m\mapsto (\langle f_1,m\rangle,\dots, \langle f_k,m\rangle)^T  .
$$
The function $\pi_k^{\mathscr F}$ is continuous.  In the following,
$\mathscr H=\{h_a\}_{a\ge1}$ is an orthonormal basis of
$\mathcal H^{J_0,j_0}(\mathbf{R}^d)$.  Let $d_R$ be a metric for the
Skorohod topology on $\mathcal D(\mathbf{R}_+,\mathbf{R})$.  Introduce
the space $\mathcal D(\mathbf{R}_+,\mathbf{R})^\infty$ defined as the
set of sequences taking values in
$\mathcal D(\mathbf{R}_+,\mathbf{R})$. We endow
$\mathcal D(\mathbf{R}_+,\mathbf{R})^\infty$ with the metric
$\rho(u,v)=\sum_{a\ge1}2^{-a}\min(1,d_R(u_a,v_a))$. We consider on
$\mathcal D(\mathbf{R}_+,\mathbf{R})^\infty$ the topology associated
with $\rho$. We have that $\rho(u^N,u)\to 0$ if and only if
$d_R(u_a^N,u_a)\to 0$ for all $a\ge1$. Notice with that with this
metric $\rho$, $\mathcal D(\mathbf{R}_+,\mathbf{R})^\infty$ is
separable, since $\mathcal D(\mathbf{R}_+,\mathbf{R})$ is
separable~\cite[Theorem 3.5.6]{ethier2009markov}.  We now define the
map
$$\Pi:  \mathcal D(\mathbf{R}_+,\mathcal H^{-J_0,j_0}(\mathbf{R}^d))\to  \mathcal D(\mathbf{R}_+,\mathbf{R})^\infty, \quad m\mapsto(\langle h_a,m\rangle) ^T_{a\ge 1}.$$
This map is injective (because ${\mathscr H}$ is a basis of
$\mathcal H^{J_0,j_0}(\mathbf{R}^d)$) and continuous.  The map $\Pi$
depends on the orthonormal basis ${\mathscr H}$ but, for ease of
notation, we have omitted to write it.
Finally, we introduce the continuous function
$$p_k: \mathcal D(\mathbf{R}_+,\mathbf{R})^\infty\to  \mathcal D(\mathbf{R}_+,\mathbf{R})^k, \quad (m_a) ^T_{a\ge 1}\mapsto (m_1,\dots,m_k)^T.$$
It holds 
$$\pi_k^{\mathscr H}=p_k\circ\Pi.$$ 
We now introduce the set
$$\mathcal{C}:= \big\{p_k^{-1}(H),\quad H\in \mathcal{B}( \mathcal D(\mathbf{R}_+,\mathbf{R})^k),\ k\ge1 \big\}\subset  \mathcal D(\mathbf{R}_+,\mathbf{R})^\infty,$$
where $\mathcal{B}( \mathcal D(\mathbf{R}_+,\mathbf{R})^k)$ denotes
the Borel $\boldsymbol \sigma$-algebra of
$ \mathcal D(\mathbf{R}_+,\mathbf{R})^k$.  The continuity of $p_k$
implies that
$\mathcal{C}\subset\mathcal{B}(\mathcal
D(\mathbf{R}_+,\mathbf{R})^\infty)$. The following result shows that
$\mathcal{C}$ is a separating class of
$(\mathcal D(\mathbf{R}_+,\mathbf{R})^\infty,\mathcal{B}(\mathcal
D(\mathbf{R}_+,\mathbf{R})^\infty))$, where we recall that this means
by definition that two probability measures on
$\mathcal{B}(\mathcal D(\mathbf{R}_+,\mathbf{R})^\infty)$ which agree
on $\mathcal C$ necessarily agree on
$\mathcal{B}(\mathcal D(\mathbf{R}_+,\mathbf{R})^\infty)$.

 
\begin{lem}\label{lem separating class}
  The set $\mathcal{C}$ is a separating class of
  $(\mathcal D(\mathbf{R}_+,\mathbf{R})^\infty,\mathcal{B}(\mathcal
  D(\mathbf{R}_+,\mathbf{R})^\infty))$.
\end{lem}
\begin{proof} We recall that any subset of
  $\mathcal{B}( \mathcal D(\mathbf{R}_+,\mathbf{R})^\infty)$ which is
  a $\pi$-system (i.e. closed under finite intersection) and which
  generates the $\boldsymbol \sigma$-algebra
  $\mathcal{B}(\mathcal D(\mathbf{R}_+,\mathbf{R})^\infty)$ is a
  separating class (see \cite[Page 9]{billingsley2013convergence}).
  Let us first prove that $\mathcal{C}$ is a $\pi$-system.  Notice
  that it holds
  $p_k^{-1}(H)=p_{k+1}^{-1}(H\times \mathcal
  D(\mathbf{R}_+,\mathbf{R}))$.  Thus, if $A$ and $A'\in \mathcal{C}$
  (write $A=p_k^{-1}(H)$ and $A'=p_{k'}^{-1}(H')$, and assume that
  $k'\ge k$), then
  $A\cap A'=p_{k'}^{-1}((H\times \mathcal
  D(\mathbf{R}_+,\mathbf{R})\ldots \times \mathcal
  D(\mathbf{R}_+,\mathbf{R}))\cap H')\in \mathcal C$. Consequently
  $\mathcal{C}$ is a $\pi$-system.  It remains to show that the
  $\boldsymbol \sigma$-algebra generated by $\mathcal{C}$ is equal to
  $\mathcal{B}( \mathcal D(\mathbf{R}_+,\mathbf{R})^\infty)$.  To
  prove it, it is sufficient to prove that any open set of
  $\mathcal D(\mathbf{R}_+,\mathbf{R})^\infty$ is a countable union of
  sets in $\mathcal C$.  Introduce, for
  $x\in \mathcal D(\mathbf{R}_+,\mathbf{R})^\infty$, $k\ge1$ and $\epsilon>0$,
\begin{equation}\label{eq. N_ke}
\mathcal N_{k,\epsilon}(x):=\{y\in\mathcal  D(\mathbf{R}_+,\mathbf{R})^\infty; \ d_R(x_a,y_a)<\epsilon,\ 1\leq a\leq k\}  \subset  \mathcal C.
\end{equation}
By straightforward arguments $\mathcal N_{k,\epsilon}(x)$ is open in
$(\mathcal D(\mathbf{R}_+,\mathbf{R})^\infty,\rho)$.  Remark that for
all $y\in \mathcal N_{k,\epsilon}(x)$, it holds
$\rho(x,y)<\epsilon+2^{-k}$. Given $r>0$, choose $\epsilon>0$ and
$k\ge1$ such that $\epsilon+2^{-k}<r$. Then,
$\mathcal N_{k,\epsilon}(x)\subset B_\rho(x,r)$ where $B_\rho(x,r)$ is
the open ball of center $x$ and radius $r$ for the metric $\rho$ of
$\mathcal D(\mathbf{R}_+,\mathbf{R})^\infty$. The space
$\mathcal D(\mathbf{R}_+,\mathbf{R})^\infty$ is separable and we
consider $\mathrm D$ a dense and countable subset of
$(\mathcal D(\mathbf{R}_+,\mathbf{R})^\infty,\rho)$. Let $\mathcal O$
be an open subset of
$(\mathcal D(\mathbf{R}_+,\mathbf{R})^\infty,\rho)$.  We claim that
$$\mathcal O= \mathcal N_{\mathcal O} \text{ where } \mathcal N_{\mathcal O}:=\bigcup_{\substack{x\in \mathrm  D\cap {\mathcal O}, \ k\ge 1\\ \epsilon \in \mathbf Q\cap \mathbf R_+^*,\ \mathcal N_{k,\epsilon}(x)\subset \mathcal O}} \mathcal N_{k,\epsilon}(x).$$
We have $\mathcal N_{\mathcal O}\subset {\mathcal O}$. Let us now show
that ${\mathcal O}\subset \mathcal N_{\mathcal O}$.  To this end, pick
$y\in {\mathcal O}$ and $r_0>0$ such that
$B_\rho(y,r_0)\subset {\mathcal O}$.  Choose $k_0\ge 1$ such that
$$2^{-k_0}< \frac{r_0}{4}.$$
Consider, for $n\ge 1$,  $x^n\in 
\mathrm  D$ such that $\rho(y,x^n)<1/n$. 
Choose  $n\ge 1$ such that  
$$\frac 1n  + \frac{r_0}{2} <r_0 \text{ and } \frac {2^{k_0}}n < \frac{r_0}{4}.$$
Since for all $a\ge 1$, $d_R(y_a,x^n_a)\to 0$ as $n\to +\infty$ (by
definition of $\rho$), we choose if necessary $n\ge 1$ larger so that
$\min(1,d_R(y_a,x^n_a))=d_R(y_a,x^n_a)$ for all $a=1,\ldots,k_0$.
Finally, choose $\epsilon_{k_0,n} = {2^{k_0}}/n\in \mathbf Q$. We have
for all $a=1,\ldots,k_0$,
$d_R(x^a_n,y_a)\le \rho(x^n,y) 2^a\le\rho(x^n,y)
2^{k_0}<2^{k_0}/n=\epsilon_{k_0,n}$, so that
$y\in \mathcal N_{k_0,\epsilon_{k_0,n}}(x^n)$. It just remains to
check that $\mathcal N_{k_0,\epsilon_{k_0,n}}(x^n)\subset \mathcal O$
to ensure that
$\mathcal N_{k_0,\epsilon_{k_0,n}}(x^n)\subset \mathcal N_{\mathcal
  O}$. We have
$\rho(y,x^n)<\epsilon_{k_0,n}+2^{-k_0}<r_0/4+r_0/4=r_0/2$ and thus
$\mathcal N_{k_0,\epsilon_{k_0,n}}(x^n)\subset B_\rho(x^n,
r_0/2)$. Since $1/n + r_0/2 <r_0$, by triangular inequality,
$B_\rho(x^n, r_0/2)\subset B_\rho(y, r_0)\subset \mathcal O$. Thus,
$\mathcal N_{k_0,\epsilon_{k_0,n}}(x^n)\subset \mathcal O$, which
proves that
$\mathcal N_{k_0,\epsilon_{k_0,n}}(x^n)\subset \mathcal N_{\mathcal
  O}$.  Consequently, we have proved that
$y\in \mathcal N_{k_0,\epsilon_{k_0,n}}(x^n) \subset \mathcal
N_{\mathcal O}$, and then that
$\mathcal O\subset \mathcal N_{\mathcal O}$. Thus,
$\mathcal O=\mathcal N_{\mathcal O}$.  In conclusion, every open set
$\mathcal O$ is a countable union of sets of the
form~\eqref{eq. N_ke}. This implies that
$\sigma(\mathcal{C})=\mathcal{B}( \mathcal
D(\mathbf{R}_+,\mathbf{R})^\infty)$ and therefore $\mathcal{C}$ is a
separating class.
\end{proof}

\begin{prop}\label{prop P Q coincide}
  Let $P,Q$ be two probability measures on
  $\mathcal D(\mathbf{R}_+,\mathcal H^{-J_0,j_0}(\mathbf R^d))$ such
  that $\pi_k^{\mathscr H}P=\pi_k^{\mathscr H}Q$ for all $k\ge
  1$. Then, $P=Q$.
\end{prop}
\begin{proof}
  The equality $\pi_k^{\mathscr H}P=\pi_k^{\mathscr H}Q$ for all
  $k\ge 1$, writes $p_k\circ\Pi P=p_k\circ \Pi Q$. By Lemma~\ref{lem
    separating class},
$\Pi P=\Pi Q$. Since $\Pi$ is injective {it admits a left inverse
  $\Pi^{-1}$}, and therefore $P = Q $. The proof is complete.
\end{proof}

We are now in position to prove  Proposition~\ref{prop.CVMN}.

\begin{proof}[Proof of Proposition~\ref{prop.CVMN}]
  By Proposition~\ref{prop relative compactness M},
  $(\sqrt{N}M^{N})_{N\geq1}$ is relatively compact in
  $\mathcal D(\mathbf{R}_+,\mathcal H^{-J_0,j_0}(\mathbf{R}^d))$. Let
  $\mathscr M^*$ be one of its limit point. Let us show that
  $\mathscr M^*$ is independent of the extracted subsequence, say of
  $N'$. Since $\pi_k^{\mathscr H}$ is continuous for all $k\ge1$, the
  continuous mapping theorem implies that in
  $\mathcal D(\mathbf{R}_+,\mathbf{R})^k$,
$$\pi_k^{\mathscr H}(\sqrt{N'}M^{N'})\to \pi_k^{\mathscr H}(\mathscr M^*)\text{ in distribution, as }   N'\to +\infty.$$
For $k\ge 1$, introduce the following continuous and bijective mapping
$\mathscr Q_k:\mathcal D(\mathbf{R}_+,\mathbf{R}^k)\to \mathcal
D(\mathbf{R}_+,\mathbf{R})^k$ defined by:
$m\mapsto (m\cdot e_1, \ldots, m\cdot e_k)^T$, where
$\{e_1, \ldots, e_k\}$ denotes the canonical basis of $\mathbf R^k$.
By Proposition~\ref{prop:convergence to martingale} applied with
$f_a=h_a$, $a\in \{1,\ldots,k\}$ (recall $\mathcal H^{J_0,j_0}(\mathbf R^d)\subset \mathcal C^{2,\gamma}(\mathbf R^d)$), 
 it holds in $\mathcal D(\mathbf{R}_+,\mathbf{R}^k)$,
 $$\forall k\ge 1, \ \mathscr Q_k^{-1}\circ \pi_k^{\mathscr H}(\sqrt{N'}M^{N'})\to  Y_k^{\mathscr H} \text{ in distribution, as }   N'\to +\infty.$$
 Since $\mathscr Q_k$ is continuous, one then has in $\mathcal D(\mathbf{R}_+,\mathbf{R})^k$, 
 $$\forall k\ge 1, \   \pi_k^{\mathscr H}(\sqrt{N'}M^{N'})\to  \mathscr Q_k(Y_k^{\mathscr H}) \text{ in distribution, as }   N'\to +\infty.$$
 It follows that
 $\pi_k^{\mathscr H}(\mathscr M^*)=\mathscr Q_k(Y_k^{\mathscr H})$ in
 distribution. By Proposition~\ref{prop P Q coincide}, the
 distribution of $\mathscr M^*$ is fully determined by the collection
 of distributions of the processes $\pi_k^{\mathscr H}(\mathscr M^*)$,
 for $k\ge 1$. Thus, $\mathscr M^*$ is independent of the subsequence,
 and therefore the whole sequence $(\sqrt N M^N)_{N\geq1}$
 convergences to $\mathscr M^*$ in
 $\mathcal D(\mathbf{R}_+,\mathcal H^{-J_0,j_0}(\mathbf{R}^d))$. By
 Lemma~\ref{le.conVMM},
 $\mathscr M^*\in \mathcal C(\mathbf R_+,\mathcal
 H^{-J_0,j_0}(\mathbf{R}^d))$. Let us now consider a family
 $\mathscr F=\{f_a\}_{a\ge1}$ of elements of
 $\mathcal H^{J_0,j_0}(\mathbf{R}^d)$.  Since $\mathscr D_k$ and
 $\pi_k^{\mathscr F}$ are continuous, and by
 Proposition~\ref{prop:convergence to martingale}, one has that
 $\mathscr Q_k^{-1}\circ \pi_k^{\mathscr F}(\mathscr M^*)=
 Y_k^{\mathscr F}\in \mathcal C(\mathbf{R}_+,\mathbf{R}^{k})$ in
 distribution.  The proof of Proposition~\ref{prop.CVMN} is complete.
\end{proof}


\subsection{Limit points of $(\eta^N)_{N\ge 1}$ and end of the proof
  of Theorem~\ref{thm:clt}}
 \label{sec.conTCL}

 \subsubsection{On the limit points of the sequence
   $(\eta^N,\sqrt NM^N)_{N\ge 1}$}


 \begin{lem}\label{le.eta0}
   Assume~\ref{as:batch}-\ref{as:batch limite}.  Then, the sequence
   $(\eta^N_0 )_{N\ge 1}$ converges in distribution in
   $\mathcal H^{-J_0+1,j_0}(\mathbf R^d)$ towards a variable $\nu_0$
   which is the unique (in distribution)
   $\mathcal H^{-J_0+1,j_0}(\mathbf R^d)$-valued random variable such
   that for all $k\ge 1$ and
   $f_1\dots,f_{k}\in \mathcal H^{J_0-1,j_0}(\mathbf R^d)$,
   $(\langle f_{1}, \nu_0\rangle,\dots, \langle f_{k}, \nu_0\rangle)^T
   \sim \mathcal N (0, \Gamma(f_{1},\ldots,f_{k}))$, where
   $\Gamma(f_{1},\ldots,f_{k})$ is the covariance matrix of the vector
   $( f_{1}(W_0^1),\dots, f_{k}(W_0^1))^T$.
 \end{lem}
 \begin{proof} 
 The sequence $(\eta^N_0 )_{N\ge 1}$ is tight in $\mathcal H^{-J_0+1,j_0}(\mathbf R^d)$.
 Let $\mathscr F=\{f_a\}_{a\ge 1}$ be a family of elements of
 $\mathcal H^{J_0-1,j_0}(\mathbf R^d)$.  Define, for $k\ge1$, the
 projection
$$
\mathscr P_k^{\mathscr F}:\mathcal H^{-J_0+1,j_0}(\mathbf{R}^d)\to
\mathbf{R}^k,\quad m\mapsto (\langle f_1,m\rangle,\dots, \langle
f_k,m\rangle).
$$
The map $\mathscr P_k^{\mathscr F}$ is continuous.  By the standard
vectorial central limit theorem, for $k\ge 1$,
$\mathscr P_k^{\mathscr F}(\eta^N_0)\to \mathcal N (0,
\Gamma(f_{1},\ldots,f_{k}))$ in distribution. In addition, we show
with the same arguments as those used to prove Lemma~\ref{lem
  separating class} and Proposition~\ref{prop P Q coincide}, that when
$\mathscr F$ is an orthonormal basis of
$\mathcal H^{J_0-1,j_0}(\mathbf R^d)$, the distribution of a
$\mathcal H^{-J_0+1,j_0}(\mathbf R^d)$-valued random variable $\nu$ is
fully determined by the collection of the distributions
$\{\mathscr P_k^{\mathscr F}(\nu),k\ge 1\}$. Hence,
$(\eta^N_0 )_{N\ge 1}$ has a unique limit point $\nu_0$ in
distribution which is the unique (in distribution)
$\mathcal H^{-J_0+1,j_0}(\mathbf R^d)$-valued random variable such
that for all $k\ge 1$,
$\mathscr P_k^{\mathscr F}(\nu_0)\sim \mathcal N (0,
\Gamma(f_{1},\ldots,f_{k}))$. In particular, the whole sequence
$(\eta^N_0 )_{N\ge 1}$ converges in distribution towards $\nu_0$.  The
proof of the lemma is complete.
   \end{proof}
\begin{sloppypar} 
\noindent
Set 
\begin{equation}\label{eq.LE-}
\mathscr E:=  \mathcal D(\mathbf R_+,\mathcal H^{-J_0+1,j_0}(\mathbf R^d))\times \mathcal D(\mathbf R_+,\mathcal H^{-J_0,j_0}(\mathbf R^d)).
\end{equation}
According to Propositions~\ref{prop relative compactness}
and~\ref{prop relative compactness M},
$( \eta^N,\sqrt N\, M^N)_{N\ge1}$ is tight in $\mathscr E$.  Let
$(\eta^*,\mathscr G^*) $ be one of its limit point in $\mathscr
E$. Along some subsequence, it holds:
$$
(\eta^{N'},\sqrt{N'}\, M^{N'})\to(\eta^*,\mathscr G^*), \text{ as } N'\to +\infty. 
$$
Considering the marginal distributions, and according to
Lemmata~\ref{lem: continuity prop eta} and~\ref{le.conVMM}, it holds
a.s.
\begin{equation}\label{eq.Ceta-}
  \eta^* \in \mathcal C(\mathbf{R}_+,\mathcal H^{-J_0+1,j_0}(\mathbf{R}^d)) \text{ and }   \mathscr G^* \in \mathcal C(\mathbf{R}_+,\mathcal H^{-J_0,j_0}(\mathbf{R}^d)).
\end{equation}
By uniqueness of the limit in distribution, using Lemma~\ref{le.eta0}
(together with the fact that the projection
$m\in \mathcal D(\mathbf R_+, \mathcal H^{-J_0+1,j_0}(\mathbf{R}^d))
\mapsto m_0\in \mathcal H^{-J_0+1,j_0}(\mathbf{R}^d)$ is continuous)
and Proposition~\ref{prop.CVMN}, it also holds:
\begin{equation}\label{eq.eta*0}
 \eta_0^*=\nu_0 \text{ and } \mathscr G^*=\mathscr{G},  \text{ in  distribution}.
\end{equation}
\end{sloppypar}

\begin{prop}\label{pr.LM-TCL}
  Let $\beta>3/4$ and assume~\ref{as:batch}-\ref{as:batch limite}.
  Then, $ \eta^*$ is a weak solution of~\eqref{eq.CLT} (see
  Definition~\ref{de.weak}) with initial distribution $\nu_0$ (see
  Lemma~\ref{le.eta0}).
\end{prop}

\begin{proof}
Let us introduce, for $\Phi\in \mathcal H^{-J_0+1,j_0}(\mathbf{R}^d)$, $f\in \mathcal H^{J_0,j_0}(\mathbf{R}^d)$, and $s\ge 0$:
 \begin{align}\label{eq.Gf} 
  \mathbf{U}_s[f](\Phi):=\int_{\mathcal{X}\times\mathcal{Y}}\alpha(y-\langle\sigma_*(\cdot,x),\bar \mu_s\rangle)\langle\nabla f\cdot\nabla\sigma_*(\cdot,x),\Phi\rangle\pi(\di x,\di y)
\end{align}
and \begin{align}\label{eq.Hf}
  \mathbf  V_s[f](\Phi):=\int_{\mathcal{X}\times\mathcal{Y}} \langle\sigma_*(\cdot,x),\Phi\rangle\langle\nabla f\cdot\nabla\sigma_*(\cdot,x),\bar \mu_s\rangle\pi(\di x,\di y).
\end{align}
 Recall $
\eta_t^N=\sqrt{N}(\mu_t^N -\bar \mu_t)$ (see~\eqref{eq.etaN}). 
Using~\eqref{eq: prelim mu} and Corollary~\ref{co.e},   it holds:
\begin{align}\label{eq.CVL3}
\langle f,\eta_t^N\rangle-\langle f,\eta_0^N\rangle-\int_0^t(   \mathbf U_s[f](\eta^N_s)-\mathbf V_s[f](\eta_s^N))\di s -\langle f,\sqrt{N}M^N_t\rangle=-\mathbf e^{N}_t[f],
\end{align} 
where 
\begin{align*}
\nonumber
\mathbf e^{N}_t [f]:&=\frac{1}{\sqrt{N}}\int_0^t\int_{\mathcal{X}\times\mathcal{Y}}\alpha\langle\sigma_*(\cdot,x),\eta_s^N\rangle\langle\nabla f\cdot\nabla\sigma_*(\cdot,x),\eta_s^N\rangle\pi(\di x,\di y)\di s\\
\label{eq.mathE}
&\quad -\sqrt{N}\langle f,V_t^N\rangle-\sqrt{N}\sum_{k=0}^{\lfloor Nt\rfloor-1}\langle f,R_k^{N}\rangle-\frac{\sqrt{N}}{N^{1+\beta}}\sum_{k=0}^{\lfloor Nt\rfloor-1}\sum_{i=1}^N\nabla f(W_k^i)\cdot\ve_k^i.
\end{align*}
In what follows, 
$f\in~\mathcal H^{J_0,j_0}(\mathbf{R}^d)$ and $t\in \mathbf R_+$ are
fixed.  \medskip
 
 \noindent
 \textbf{Step 1.}  In this step we study the continuity of the mapping
  $$\mathbf B_t[f]: m\in    \mathcal D(\mathbf{R}_+,\mathcal H^{-J_0+1,j_0}(\mathbf{R}^d))  \mapsto \langle f,m_t\rangle -\int_0^t( \mathbf{U}_s[f](m_s)- \mathbf V_s[f](m_s))\di s\in \mathbf R.$$  
  Let $(m^N)_{N\ge 1}$ such that $m^N\to m$  in
  $\mathcal D(\mathbf{R}_+,\mathcal H^{-J_0+1,j_0}(\mathbf{R}^d))$.  Recall that
  $\sup_{x\in \mathcal X}\Vert \sigma_*(\cdot,x)\Vert_{\mathcal
    H^{J_0-1,j_0}} <+\infty $ (by~\ref{as:sigma} and because $j_0>d/2$). Then, for all $N\ge 1$ and $s\in [0,t]$, it holds:
\begin{align*}
 \big | (y-\langle\sigma_*(\cdot,x),\bar \mu_s\rangle)\langle\nabla f\cdot\nabla\sigma_*(\cdot,x),m_s^N\rangle \big | &\le C( |y| +1) \Vert f\Vert_{\mathcal H^{J_0,j_0} }  \sup_{N\ge 1}\sup_{s\in[0,t]}  \Vert   m_s^N\Vert_{\mathcal H^{-J_0+1,j_0}}<+\infty  
 \end{align*} 
 and since $f\in  \mathcal H^{J_0,j_0}(\mathbf{R}^d)\hookrightarrow \mathcal C^{1,\gamma}(\mathbf R^d)$, 
 \begin{align*} 
 |\langle\sigma_*(\cdot,x),m_s^N\rangle\langle\nabla f\cdot\nabla\sigma_*(\cdot,x),\bar \mu_s\rangle\big| &\le  C   \sup_{N\ge 1}\sup_{s\in[0,t]} \Vert m_s^N\Vert_{\mathcal H^{-J_0+1,j_0}}\\
 &\quad \times \Vert f\Vert_{\mathcal H^{J_0,j_0} }    {\sup}_{s\in[0,t]}| \langle (1+|\cdot |^\gamma), \bar \mu_s\rangle |<+\infty,
 \end{align*} 
 for some $C>0$ independent of $N\ge 1$,
 $f\in \mathcal H^{J_0,j_0}(\mathbf R^d) $, $s\ge 0$, and
 $(x,y)\in \mathcal X\times \mathcal Y$.  With these two upper bounds,
 and using the same arguments as those used in the proof of
 Lemma~\ref{le.Fc}, one deduces that for all continuity points
 $t\in\mathbf{R}_+$ of $\lbrace t\mapsto m_t,t\in\mathbf{R}_+\rbrace$,
 we have $\mathbf B_t[f](m^N)\rightarrow \mathbf B_t[f](m)$ as
 $N\rightarrow+\infty$.  Consequently, using~\eqref{eq.Ceta-} and the
 continuous mapping
 theorem~\cite[Theorem~2.7]{billingsley2013convergence}, for all
 $f\in \mathcal H^{J_0,j_0}(\mathbf{R}^d)$ and $t\in \mathbf R_+$, it
 holds in distribution and as $N'\to +\infty$:
 \begin{equation}
 \label{eq.CVL1}
 \mathbf B_t[f](\eta^{N'})-\langle f,\eta_0^{N'}\rangle-\langle f,\sqrt{N'}M_t^{N'}\rangle\to \mathbf B_t[f](  \eta^*)-{\langle f,\eta_0^*\rangle}- \langle f,\mathscr G_t^*\rangle.
 \end{equation}

 \noindent
 \textbf{Step 2.}  In this step we prove that for any
 $t\in \mathbf R_+$ and $f\in \mathcal H^{J_0,j_0}(\mathbf R^d) $:
 \begin{equation}\label{eq.CVL2}
\mathbf E\big [ \big |\mathbf e^N_t[f] \big  |\big ]\to 0 \text{ as } N\to +\infty.
\end{equation}
By~\eqref{equation de lemme cc eta^N}, we have since
$f\in \mathcal H^{J_0,j_0}(\mathbf R^d)\hookrightarrow \mathcal
H^{J_2+1,j_2}(\mathbf R^d)$ (because $J_0\ge J_2+1$ and $j_2\ge j_0$,
see~\eqref{eq.J0} and~\eqref{eq.J1}),
\begin{align*}
\frac{1}{\sqrt{N}}\mathbf E\Big [\int_0^t\int_{\mathcal{X}\times\mathcal{Y}}\alpha|\langle\sigma_*(\cdot,x),\eta_s^N\rangle||\langle\nabla f\cdot\nabla\sigma_*(\cdot,x),\eta_s^N\rangle|\pi(\di x,\di y)\di s\Big ]&\le \frac{Ct \Vert \nabla f\Vert_{\mathcal H^{J_2,j_2}} }{\sqrt N}  \mathbf{E}\Big[ \sup_{t\in[0,T]} \|\eta_t^N \|_{\mathcal H^{-J_2,j_2}}^2\Big]\\
\le \frac{Ct \Vert   f\Vert_{\mathcal H^{J_2+1,j_2}} }{\sqrt N} \to 0 \text{ as } N\to +\infty. 
\end{align*}
Using Lemma~\ref{lem:remainder terms}, we also have since
$f\in \mathcal H^{J_0,j_0}(\mathbf R^d)\hookrightarrow \mathcal
C^{2,\gamma_*}(\mathbf R^d)$,
\begin{align*}
\mathbf E\Big [\Big|\sqrt{N}\langle f,V_t^N\rangle-\sqrt{N}\sum_{k=0}^{\lfloor Nt\rfloor-1}\langle f,R_k^{N}\rangle-\frac{\sqrt{N}}{N^{1+\beta}}\sum_{k=0}^{\lfloor Nt\rfloor-1}\sum_{i=1}^N\nabla f(W_k^i)\cdot\ve_k^i\Big|\Big]&\le C\|f\|_{\mathcal C^{2,\gamma_*}}\Big[ \frac{\sqrt N}{ N} \\
&\quad + \sqrt N N \big[\frac{1}{N^2}+\frac{1}{N^{2\beta}}\big]+ \frac{\sqrt N}{N^\beta}\Big]. 
\end{align*}
The right-hand-side of the previous term goes to $0$ as
$N\to +\infty$, since $\beta>3/4$. This proves~\eqref{eq.CVL2}.

  \medskip
 
 \noindent
 \textbf{Step 3.} End of the proof of Proposition~\ref{pr.LM-TCL}.
 By~\eqref{eq.CVL1},~\eqref{eq.CVL2}, and~\eqref{eq.CVL3}, we deduce
 that for all $f\in \mathcal H^{J_0,j_0}(\mathbf{R}^d)$ and
 $t\in \mathbf R^+$, it holds a.s.
 $\mathbf B_t[f]( \eta^*)-\langle f,\eta_0^*\rangle- \langle
 f,\mathscr G^*_t\rangle =0$. Since
 $\mathcal H^{J_0,j_0}(\mathbf{R}^d)$ and $\mathbf R_+$ are separable,
 we conclude by a standard continuity argument   that a.s. for
 all $f\in \mathcal H^{J_0,j_0}(\mathbf{R}^d)$ and $t\in \mathbf R^+$, 
 $\mathbf B_t[f]( \eta^*)-\langle f,\eta_0^*\rangle- \langle
 f,\mathscr G^*_t\rangle =0$. Hence, $\eta^*$ is a weak solution
 of~\eqref{eq.CLT} (see Definition~\ref{de.weak}) with initial
 distribution $\nu_0$ (see~\eqref{eq.eta*0}). This ends the proof of
 Proposition~\ref{pr.LM-TCL}.
\end{proof}

Inspired by the proof of \cite[Corollary 5.7]{delarue} (see also
\cite{kurtz2004stochastic}), to end the proof of
Theorem~\ref{thm:clt}, we will show that~\eqref{eq.CLT} has a unique
strong solution. This is the purpose of the next section, where we
also conclude the proof of Theorem~\ref{thm:clt}.

\subsubsection{Pathwise uniqueness}

\begin{prop}\label{prop:uniqueness clt} 
  Let $\beta>3/4$ and assume~\ref{as:batch}-\ref{as:batch limite}.
  Then strong uniqueness holds for~\eqref{eq.CLT}, namely, on a fixed
  probability space, given a
  $\mathcal H^{-J_0+1,j_0}(\mathbf{R}^d)$-valued random variable $\nu$
  and a {\rm G}-process
  $\mathscr G\in \mathcal C(\mathbf R_+,\mathcal
  H^{-J_0,j_0}(\mathbf{R}^d))$, there exists at most one
  $ \mathcal C(\mathbf{R}_+,\mathcal
  H^{-J_0+1,j_0}(\mathbf{R}^d))$-valued process $\eta$ solution
  to~\eqref{eq.CLT} with $\eta_0=\nu$ almost surely.
\end{prop}

\begin{proof}  
  By linearity of the involved operators in~\eqref{eq.CLT}, it is
  enough to consider a
  $ \mathcal C(\mathbf{R}_+,\mathcal
  H^{-J_0+1,j_0}(\mathbf{R}^d))$-valued process $\eta$ solution
  to~\eqref{eq.CLT} when a.s. $\nu=0$ and $\mathscr G=0$, i.e., for
  every $f\in \mathcal H^{J_0,j_0}(\mathbf R^d)$ and
  $t\in\mathbf{R}_+$:
\begin{align}\begin{cases}\label{systeme pour Phi}
\langle f,\eta_t\rangle-\int_0^t( \mathbf  U_s[f](\eta_s)- \mathbf  V_s[f](\eta_s))\di s=0, \\
\langle f,\eta_0\rangle=0,
\end{cases}\end{align}
where $\mathbf U$ and $\mathbf V$ are defined respectively
in~\eqref{eq.Gf} and~\eqref{eq.Hf}.  Pick $T>0$.  By~\eqref{systeme
  pour Phi}, a.s. for all $f\in \mathcal H^{J_0,j_0}(\mathbf{R}^d)$
and $t\in[0,T]$, we have
$\langle f,\eta_t\rangle^2=2\int_0^t( \mathbf U_s[f](\eta_s)- \mathbf
V_s[f](\eta_s))\langle f,\eta_s\rangle\di s$.  Recall
$ {\sup}_{s\in[0,T]}| \langle (1+|\cdot |^\gamma), \bar \mu_s\rangle |<+\infty$ (because
$\bar \mu\in \mathcal D(\mathbf{R}_+,\mathcal
P_{\gamma}(\mathbf{R}^d))$) and $\mathcal
H^{L,\gamma}(\mathbf{R}^d)\hookrightarrow \mathcal
C^{1,\gamma}(\mathbf{R}^d)$. Then, by Cauchy-Schwarz inequality,
a.s. for all $f\in \mathcal H^{J_0,j_0}(\mathbf{R}^d)$ and
$t\in[0,T]$:
\begin{align}
\nonumber
-\int_0^t \mathbf  V_s[f](\eta_s)\langle f,\eta_s\rangle\di s&\leq    \alpha\int_0^t\Big[\langle f,\eta_s\rangle^2+\int_{\mathcal{X}\times\mathcal{Y}}\langle\sigma_*(\cdot,x),\eta_s\rangle^2\langle\nabla f\cdot\nabla\sigma_*(\cdot,x),\bar \mu_s\rangle^2\pi(\di x,\di y)\Big]\di s \\
\label{76}
 &\leq C\int_0^t\Big[\langle f,\eta_s\rangle^2+\|\eta_s\|_{\mathcal H^{-J_0,j_0}}^2\| f\|_{ {\mathcal H^{L,\gamma}}}^2\Big]\di s.
\end{align}
Let $\{f_a\}_{a\ge1}$ be an orthonormal basis of
$\mathcal H^{J_0,j_0}(\mathbf{R}^d)$.  Using the operator
$\mathsf T_x:f\in \mathcal H^{J_0,j_0}(\mathbf{R}^d)\mapsto\nabla
f\cdot\nabla\sigma(\cdot,x)\in \mathcal H^{J_0-1,j_0}(\mathbf{R}^d)$
defined for all $x\in\mathcal{X}$ and Lemma~\ref{lem : <xi,gxi>}, we
have a.s.  for all $t\in[0,T]$:
\begin{align}\label{77}
\sum_{a\ge1}\int_0^t \mathbf  U_s[f_a](\eta_s)\langle f_a,\eta_s\rangle\di s&=\int_0^t\int_{\mathcal{X}\times\mathcal{Y}}(y-\langle\sigma_*(\cdot,x),\bar \mu_s\rangle)\Big(\sum_{a\ge1}\langle f_a,\eta_s\rangle\langle\mathsf T_xf_a,\eta_s\rangle\Big)\pi(\di x,\di y)\di s\nonumber\\
&= \int_0^t\int_{\mathcal{X}\times\mathcal{Y}}(y-\langle\sigma_*(\cdot,x),\bar \mu_s\rangle)\langle\eta_s,\mathsf T_x^*\eta_s\rangle_{\mathcal H^{-J_0,j_0}}\pi(\di x,\di y)\di s\nonumber\\
&\leq C\int_0^t\|\eta_s\|^2_{\mathcal H^{-J_0,j_0}}\di s.
\end{align} 
Therefore, using the bounds~\eqref{76} and~\eqref{77}, together with
$\mathcal H^{J_0,j_0}(\mathbf R^d)\hookrightarrow_{\text{H.S.}}
\mathcal H^{L,\gamma}(\mathbf R^d)$, we have a.s.  for all
$t\in[0,T]$:
\begin{align*}
\|\eta_t\|_{\mathcal H^{-J_0,j_0}}^2=\sum_{a\ge1}\langle f_a,\eta_t\rangle^2\leq C\int_0^t\|\eta_s\|^2_{\mathcal H^{-J_0,j_0}}\di s.
\end{align*}
By Gronwall's lemma, a.s. $\|\eta_t\|_{\mathcal H^{-J_0,j_0}}=0$ for
all $t\in[0,T]$. This concludes the proof of
Proposition~\ref{prop:uniqueness clt}.
\end{proof}

  \subsubsection{End of the proof of Theorem~\ref{thm:clt}}
  
  \begin{proof}[Proof of Theorem~\ref{thm:clt}] \begin{sloppypar} Let
      $\ell \in\{1,2\}$ and $N_\ell$ be such that in distribution
      $\eta^{N_\ell}\to \eta^{\ell}$ in
      $\mathcal D(\mathbf R_+,\mathcal H^{-J_0+1,j_0}(\mathbf R^d))$
      (see Proposition~\ref{prop relative compactness}). By
      Lemma~\ref{lem: continuity prop eta}, a.s.
      $ \eta^{\ell}\in \mathcal C(\mathbf{R}_+,\mathcal
      H^{-J_0+1,j_0}(\mathbf{R}^d))$.  Consider now
      $( \eta^{\ell,*},\mathscr G^{\ell,*}) $ a limit point of
      $( \eta^{N_\ell},\sqrt {N_\ell}\, M^{N_\ell})_{N_\ell\ge1}$ in
      $\mathscr E$. Up to extracting a subsequence from $N_\ell$, we
      assume that in distribution and as $N_\ell\to +\infty$,
 $$(  \eta^{N_\ell},\sqrt {N_\ell}\, M^{N_\ell})\to ( \eta^{\ell,*},\mathscr G^{\ell,*}) \text{ in } \mathscr E.$$
 Considering the marginal distributions, we then have by uniqueness of
 the limit in distribution, for $\ell=1,2$ (see also
 Proposition~\ref{prop.CVMN}):
\begin{equation}\label{eq.=law}
    \eta^{\ell,*}=\eta^{\ell},  \text{ and }  \mathscr G^{\ell,*}=\mathscr{G} \text{ in distribution}.
\end{equation} 
By Proposition~\ref{pr.LM-TCL}, $\eta^{1,*}$ and $\eta^{2,*}$ are two
weak solutions of~\eqref{eq.CLT} with initial distribution $\nu_0$
(see also Lemma~\ref{le.eta0}). Since strong uniqueness
for~\eqref{eq.CLT} (see Proposition~\ref{prop:uniqueness clt}) implies
weak uniqueness for~\eqref{eq.CLT}, we deduce that
$\eta^{1,*}=\eta^{2,*}$ in distribution. By~\eqref{eq.=law}, this
implies that $\eta^{1}=\eta^{2}$ in distribution and then, that the
whole sequence $(\eta^N)_{N\ge1}$ converges in distribution in
$\mathcal D(\mathbf R_+,\mathcal H^{-J_0+1,j_0}(\mathbf
R^d))$. Denoting by $\eta^*$ its limit, we have proved that $\eta^*$
has the same distribution as the unique weak solution $\eta^\star$
of~\eqref{eq.CLT} with initial distribution $\nu_0$.  This concludes
the proof of Theorem~\ref{thm:clt}.
\end{sloppypar}
 \end{proof}


%
%
%
%
%

\subsection{The case when $\beta=3/4$}
 
In this section, we assume that $d=1$. Recall
$\mathsf f_2:x\in \mathbf R\mapsto |x|^2$, which belongs to
$\mathcal H^{J_0,j_0}(\mathbf R)$ because $j_0-2=3-2>1/2$
(see~\eqref{eq.J0}).


\begin{proof}[Proof of Proposition \ref{pr.beta34}]
  Assume $\beta=3/4$.  The proof of Proposition~\ref{pr.beta34} is
  divided into two steps.  \medskip

\noindent
\textbf{Step 1.} Let $f\in \mathcal H^{J_0,j_0}(\mathbf{R})$.  When
$\beta=3/4$, it appears that for non affine test functions $f$, the
term $\langle f,R_k^{N}\rangle$ is not negligible any more.  In this
step we simply rewrite~\eqref{eq.CVL3} by decomposing the term
$\langle f,R_k^{N}\rangle$ into two terms:
$\langle f,R_k^{N}\rangle= \langle f,\mathscr R_k^N\rangle+\langle
f,\mathscr B_k^N\rangle$ where $\langle f,\mathscr R_k^N\rangle$ will
be negligible and $\langle f,\mathscr B_k^N\rangle$ will not be
negligible.  More precisely, by~\eqref{R_k^{2,N}}
and~\eqref{algorithm}, it holds
\begin{align*}
\langle f,R_k^{N}\rangle&=\frac{1}{2N}\sum_{i=1}^N(W_{k+1}^i-W_k^i)^2  f(\widehat{W}_k^i)\\
&=\frac{1}{2N}\sum_{i=1}^N\Big|\frac{\alpha}{N|B_k|}\sum_{(x,y)\in B_k}(y-g_{W_k}^N(x))\nabla_W\sigma_*(W_k^i,x)+\frac{\ve_k^i}{N^{3/4}}\Big|^2 f''(\widehat{W}_k^i)\\
&=\langle f,\mathscr R_k^N\rangle +\langle f,\mathscr B_k^N\rangle,
\end{align*}
where 
\begin{align*}
\langle f,\mathscr R_k^N\rangle&=\frac{1}{2N}\sum_{i=1}^N\Big|\frac{\alpha}{N|B_k|}\sum_{(x,y)\in B_k}(y-g_{W_k}^N(x))\nabla_W\sigma_*(W_k^i,x)\Big|^2f''(\widehat{W}_k^i)\\
&\quad+\frac{1}{N}\sum_{i=1}^N\frac{\alpha}{N|B_k|}\sum_{(x,y)\in B_k}(y-g_{W_k}^N(x))\nabla_W\sigma_*(W_k^i,x)\frac{\ve_k^i}{N^{3/4}} f''(\widehat{W}_k^i)
\end{align*}
and 
$$
\langle f,\mathscr B_k^N\rangle= \frac{1}{2N^{5/2}}\sum_{i=1}^N|\ve_k^i|^2 f''(\widehat{W}_k^i).
$$
From~\eqref{eq.CVL3}, one then has:
\begin{align}\label{eq.CLT-3/4}
\langle f,\eta_t^N\rangle-\langle f,\eta_0^N\rangle-\int_0^t(   \mathbf U_s[f](\eta^N_s)-\mathbf V_s[f](\eta_s^N))\di s -\langle f,\sqrt{N}M^N_t\rangle=-\mathbf{\tilde e}^{N}_t[f]+\sqrt{N}\sum_{k=0}^{\lfloor Nt\rfloor-1}\langle f,\mathscr B_k^N\rangle,
\end{align} 
where 
\begin{align*}
\nonumber
\mathbf{\tilde e}^{N}_t [f]:&=\frac{1}{\sqrt{N}}\int_0^t\int_{\mathcal{X}\times\mathcal{Y}}\alpha\langle\sigma_*(\cdot,x),\eta_s^N\rangle\langle\nabla f\cdot\nabla\sigma_*(\cdot,x),\eta_s^N\rangle\pi(\di x,\di y)\di s\\
&\quad -\sqrt{N}\langle f,V_t^N\rangle-\sqrt{N}\sum_{k=0}^{\lfloor Nt\rfloor-1}\langle f,\mathscr R_k^{N}\rangle-\frac{\sqrt{N}}{N^{1+\beta}}\sum_{k=0}^{\lfloor Nt\rfloor-1}\sum_{i=1}^N\nabla f(W_k^i)\, \ve_k^i.
\end{align*}

\noi \textbf{Step 2.}  Let $\eta$ be a limit point of
$(\eta^N)_{N\ge 1}$ in
$\mathcal D(\mathbf R_+,\mathcal H^{-J_0+1,j_0}(\mathbf R))$. Let $N'$
be such that in distribution $\eta^{N'}\to \eta$ as $N'\to
+\infty$. In this step, we pass to the limit in~\eqref{eq.CLT-3/4}
with the test function $\mathsf f_2:x\in \mathbf R\mapsto |x|^2$.  By
Propositions~\ref{prop relative compactness} and~\ref{prop relative
  compactness M}, the sequence $(\eta^{N'},\sqrt {N'}M^{N'})_{N'\ge1}$
is tight in $\mathscr E$ (see~\eqref{eq.LE-}).  Let
$(\eta^*,\mathscr G^*) $ be one of its limit point in~$\mathscr E$. Up
to extracting a subsequence from $N'$, it holds:
$$
(\eta^{N'},\sqrt{N'}\, M^{N'})\to(\eta^*,\mathscr G^*), \text{ as } N'\to +\infty. 
$$
  Considering the marginal distributions, 
  it holds  in distribution, 
 $$
    \eta^*=\eta  \text{ and }  \mathscr G^*=\mathscr{G} \in \mathcal C(\mathbf{R}_+,\mathcal H^{-J_0,j_0}(\mathbf{R})).
$$
Introduce $\mathcal C(\eta^*)\subset \mathbf R_+$, whose complementary
in $\mathbf R_+$ is at most countable, such that for all
$u\in \mathcal C(\eta^*)$ ,
$s\in \mathbf R_+\mapsto \eta^*_s\in \mathcal H^{-J_0+1,j_0}(\mathbf
R)$ is a.s. continuous at $u$. Then, with the same arguments as those
used to derive~\eqref{eq.CVL1} and using also the fact that
$0\in\mathcal C(\eta^*)$, one has for all $t\in \mathcal C(\eta^*)$ and
in distribution,
  \begin{equation}\label{beta34 eq1}
\mathbf B_t[f](\eta^{N'})-\langle f,\eta_0^{N'}\rangle-\langle f,\sqrt{N'}M_t^{N'}\rangle\to \mathbf B_t[f](  \eta^*)-\langle f,\eta_0^*\rangle- \langle f,\mathscr G_t^*\rangle \text{ as } N'\to +\infty.
\end{equation}
Let us now deal with the two terms in the right-hand side
of~\eqref{eq.CLT-3/4}.  Using~\eqref{convexity
  inequality},~\eqref{eq.conditio} and~\ref{as:data},
\begin{align*}
\mathbf E\Big[\frac{1}{|B_k|}\Big|\sum_{(x,y)\in B_k}(y-g_{W_k}^N(x))\nabla_W\sigma_*(W_k^i,x)\ve_k^i\Big|\Big]&\leq 2\mathbf E\Big[\frac{1}{|B_k|}\sum_{(x,y)\in B_k}\left|(y-g_{W_k}^N(x))\nabla_W\sigma_*(W_k^i,x)\right|^2\Big]\\
&\quad+2\mathbf E\Big[|\ve_k^i|^2\Big]\\
&\leq C\Big[\mathbf E\Big[\frac{1}{|B_k|}\sum_{(x,y)\in B_k}(|y|^2+1)\Big]+1\Big]\leq C.
\end{align*}
We now set $f=\mathsf f_2$.  Then, we have
$\mathbf E\big [|\langle\mathsf f_2,\mathscr R_k^N\rangle|\big ]\leq
C(N^{-2}+N^{-7/4})$.  Using also the lines below~\eqref{eq.CVL2} and
Lemma~\ref{lem:remainder terms}, it holds:
\begin{align}\label{beta34 eq2}
\mathbf E\big [|\mathbf{\tilde e}^{N}_t [\mathsf  f_2]|\big ]\leq C\Big[\frac{1}{\sqrt N}+N^{3/2}\Big(\frac{1}{N^2}+\frac{1}{N^{7/4}}\Big)+\frac{1}{N^{1/4}}\Big].
\end{align}
On the other hand, using~\ref{as:noise} and the law of large number,
it holds a.s. as $N\to+\infty$,
\begin{align}\label{beta34 eq3}
\sqrt{N}\sum_{k=0}^{\lfloor Nt\rfloor-1}\langle \mathsf f_2,\mathscr B_k^N\rangle =\frac{1}{2N^2}\sum_{k=0}^{\lfloor Nt\rfloor-1}\sum_{i=1}^N|\ve_k^i|^2\mathsf f_2''(\widehat{W}_k^i)=\frac{1}{N^2}\sum_{k=0}^{\lfloor Nt\rfloor-1}\sum_{i=1}^N|\ve_k^i|^2\to t\textbf{E}[|\ve_1^1|^2].
\end{align}
Therefore, using~\eqref{beta34 eq1},\eqref{beta34 eq2}
and~\eqref{beta34 eq3}, it holds for all $t\in \mathcal C(\eta^*)$,
a.s.
$\mathbf B_t[\mathsf f_2]( \eta^*)-\langle \mathsf
f_2,\eta_0^*\rangle- \langle \mathsf f_2,\mathscr G^*_t\rangle
=t\textbf{E}[|\ve_1^1|^2]$. 
 The mapping\footnote{For all
  $m\in\mathcal D(\mathbf R_+,\mathcal H^{-J_0+1,j_0}(\mathbf R^d))$
  and $f\in \mathcal H^{J_0,j_0}(\mathbf R^d)$,
  $t\in\mathbf R_+\mapsto \mathbf B_t[f](m)$ is right-continuous. This
  is clear since $t\mapsto \langle f,m_t\rangle$ is right-continuous,
  and because
  $s\mapsto\mathbf U_s[f](m_s)-\mathbf V_s[f](m_s) \in
  L^\infty_{loc}(\mathbf R_+)$.}  
$s\in\mathbf R_+\mapsto \mathbf B_s[\mathsf f_2](\eta^*)$ is right continuous and
  $s\mapsto \langle\mathsf f_2,\mathscr G^*_s\rangle$  is  continuous.  By a standard continuity argument (the same as the one used in Proposition \ref{lem:convergence to lim eq}), it holds a.s. for all $t\in \mathbf R_+$,
$\mathbf B_t[\mathsf f_2]( \eta^*)-\langle \mathsf
f_2,\eta_0^*\rangle- \langle\mathsf f_2,\mathscr G^*_t\rangle
=t\textbf{E}[|\ve_1^1|^2]$.  The proof of Proposition~\ref{pr.beta34}
is complete.
%
%
%
\end{proof}


\bigskip
\noindent
\textbf{Acknowledgment}.  The authors are grateful to Benoît Bonnet
for fruitful discussions about the proof of Proposition~\ref{prop:
  existence uniqueness eq transport}.  A.D. is grateful for the
support received from the Agence Nationale de la Recherche (ANR) of
the French government through the program "Investissements d'Avenir"
(16-IDEX-0001 CAP 20-25).  A.G. is supported by the French ANR under
the grant ANR-17-CE40-0030 (project \emph{EFI}) and the Institut
Universtaire de France. M.M. acknowledges the support of the the
French ANR under the grant ANR-20-CE46-0007 (\emph{SuSa} project).
B.N. is supported by the grant IA20Nectoux from the Projet I-SITE
Clermont CAP 20-25.

\bibliography{biblio}
 
\appendix

\section{A note on relative compactness}
\label{app:A}

In this section, we prove, in our Hilbert setting, that the condition
(4.21) in \cite[Theorem (4.20)]{kurtz1975semigroups} can be replaced
by the slightly modified condition, namely the regularity condition of
item 2 in Proposition~\ref{lem:note kurtz} below.

In the following $\mathcal H_1$ and $\mathcal H_2$ are two Hilbert
spaces (whose duals are respectively denoted by $\mathcal H_1^{-1}$
and $\mathcal H_2^{-1}$) such that
$\mathcal H_1\hookrightarrow_{\text{H.S.}} \mathcal H_2$.

\begin{prop}
\label{lem:note kurtz}
Let
$(\mu^N)_{N\ge 1}\subset\mathcal D(\mathbf R_+, \mathcal
H_2^{-1})$ 
be a sequence of processes satisfying the  following two   conditions:
\begin{enumerate}
\item\label{kurtz item cc} Compact containment condition : for every $T>0$ and $\eta>0$, there exists $C>0$ such that 
$$\sup_{N\geq1}\ \mathbf{P}\big({\sup}_{t\in[0,T]} \|\mu_t^N\|_{\mathcal H_2^{-1}}^2>C\big)\leq\eta.$$  
\item\label{kurtz item reg} Regularity condition : for every
  $\delta>0$, $N\geq1,\ 0\leq t\leq T$ and
  $0\leq u\leq (T-t)\wedge\delta$, there exists $F_N(\delta)<\infty$
  such that
$$\mathbf{E}\left[ \|\mu_{t+u}^N-\mu_t^N\|_{\mathcal H_2^{-1}}^2\right]\leq F_N(\delta),$$
with $\underset{\delta\rightarrow 0}{\lim} \ \underset{N\ge 1}{\limsup} \ F_N(\delta)=0$. 
\end{enumerate}
Then, the sequence $(\mu^N)_{N\ge 1}\subset\mathcal D(\mathbf R_+, \mathcal H_1^{-1})$ is relatively compact. 
\end{prop}


\begin{proof}
  To prove this result, we follow the proof of \cite[Theorem
  (4.20)]{kurtz1975semigroups}.  More precisely, we show that the
  assumptions of \cite[Theorem 1]{kurtz1975semigroups} are satisfied
  (namely conditions (4.2) and (4.3) there), when
  $E= \mathcal H_1^{-1}$
  there.  
\medskip

\noindent
\textbf{Step 1.}  The condition (4.2) in \cite[Theorem
1]{kurtz1975semigroups} is satisfied (when $E= \mathcal H_1^{-1}$
there).  \medskip

\noindent
We have that $\mathcal H_1$ is compactly embedded in $\mathcal H_2$
(since a Hilbert-Schmidt embedding is compact).  By Schauder's
theorem, $\mathcal H_2^{-1}$ is compactly embedded in
$\mathcal H_1^{-1}$.  Thus, for all $C>0$, the set
$\lbrace\phi\in \mathcal H_1^{-1},\ \|\phi\|_{\mathcal H_2^{-1}}\leq
C\rbrace$ is compact. Therefore, the condition (4.2) in
\cite{kurtz1975semigroups} is satisfied.  \medskip

\noindent
\textbf{Step 2.} The condition (4.3) in \cite[Theorem
1]{kurtz1975semigroups} is satisfied (when $E= \mathcal H_1^{-1}$
there).  \medskip

\noindent 
By \cite[Lemma 4.4]{kurtz1975semigroups},
$\lbrace t\mapsto\mu_t^N,t\in\mathbf{R}_+\rbrace_{N\geq1}$ is
relatively compact in $\mathcal D(\mathbf R_+, \mathcal H_1^{-1})$ if
for all $\ve >0$, there exists a tight sequence
$\lbrace
t\mapsto\mu_t^{N,\ve},t\in\mathbf{R}_+\rbrace_{N\geq1}\subset\mathcal
D(\mathbf R_+, \mathcal H_1^{-1})$ which is $\ve$-close to
$\lbrace t\mapsto\mu_t^N,t\in\mathbf{R}_+\rbrace_{N\geq1}$.  Following
\cite{kurtz1975semigroups}, we define, for $\ve>0$, the sequence
$\lbrace
t\mapsto\mu_t^{N,\ve},t\in\mathbf{R}_+\rbrace_{N\geq1}\subset\mathcal
D(\mathbf R_+, \mathcal H_1^{-1})$ in
$\mathcal D(\mathbf R_+, \mathcal H_2^{-1})$ of pure jump processes as
follows. Let us first introduce, for $N\geq1$ and $\ve>0$,
$\tau_0^{N,\ve}:=0$ and, for $k>0 :$
\begin{equation*}\begin{split}
\tau_k^{N,\ve}&:=\inf\lbrace t>\tau_{k-1}^{N,\ve} :  \|\mu_t^N-\mu^N_{\tau_{k-1}^{N,\ve}}\|_{\mathcal H_2^{-1}}>\ve\rbrace,\\
s_k^{N,\ve}&:=\sup\lbrace t<\tau_{k}^{N,\ve} :   \|\mu^N_t-\mu^N_{\tau_{k}^{N,\ve}}\|_{\mathcal H_2^{-1}}\geq\ve\rbrace.
\end{split}\end{equation*}
Then we define, for $\ve>0$, 
$$\mu_t^{N,\ve}:=\left\{
\begin{array}{ll}
\mu^N_0 & \mbox{for} \quad t<\frac{1}{2}(s_1^{N,\ve}+\tau_1^{N,\ve})\\
\mu_{\tau_k^{N,\ve}}^N & \text{for} \ \frac{1}{2}(s_k^{N,\ve}+\tau_k^{N,\ve})\leq t<\frac{1}{2}(s_{k+1}^{N,\ve}+\tau_{k+1}^{N,\ve}).\\
\end{array}
\right.$$ We claim that for any $\ve>0$, the sequence
$\lbrace t\mapsto\mu_t^{N,\ve},t\in \mathbf R_+\rbrace_{N\geq1}$ verifies
condition (4.2) of \cite[Theorem 4.1]{kurtz1975semigroups} when
$E= \mathcal H_1^{-1}$ there. Indeed, by the discussion in the first
step above, this follows from the compact containment condition
verified by
$\lbrace t\mapsto\mu_t^{N},t\in\mathbf{R}_+\rbrace_{N\geq1}$ in
$\mathcal D(\mathbf R_+, \mathcal H_2^{-1})$ (see item 1 in
Proposition~\ref{lem:note kurtz}) together with the fact that
$ {\sup}_{t\in\mathbf{R}_+} \|\mu_t^N-\mu_t^{N,\ve}\|_{\mathcal
  H_2^{-1}}\leq C_0\ve$, where the constant $C_0>0$ is independent
of~$N$ and $\ve$.

It remains to prove that for any $\ve>0$,
$\lbrace t\mapsto\mu_t^{N,\ve},t\in\mathbf{R}_+\rbrace_{N\geq1}$
satisfies the condition (4.3) in \cite[Theorem
4.1]{kurtz1975semigroups} when $E= \mathcal H_1^{-1}$ there (so that
it will be tight for each $\ve >0$).  Since
$\mathcal H_2^{-1} \hookrightarrow \mathcal H_1^{-1}$, it is enough to
show that for any $\ve>0$,
$\lbrace t\mapsto\mu_t^{N,\ve},t\in\mathbf{R}_+\rbrace_{N\geq1}$
satisfies the condition (4.3) in \cite[Theorem
4.1]{kurtz1975semigroups} when $E=\mathcal H_2^{-1}$ there.  By
\cite[Lemma 4.5]{kurtz1975semigroups} (and its note) and the
construction of
$\lbrace t\mapsto\mu_t^{N,\ve},t\in\mathbf{R}_+\rbrace_{N\geq1}$, it
is sufficient to bound
$\delta\mapsto\textbf{P} (\tau_1^{N,\ve}\leq\delta )$ and
$\delta\mapsto\textbf{P} (\tau_{k+1}^{N,\ve}-s_k^{N,\ve}\leq\delta )$
by a function $\delta\mapsto G_N^\ve(\delta)$ such that for every
$\ve>0,$
\begin{equation}\label{eq.ConFN}
\underset{\delta\rightarrow 0}{\lim} \ \underset{N}{\limsup} \ G_N^\ve(\delta)=0.
\end{equation}
Let us prove that we can indeed bound these two terms by such
$ G_N^\ve(\delta)$ satisfying~\eqref{eq.ConFN}.  Recall that by item 2
in Proposition~\ref{lem:note kurtz}, for every $\delta>0$,
$N\geq1,\ 0\leq t\leq T$ and $0\leq u\leq (T-t)\wedge\delta$,
$$\textbf{E}\left[ \|\mu_{t+u}^N-\mu^N_t\|_{\mathcal H_2^{-1}}^2 \right]\leq F_N(\delta), \text{ with $\underset{\delta\rightarrow 0}{\lim} \ \underset{N\ge 1}{\limsup} \ F_N(\delta)=0$. }$$
Now, introduce as in \cite{kurtz1975semigroups}, the distance $r$ on $\mathcal H_2^{-1}$ defined by $r(\varphi^1,\varphi^2)=\min(1, \|\varphi^1-\varphi^2\|_{\mathcal H^{-1}_2})$.  
We thus have: 
 \begin{align}\label{kurtz eq3}
\textbf{E}\left[ r\left(\mu^N_{t+u},\mu^N_t\right)^2\right]\leq F_N(\delta),
\end{align}
On the one hand, we have, for $0<\ve<1$,  
\begin{align}\label{kurtz eq1}
\textbf{P} (\tau_1^{N,\ve}\leq \delta )&=\textbf{P}\big( \|\mu^N_{\tau_1^{N,\ve}\wedge\delta}-\mu^N_0\|_{\mathcal H_2^{-1}} >\ve\big)=\textbf{P}\big(r \big(\mu^N_{\tau_1^{N,\ve}\wedge\delta},\mu^N_0\big)>\ve\big) \leq\frac{1}{\ve^2}\textbf{E}\big[r\big(\mu^N_{\tau_1^{N,\ve}\wedge\delta},\mu^N_0\big)^2\big].
\end{align}
On the other hand, we have, for $k\geq1$ and $0<\ve<1$, 
\begin{align}\label{kurtz eq2}
\textbf{P} (\tau_{k+1}^{N,\ve}-s_k^{N,\ve}\leq \delta )\leq\textbf{P} (\tau_{k+1}^{N,\ve}-\tau_k^{N,\ve}\leq \delta )&=\textbf{P}\big(r (\mu^N_{\tau_{k+1}^{N,\ve}\wedge(\tau_k^{N,\ve}+\delta)},\mu^N_{\tau_k^{N,\ve}} )>\ve\big)\nonumber\\
&\leq \frac{1}{\ve^2}\textbf{E}\big[r\big(\mu^N_{\tau_{k+1}^{N,\ve}\wedge(\tau_k^{N,\ve}+\delta)},\mu^N_{\tau_k^{N,\ve}}\big)^2\big].
\end{align}
From~\eqref{kurtz eq3}, and because the stopping times appearing
in~\eqref{kurtz eq1} and~\eqref{kurtz eq2} can be approximated by
sequences of decreasing discrete stopping times, we can indeed bound
$\delta\mapsto\textbf{P} (\tau_1^{N,\ve}\leq\delta )$ and
$\delta\mapsto\textbf{P} (\tau_{k+1}^{N,\ve}-s_k^{N,\ve}\leq\delta )$
by $F_N(\delta)$ which satisfies~\eqref{eq.ConFN}. Consequently, for
each $0<\ve<1$, the condition (4.3) in \cite[Theorem
4.1]{kurtz1975semigroups} is satisfied for the sequence
$\lbrace t\mapsto\mu_t^{N,\ve},t\in\mathbf{R}_+\rbrace_{N\geq1}$ when
$E= \mathcal H_2^{-1}$ there. We can thus apply \cite[Theorem
4.1]{kurtz1975semigroups} to
$\lbrace t\mapsto\mu_t^{N,\ve},t\in\mathbf{R}_+\rbrace_{N\geq1}$,
which is therefore tight in
$\mathcal D(\mathbf R_+, \mathcal H_1^{-1})$ for each $0<\ve<1$.
Using~\cite[Lemma~4.4]{kurtz1975semigroups} (with
$E= \mathcal H_1^{-1}$ there), this concludes the proof of the lemma.
\end{proof}

 \medskip

 \noindent
 \section{Technical lemmata}
 \label{app:B}
In this section we state and prove Lemma~\ref{lem : additionnal terms}, Lemma~\ref{lem : <xi,gxi>} and Lemma~\ref{le.CR}.

\begin{lem}\label{lem : additionnal terms}
  Let $\beta\ge 1/2$ and assume~\ref{as:batch}-\ref{as:batch
    limite}. Recall
  $\mathcal H^{J_1,j_1}(\mathbf{R}^d)\hookrightarrow \sob$
  (see~\eqref{eq.SE2}). Then, for all $T>0$, there exists $C>0$ such
  that for all $f\in \mathcal H^{J_1,j_1}(\mathbf{R}^d)$, $N\ge 1$,
  and $t\in[0,T]$, it holds:
\begin{enumerate}[label=(\roman*), ref=\textit{(\roman*)}]
\item\label{add terms item i} $\mathbf{E}\left[\sum_{k=0}^{\lfloor Nt\rfloor-1}2\langle f,\Upsilon_{\frac{k+1}{N}^-}^N\rangle\sqrt{N}\langle f,M_k^N\rangle+4 N\langle f,M_k^N\rangle^2\right]\leq C\| f\|_{\mathcal H^{L,\gamma}}^2$.
\item\label{add terms item ii} $\mathbf{E}\left[\sum_{k=0}^{\lfloor Nt\rfloor-1}\frac{2\sqrt{N}}{N^{1+\beta}}\langle f,\Upsilon_{\frac{k+1}{N}^-}^N\rangle\sum_{i=1}^N\nabla f(W_k^i)\cdot\ve_k^i+\frac{4 }{N^{1+2\beta}}\Big|\sum_{i=1}^N\nabla f(W_k^i)\cdot\ve_k^i\Big|^2\right]\leq \frac{C}{N^{2\beta-1}}\| f\|_{\mathcal H^{L,\gamma}}^2$.
\item\label{add terms item iii} $\mathbf{E}\left[\sum_{k=0}^{\lfloor Nt\rfloor-1}N|\langle f,R_k^{N}\rangle|^2\right]\leq C\| f\|_{\mathcal H^{L,\gamma}}^2\left[\frac{1}{N^2}+\frac{N^2}{N^{4\beta}}\right]$.
\item\label{add terms item iv} $\mathbf{E}\left[\sum_{k=0}^{\lfloor Nt\rfloor-1}\langle f,\Upsilon_{\frac{k+1}{N}^-}\rangle\sqrt{N}\langle f,R_k^{N}\rangle\right]\leq C\| f\|_{\mathcal H^{L,\gamma}}^2\left[1+\frac{1}{N}+\frac{N^3}{N^{4\beta}}\right]+\mathbf{E}\left[\int_0^t\langle f,\Upsilon_s^N\rangle^2\di s\right]$.
\item\label{add terms item v} $\mathbf{E}\left[\big|\sum_{k=0}^{\lfloor Nt\rfloor-1}\langle f,\Upsilon_{\frac{k+1}{N}^-}^N\rangle \mathbf  a_k^N[f]-\sqrt{N}\int_0^t\langle f,\Upsilon_s^N\rangle \mathbf L_s^N[f]\di s\big|\right]\leq C\| f\|_{\mathcal H^{L,\gamma}}^2$.
\item\label{add terms item vi} $\mathbf{E}\left[\sum_{k=0}^{\lfloor Nt\rfloor-1}|\mathbf   a_k^N[f]|^2\right]\leq C\| f\|_{\mathcal H^{L,\gamma}}^2$.
\end{enumerate}
\end{lem}

\begin{proof}
  Let $T>0$ and $f\in \mathcal H^{J_1,j_1}(\mathbf{R}^d)$. In what
  follows, $C>0$ is a constant, independent of $N\ge1$, $t\in [0,T]$,
  $f$, and $k\in\{0,\dots,\lfloor NT\rfloor-1\}$ which can change from
  one occurence to another.  We recall that for $N\geq 1$ and
  $k\geq 1$, $\mathcal{F}_k^N$ is the $\boldsymbol \sigma$-algebra
  generated by $\{W_0^i\}_{i=1}^N$, $B_j$ and $(\ve_j^i)_{i=1}^N$ for
  $j=0,\dots,k-1$, and that
  $\mathcal{F}_0^N:=\boldsymbol{\sigma}\{\{W_0^i\}_{i=1}^N\}$,
  see~\eqref{eq.filtration}. Recall also the definitions of $M_k^N$
  and $R_k^N$ in~\eqref{def Mk} and~\eqref{R_k^{2,N}} respectively, of
  $\mathbf a_s^N$ and $\mathbf L_s^N$ in~\eqref{eq.akN}, and that for
  $N\geq1$ and $t\in\mathbf{R}_+$,
  $\Upsilon_t^N=\sqrt{N}(\mu_t^N-\bar \mu_t^N)$ (see
  also~\eqref{eq.mubarparticule}).  We start by proving item~\ref{add
    terms item i} in Lemma~\ref{lem : additionnal terms}.  For all
  $t\in[0,T]$, because for all $a\in \mathbf N$ and
  $b\in \{1,\ldots,N\}$, $W_a^b$ is $\mathcal{F}_a^N$-measurable and
  $\ve_a^b\indep \mathcal F_a^N$ (see~\ref{as:noise}) together with
  the fact that $\bar X_s^b$ is $\mathcal{F}_0^N$-measurable (for all
  $s\ge 0$), one has using also~\eqref{eq.MkE=0}:
\begin{align}\label{86}
\mathbf{E}\Big[\sum_{k=0}^{\lfloor Nt\rfloor-1}\langle f,\Upsilon_{\frac{k+1}{N}^-}^N\rangle\sqrt{N}\langle f,M_k^N\rangle\Big]&=N\sum_{k=0}^{\lfloor Nt\rfloor-1}\textbf{E}\Big[\big\langle f,\mu_{\frac{k+1}{N}^-}^N-\bar \mu_{\frac{k+1}{N}}^N\big\rangle\langle f,M_k^N\rangle\Big]\nonumber\\
&=N\sum_{k=0}^{\lfloor Nt\rfloor-1}\textbf{E}\big[\langle f,\nu_k^N\rangle\textbf{E}\left[\left.\langle f,M_k^N\rangle\right|\mathcal{F}_k^N\right]\big]\nonumber\\
&\quad -N\sum_{k=0}^{\lfloor Nt\rfloor-1} \textbf{E}\Big[  \langle f,\bar \mu_{\frac{k+1}{N}}^N \rangle\textbf{E}\Big[ \langle f,M_k^N\rangle \big |\mathcal{F}_k^N\Big]\Big]  =0.
\end{align}
By~\eqref{eq.MTT} $\textbf{E}\left[\langle f,M_k^N\rangle^2\right]\leq C\|f\|_{\mathcal H^{L,\gamma}}^2/N^2$. Together with~\eqref{86}, we deduce  item~\ref{add terms item i}.

Let us now prove item~\ref{add terms item ii}.  
We have,   using~\ref{as:noise}, 
\begin{equation*}\begin{split}
\textbf{E}\big[\langle f,\Upsilon_{\frac{k+1}{N}^-}^N\rangle\sum_{i=1}^N\nabla f(W_k^i)\cdot \ve_k^i\big]&=\sqrt{N}\, \textbf{E}\Big[(\langle f,\nu_k^N\rangle-\langle f,\bar \mu_{\frac{k+1}{N}}^N\rangle)\sum_{i=1}^N\nabla f(W_k^i)\cdot \ve_k^i\Big]\\
&=\sqrt{N}\, \sum_{i=1}^N\textbf{E}\left[(\langle f,\nu_k^N\rangle-\langle f,\bar \mu_{\frac{k+1}{N}}^N\rangle)  \nabla f(W_k^i)  \right] \cdot\textbf{E}[\ve_k^i] =0.
\end{split}\end{equation*} 
On the other hand, using~\eqref{eq.ekekj},
$\sob\hookrightarrow \mathcal C^{2,\gamma_*}(\mathbf{R}^d)$
(see~\eqref{eq.SE1}), and the same arguments as those used
in~\eqref{eq333}, it holds:
$$\textbf{E}\big[\big|\sum_{i=1}^N\nabla f(W_k^i)\cdot\ve_k^i\big|^2\big]=\sum_{i=1}^N \textbf{E}\big[\big| \nabla f(W_k^i)\cdot\ve_k^i\big|^2\big]  \leq CN\|f\|_{\mathcal C^{2,\gamma_*}}^2 \le CN\|f\|_{\mathcal H^{L,\gamma}}^2.$$
This ends the proof of item~\ref{add terms item ii}.  Item~\ref{add
  terms item iii} is a direct consequence of~\eqref{borneERk]} and
$\sob\hookrightarrow \mathcal C^{2,\gamma_*}(\mathbf{R}^d)$.
 
Let us now prove item~\ref{add terms item iv}.  We have that
\begin{equation}\label{eq.la-ici}
\sum_{k=0}^{\lfloor Nt\rfloor-1} \langle f,\Upsilon_{\frac{k+1}{N}^-}\rangle\sqrt{N}\langle f,R_k^{N}\rangle\leq\sum_{k=0}^{\lfloor Nt\rfloor-1}\frac{1}{N}\langle f,\Upsilon_{\frac{k+1}{N}^-}\rangle^2+\sum_{k=0}^{\lfloor Nt\rfloor-1}N^2|\langle f,R_k^{N}\rangle|^2.
\end{equation}
 On the one hand, by item~\ref{add terms item iii}, 
\begin{equation}\label{eq.la}
\mathbf{E}\Big[\sum_{k=0}^{\lfloor Nt\rfloor-1}N^2|\langle f,R_k^{N}\rangle|^2\Big]\leq C\|f\|_{\mathcal H^{L,\gamma}}^2\Big[\frac{1}{N}+\frac{N^3}{N^{4\beta}}\Big].
\end{equation}
On the other hand, we have 
\begin{equation*}\begin{split}
\Big|\int_{0}^{\frac{\lfloor Nt\rfloor}{N}}\langle f,\Upsilon_s^N\rangle^2\di s-\sum_{k=0}^{\lfloor Nt\rfloor-1}\frac{1}{N}\langle f,\Upsilon_{\frac{k+1}{N}^-}^N\rangle^2\Big|&=\Big|\sum_{k=0}^{\lfloor Nt\rfloor-1}\int_{\frac{k}{N}}^{\frac{k+1}{N}}\Big(\langle f,\Upsilon_s^N\rangle^2-\langle f,\Upsilon_{\frac{k+1}{N}^-}^N\rangle^2\Big)\di s\Big|\\
&\leq \sum_{k=0}^{\lfloor Nt\rfloor-1}\int_{\frac{k}{N}}^{\frac{k+1}{N}}\Big|\langle f,\Upsilon_s^N\rangle^2-\langle f,\Upsilon_{\frac{k+1}{N}^-}^N\rangle^2\Big|\di s.
\end{split}\end{equation*}
Let   $0\leq k<\lfloor Nt\rfloor$ and  $s\in\left(\frac{k}{N},\frac{k+1}{N}\right)$. 
We have  
\begin{align}
\langle f,\Upsilon_s^N\rangle^2-\langle f,\Upsilon_{\frac{k+1}{N}^-}^N\rangle^2&=N\left[\left(\langle f,\nu_k^N\rangle-\langle f,\bar \mu_s^N\rangle\right)^2-\big(\langle f,\nu_k^N\rangle-\langle f,\bar \mu_{\frac{k+1}{N}}^N\rangle\big)^2\right]\nonumber\\
&=N\left[2\langle f,\nu_k^N\rangle(\langle f,\bar \mu_{\frac{k+1}{N}}^N\rangle-\langle f,\bar \mu_s^N\rangle)+\langle f,\bar \mu_s^N\rangle^2-\langle f,\bar \mu_{\frac{k+1}{N}}^N\rangle^2\right]\nonumber\\
&=N\Big[2\langle f,\nu_k^N\rangle\left(\langle f,\bar \mu_{\frac{k+1}{N}}^N\rangle-\langle f,\bar \mu_s^N\rangle\right)\nonumber\\
&\quad+\big(\langle f,\bar \mu_s^N\rangle-\langle f,\bar \mu_{\frac{k+1}{N}}^N\rangle\big)\big(\langle f,\bar \mu_s^N\rangle+\langle f,\bar \mu_{\frac{k+1}{N}}^N\rangle\big)\Big].\label{eq, lem approx integrale}
\end{align}
It holds using~\eqref{tilde w borne} and that
$|\bar{X}_{\frac{k+1}{N}}-\bar{X}_s|\leq C/N$
(by~\eqref{eq.particule-Lip}):
\begin{align}\label{equation borne interieur integrale}
&|\langle f,\bar \mu_{\frac{k+1}{N}}^N\rangle-\langle f,\bar \mu_s^N\rangle|=\Big|\frac{1}{N}\sum_{i=1}^Nf(\bar{X}_{\frac{k+1}{N}}^i)-f(\bar{X}^i_s)\Big|\nonumber\\
&\leq\frac{1}{N}\Big[\sum_{i=1}^N|\bar{X}^i_{\frac{k+1}{N}}-\bar{X}^i_s||\nabla f(\bar{X}^i_s)|+C|\bar{X}^i_{\frac{k+1}{N}}-\bar{X}^i_s|^2\, \sup_{t\in (0,1)}|\nabla^2f|(t\bar{X}^i_{\frac{k+1}{N}}+(1-t)\bar{X}^i_s)\Big]\nonumber\\
&\leq\frac{C}{N}\|f\|_{\mathcal C^{2,\gamma_*}}  \sum_{i=1}^N|\bar{X}^i_{\frac{k+1}{N}}-\bar{X}^i_s|+|\bar{X}^i_{\frac{k+1}{N}}-\bar{X}^i_s|^2\leq\frac{C}{N} \|f\|_{\mathcal H^{L,\gamma}}. 
\end{align}
Going back to~\eqref{eq, lem approx integrale}, and using
also~\eqref{tilde w borne}, we have:
\begin{equation*}\begin{split}
\Big |\langle f,\Upsilon_s^N\rangle^2-\langle f,\Upsilon_{\frac{k+1}{N}^-}^N\rangle^2\Big |&\leq C\|f\|_{\mathcal H^{L,\gamma}}\left(|\langle f,\nu_k^N\rangle|+\left|\langle f,\bar \mu_s^N\rangle+\langle f,\bar \mu_{\frac{k+1}{N}}^N\rangle\right|\right)\\
&\leq C\|f\|_{\mathcal H^{L,\gamma}}\Big ( \frac{\|f\|_{\mathcal C^{2,\gamma_*}}}{N}\sum_{i=1}^N(1+|W_k^i|^{\gamma_*})+C\|f\|_{\mathcal C^{2,\gamma_*}}\Big ).
\end{split}\end{equation*}
Therefore,   using Lemma~\ref{le.W}, we have shown that 
\begin{equation*}\begin{split}
\textbf{E}\Big[\Big|\int_{0}^{\frac{\lfloor Nt\rfloor}{N}}\langle f,\Upsilon_s^N\rangle^2\di s-\sum_{k=0}^{\lfloor Nt\rfloor-1}\frac{1}{N}\langle f,\Upsilon_{\frac{k+1}{N}^-}^N\rangle^2\Big|\Big]&\leq \sum_{k=0}^{\lfloor Nt\rfloor-1}\int_{\frac{k}{N}}^{\frac{k+1}{N}}\textbf{E}\Big[\Big|\langle f,\Upsilon_s^N\rangle^2-\langle f,\Upsilon_{\frac{k+1}{N}^-}^N\rangle^2\Big|\Big]\di s\\
&\leq C\|f\|_{\mathcal H^{L,\gamma}}^2,
\end{split}\end{equation*}
so that
$$
\textbf{E}\Big[ \sum_{k=0}^{\lfloor Nt\rfloor-1}\frac{1}{N}\langle
f,\Upsilon_{\frac{k+1}{N}^-}^N\rangle^2 \Big]\le C\|f\|_{\mathcal
  H^{L,\gamma}}^2 + \textbf{E}\Big[ \int_{0}^{\frac{\lfloor
    Nt\rfloor}{N}}\langle f,\Upsilon_s^N\rangle^2\di s \Big].$$ On the
other hand, using~\eqref{cc eq obj},~\eqref{tilde w borne}, and
$\sob\hookrightarrow \mathcal C^{2,\gamma_*}(\mathbf{R}^d)$,
\begin{align*}
\mathbf E[\langle f,\Upsilon_s^N\rangle^2]\le CN\Big[\mathbf E\big[ \langle f,  \mu_s^N\rangle^2 \big]+ \mathbf E\big[ \langle f,\bar \mu_{s}^N\rangle^2 \big]\Big] &\le CN\Big[\|f\|_{\mathcal C^{2,\gamma_*}}^2    +   \frac{1}{N^2} \mathbf E\big[\big |\sum_{i=1}^{N}f(\bar{X}_t^i) \big|^2 \big]\Big]\\
&\le CN\Big[\|f\|_{\mathcal C^{2,\gamma_*}}^2    + \frac{1}{N}  \mathbf E\big[  \sum_{i=1}^{N}|f(\bar{X}_t^i) |^2 \big]\Big]\\
&\le CN \|f\|_{\mathcal H^{L,\gamma}}^2. 
\end{align*}
Therefore, it holds:
$\textbf{E[}\int_{\frac{\lfloor Nt\rfloor}{N}}^t\langle f,\Upsilon_s^N\rangle^2\di s]\leq C\|f\|_{\mathcal H^{L,\gamma}}^2$.     
Hence,
\begin{equation}\label{ici}
\textbf{E}\Big[ \sum_{k=0}^{\lfloor Nt\rfloor-1}\frac{1}{N}\langle f,\Upsilon_{\frac{k+1}{N}^-}^N\rangle^2 \Big]\le C\|f\|_{\mathcal H^{L,\gamma}}^2 + \textbf{E}\Big[ \int_{0}^{t}\langle f,\Upsilon_s^N\rangle^2\di s \Big]
\end{equation}
Item~\ref{add terms item iv} is then a consequence
of~\eqref{eq.la-ici},~\eqref{eq.la}, and~\eqref{ici}.

Let us now prove item~\ref{add terms item v}.
We have (see~\eqref{eq.akN})
\begin{equation*}\begin{split}
\mathbf{E}\Big[\sum_{k=0}^{\lfloor Nt\rfloor-1}\langle f,\Upsilon_{\frac{k+1}{N}^-}^N\rangle \mathbf  a_k^N[f]&-\sqrt{N}\int_0^{\frac{\lfloor Nt\rfloor}{N}}\!\!\!\langle f,\Upsilon_s^N\rangle \mathbf  L_s^N[f]\di s\Big]\\
&=\sqrt{N}\sum_{k=0}^{\lfloor Nt\rfloor-1}\int_{\frac{k}{N}}^{\frac{k+1}{N}}\!\!\!\textbf{E}\Big[\Big(\langle f,\Upsilon_{\frac{k+1}{N}^-}^N\rangle-\langle f,\Upsilon_s^N\rangle\Big)\mathbf  L_s^N[f]\Big]\di s.
\end{split}\end{equation*}
Using~\eqref{equation borne interieur integrale}, for
$s\in\left(\frac{k}{N},\frac{k+1}{N}\right)$, it holds:
$$
\big|\langle f,\Upsilon_{\frac{k+1}{N}^-}^N\rangle-\langle
f,\Upsilon_s^N\rangle\big|=\sqrt N\, \big|\langle f,\bar
\mu_s^N\rangle-\langle f,\bar \mu_{\frac{k+1}N}^N\rangle\big|\leq
C\frac{\|f\|_{\mathcal H^{L,\gamma}}}{\sqrt{N}},$$ and
using~\eqref{eq.boundLs}, Lemma~\ref{le.W}, and
$\sob\hookrightarrow \mathcal C^{2,\gamma_*}(\mathbf{R}^d)$, one
deduces that:
 $$ 
\ \textbf{E}\big[\left|\mathbf  L_s^N[f]\right|^2\big]\leq C\|f\|^2_{\mathcal H^{L,\gamma}}.
$$
Thus, 
\begin{equation*}\begin{split}
\mathbf{E}\Big[\Big |\sum_{k=0}^{\lfloor Nt\rfloor-1}\langle f,\Upsilon_{\frac{k+1}{N}^-}^N\rangle \mathbf  a_k^N[f]-\sqrt{N}\int_0^{\frac{\lfloor Nt\rfloor}{N}}\langle f,\Upsilon_s^N\rangle\mathbf   L_s^N[f]\di s\Big |\Big]
\leq\sqrt{N}\sum_{k=0}^{\lfloor Nt\rfloor-1}\int_{\frac{k}{N}}^{\frac{k+1}{N}}C\frac{\|f\|_{\mathcal H^{L,\gamma}}^2}{\sqrt{N}}\di s\leq C\|f\|_{\mathcal H^{L,\gamma}}^2.
\end{split}\end{equation*}
In addition we have:
\begin{align*}
\textbf{E}\Big[\sqrt{N}\Big |\int_{\frac{\lfloor Nt\rfloor}{N}}^t\langle f,\Upsilon_s^N\rangle\mathbf   L_s^N[f]\di s\Big |\Big] &\le \ \sqrt{N}\int_{\frac{\lfloor Nt\rfloor}{N}}^t\sqrt{\textbf{E} [|\langle f,\Upsilon_s^N\rangle|^2\big]}\sqrt{\textbf{E} [|\mathbf   L_s^N[f]|^2\big]}\di s \\
&\leq C\|f\|_{\mathcal H^{L,\gamma}}^2.
\end{align*}
We have thus proved item~\ref{add terms item v}.

Finally, 
\begin{equation*}\begin{split}
\textbf{E}\Big[\sum_{k=0}^{\lfloor Nt\rfloor-1}|\mathbf  a_k^N[f]|^2\Big]=N\textbf{E}\Big[\sum_{k=0}^{\lfloor Nt\rfloor-1}\Big|\int_{\frac{k}{N}}^{\frac{k+1}{N}}\mathbf  L_s^N[f]\di s\Big|^2\Big]\leq \sum_{k=0}^{\lfloor Nt\rfloor-1}\int_{\frac{k}{N}}^{\frac{k+1}{N}}\textbf{E} [ |\mathbf  L_s^N[f]|^2 ]\di s\leq C\|f\|_{\mathcal H^{L,\gamma}}^2,
\end{split}\end{equation*}
which proves item~\ref{add terms item vi}. This ends the proof of the
lemma.
\end{proof}

\begin{lem}\label{lem : <xi,gxi>}
  Let $J\ge1$ and $\gamma\ge0$. For $x\in\mathcal{X}$, recall the
  definition of $\mathsf T_x$ (see~\eqref{eq.Tx}),
  $\mathsf T_x:f\in \mathcal H^{J,\gamma}(\mathbf R^d) \mapsto \nabla
  f\cdot\nabla\sigma_*(\cdot,x)\in \mathcal H^{J-1,\gamma}(\mathbf R
  ^d)$. Then, there exists $C>0$ such that for any
  $\Upsilon\in \mathcal H^{-J+1,\gamma}(\mathbf{R}^d)$ and
  $x\in\mathcal{X}$,
\begin{equation}\label{eq du lemme 4.4}
|\langle\Upsilon,\mathsf T_x^*\Upsilon\rangle_{\mathcal H^{-J,\gamma}}|\leq C\|\Upsilon\|_{\mathcal H^{-J,\gamma}}^2. 
\end{equation}
\end{lem}
This result is stronger than what one obtains with the Cauchy-Schwarz
inequality. Indeed, the Cauchy-Schwarz inequality only implies
\begin{align*}
|\langle\Upsilon,\mathsf T_x^*\Upsilon\rangle_{\mathcal H^{-J,\gamma}}|\le \|\Upsilon\|_{\mathcal H^{-J,\gamma}}\|\mathsf T_x^*\Upsilon\|_{\mathcal H^{-J,\gamma}}\leq C\|\Upsilon\|_{\mathcal H^{-J,\gamma}}\|\Upsilon\|_{\mathcal H^{-J+1,\gamma}}.
\end{align*}
Let us mention that Lemma~\ref{lem : <xi,gxi>} extends \cite[Lemma
B.1]{sirignano2020clt} to the non compact and weighted case.
\begin{proof} Let $x\in \mathcal{X}$ and
  $\Upsilon\in \mathcal H^{-J+1,\gamma}(\mathbf{R}^d)\hookrightarrow
  \mathcal H^{-J,\gamma}(\mathbf{R}^d)$. By the Riesz representation
  theorem, there exists a unique
  $\Psi\in \mathcal H^{J,\gamma}(\mathbf{R}^d)$ such that,
$$\langle f,\Upsilon\rangle=\langle f,\Psi\rangle_{\mathcal H^{J,\gamma}}, \ \text{for} \ f\in \mathcal H^{J,\gamma}(\mathbf{R}^d).$$
We set $F(\Upsilon)= \Psi$.  The density of
$\mathcal C_c^\infty(\mathbf{R}^d)$ in
$\mathcal H^{J,\gamma}(\mathbf{R}^d)$ implies that
$\lbrace\Upsilon\in \mathcal H^{-J,\gamma}(\mathbf{R}^d) :
F(\Upsilon)\in \mathcal C_c^\infty(\mathbf{R}^d)\rbrace$ is dense in
$\mathcal H^{-J,\gamma}(\mathbf{R}^d)$.  It is thus sufficient to
show~\eqref{eq du lemme 4.4} for $\Upsilon$ such that
$\Psi=F(\Upsilon)\in \mathcal C_c^\infty(\mathbf{R}^d)$.
We have 
\begin{equation}\label{riesz}
\langle \Upsilon,\mathsf T_x^*\Upsilon\rangle_{\mathcal H^{-J,\gamma}}=\langle\Psi,\mathsf T_x^*\Upsilon\rangle=\langle\mathsf T_x\Psi,\Upsilon\rangle=\langle\mathsf T_x\Psi,\Psi\rangle_{\mathcal H^{J,\gamma}}.
\end{equation}
Let us prove that
$\left|\langle\mathsf T_x\Psi,\Psi\rangle_{J,\gamma}\right|\leq
C\|\Psi\|_{\mathcal H^{J,\gamma}}^2$ for
$\Psi \in \mathcal C_c^\infty(\mathbf{R}^d)$.  By definition, we
have
$$\langle\mathsf T_x\Psi,\Psi\rangle_{J,\gamma}=\sum_{|k|\leq
  J}\int_{\mathbf{R}^d}\left[D^k(\nabla\Psi(w)\cdot\nabla\sigma_*(w,x))D^k\Psi(w)\right]\times\frac{1}{1+|w|^{2\gamma}}\di
w.$$ In the previous sum, the only terms involving derivatives of
$\Psi$ of order greater than $J$ are the terms for which
$|k|=J$. Therefore, it is sufficient to only deal with such $k$. Pick
a multi-index $k$ such that $|k|=J$. For all $x\in\mathcal{X}$, we
have
\begin{align*}
\int_{\mathbf{R}^d}\!\! \big[D^k(\nabla\Psi(w)\cdot\nabla\sigma_*(w,x))D^k\Psi(w)\big]\times\frac{\di w }{1+|w|^{2\gamma}} &=\int_{\mathbf{R}^d}\!\! D^k\Big(\sum_{i=1}^d\partial_i\Psi(w)\partial_i\sigma_*(w,x)\Big)\times\frac{D^k\Psi(w)}{1+|w|^{2\gamma}}\di w\\
&=\sum_{i=1}^d\int_{\mathbf{R}^d}D^k\left(\partial_i\Psi(w)\partial_i\sigma_*(w,x)\right)\times\frac{D^k\Psi(w)}{1+|w|^{2\gamma}}\di w.
\end{align*} 
Let us consider the case when $i=1$ and $k=(J,0\dots,0)$. The other
cases can be treated similarly. For all $x\in \mathcal{X}$,
\begin{align}
\nonumber
 \int_{\mathbf{R}^d}D^k\left(\partial_1\Psi(w)\partial_1\sigma_*(w,x)\right)\times\frac{D^k\Psi(w)}{1+|w|^{2\gamma}}\di w&=\int_{\mathbf{R}^d}\partial_1^J\left(\partial_1\Psi(w)\partial_1\sigma_*(w,x)\right)\times\frac{\partial_1^J\Psi(w)}{1+|w|^{2\gamma}}\di w\nonumber\\
&=\int_{\mathbf{R}^d}\partial_1^{J+1}\Psi(w)\partial_1\sigma_*(w,x)\times\frac{\partial_1^J\Psi(w)}{1+|w|^{2\gamma}}\di w\nonumber\\
\label{eq before ipp}
&\quad +\sum_{j=0}^{J-1}\binom{J}{j}\int_{\mathbf{R}^d}\partial_1^{j+1}\Psi(w)\partial_1^{J-j+1}\sigma_*(w,x)\times\frac{\partial_1^J\Psi(w)}{1+|w|^{2\gamma}}\di w.
\end{align}
Since $\sigma_*$ and all its derivatives are bounded, one
has:\begin{equation*}\begin{split}
    \sum_{j=0}^{J-1}\binom{J}{j}\int_{\mathbf{R}^d}\left|\partial_1^{j+1}\Psi(w)\partial_1^{J-j+1}\sigma_*(w,x)\times\frac{\partial_1^J\Psi(w)}{1+|w|^{2\gamma}}\right|\di
    w
&\leq C\|\Psi\|_{\mathcal H^{J,\gamma}}.
\end{split}\end{equation*}
Let us now deal with the first term in the right-hand side
of~\eqref{eq before ipp}.  By Fubini's theorem :
\begin{equation*}
\int_{\mathbf{R}^d}\partial_1^{J+1}\Psi(w)\partial_1\sigma_*(w,x)\times\frac{\partial_1^J\Psi(w)}{1+|w|^{2\gamma}}\di w=\int_{\mathbf{R}^{d-1}}\int_\mathbf{R}\partial_1^{J+1}\Psi(z,w')\partial_1\sigma_*(z,w',x)\times\frac{\partial_1^J\Psi(z,w')}{1+|(z,w')|^{2\gamma}}\di z\di w'.
\end{equation*}
An integration by parts yields, for all $w'\in \mathbf{R}^{d-1}$,
using that $\Psi$ is compactly supported,
\begin{equation*}\begin{split}
&2\int_{\mathbf{R}}\partial_1^{J+1}\Psi(z,w')\partial_1\sigma_*(z,w',x)\times\frac{\partial_1^J\Psi(z,w')}{1+|(z,w')|^{2\gamma}}\di z\\
&=\left[|\partial_1^J\Psi(z,w')|^2\frac{\partial_1\sigma_*(z,w',x)}{1+|(z,w')|^{2\gamma}}\right]_{-\infty}^{+\infty}-\int_{\mathbf{R}}|\partial_1^J\Psi(z,w')|^2\partial_1\left(\frac{\partial_1\sigma_*(z,w',x)}{1+|(z,w')|^{2\gamma}}\right)\di z\\
&=\int_\mathbf{R}|\partial_1^J\Psi(z,w')|^2\times\frac{2\gamma z|(z,w')|^{2\gamma-2}\partial_1\sigma_*(z,w',x)-\partial_1^2\sigma_*(z,w',x)(1+|(z,w')|^{2\gamma})}{\left(1+|(z,w')|^{2\gamma}\right)^2}\di z.
\end{split}\end{equation*}
Therefore, 
\begin{equation*}\begin{split}
&\left|\int_{\mathbf{R}^d}\partial_1^{J+1}\Psi(w)\partial_1\sigma_*(w,x)\times\frac{\partial_1^J\Psi(w)}{1+|w|^{2\gamma}}\di w\right|\\
&=\frac 12\left|\int_{\mathbf{R}^{d-1}}\int_\mathbf{R}|\partial_1^J\Psi(z,w')|^2\times\frac{2\gamma w|(z,w')|^{2\gamma-2}\partial_1\sigma_*(z,w',x)-\partial_1^2\sigma_*(z,w',x)(1+|(z,w')|^{2\gamma})}{\left(1+|(z,w')|^{2\gamma}\right)^2}\di z\di w'\right|\\
&\leq C\int_{\mathbf{R}^{d-1}}\int_\mathbf{R}\frac{|\partial_1^J\Psi(z,w')|^2}{1+|(z,w')|^{2\gamma}}\times\di z\di w'\leq C\|\Psi\|_{\mathcal H^{J,\gamma}}^2.
\end{split}\end{equation*}
To summarize, we have shown the existence of $C<\infty$ (independent
of $x$) such that for any $\Psi\in \mathcal C^\infty_c(\mathbf{R}^d)$,
$ \left|\langle\mathsf T_x\Psi,\Psi\rangle_{J,\gamma}\right|\leq
C\|\Psi\|_{\mathcal H^{J,\gamma}}^2$.  Consequently, by~\eqref{riesz},
$ \left|\langle\Upsilon,\mathsf
  T_x^*\Upsilon\rangle_{-J,\gamma}\right|\leq C\|\Upsilon\|_{\mathcal
  H^{J,\gamma}}^2$.  This completes the proof of the lemma.
\end{proof}

\begin{lem}\label{le.CR}
  Let $N\ge 1$ and $f:\mathbf R_+\to\mathbf R$ be a piecewise
  continuous function whose jumps occur only at the times $k/N$,
  $k\ge1$. Introduce the function $g:\mathbf R_+\to\mathbf R$ defined
  by, for all $t\ge 0$,
  $g(t)=\sum_{k=0}^{\lfloor Nt\rfloor-1}\alpha_k$ where for all
  $k\ge 0$, $\alpha_k\in \mathbf R$ (recall the convention
  $\sum_{k=0}^{-1}=0$). Set for $t\ge 0$,
$$F(t)=\int_0^tf(s)\di s \text{ and } \psi(t) =F(t) +g(t).$$  
Then, for all $t\ge 0$,
$$\psi (t)^2=2\int_0^tf(s)\psi(s)\di s+\sum_{k=0}^{\lfloor Nt\rfloor-1}|\alpha_k|^2+2\sum_{k=0}^{\lfloor Nt\rfloor-1}\psi\Big(\frac{k+1}{N}^-\Big)\Big(\underbrace{g\Big(\frac{k+1}{N}\Big)-g\Big(\frac kN\Big)}_{=\alpha_k}\Big).$$
\end{lem}
\begin{proof}
  For all $k\ge 0$ and $t\in [\frac{k}{N},\frac{k+1}{N})$, it holds
  $\psi(t)^2-\psi(\frac{k}{N})^2=2\int_{\frac{k}{N}}^t\psi'(s)\psi(s)\di
  s$.  Letting $t\to \frac{k+1}{N}^-$, we obtain
\begin{equation}\label{eq. psi - psi}
\psi\Big(\frac{k+1}{N}^-\Big)^2-\psi\Big(\frac{k}{N}\Big)^2=2\int_{\frac{k}{N}}^{\frac{k+1}{N}}\psi'(s)\psi(s)\di s.
\end{equation}
Since   $F$ is continuous and by definition of  $g$, it holds  $\psi(\frac{k+1}{N}^-)^2=(F(\frac{k+1}{N})+g(\frac kN))^2$. Hence,  
$$\psi\Big(\frac{k+1}{N}^-\Big)^2-\psi\Big(\frac{k}{N}\Big)^2=F\Big(\frac{k+1}{N}\Big)^2-F\Big(\frac kN\Big)^2+2g\Big(\frac kN\Big)\Big(F\Big(\frac{k+1}{N}\Big)-F\Big(\frac kN\Big)\Big).$$
Therefore,~\eqref{eq. psi - psi} reads (using also that $g'(s)=0$ for all $s\in (\frac kN,\frac{k+1}{N})$) 
$$F\Big(\frac{k+1}{N}\Big)^2-F\Big(\frac kN\Big)^2+2g\Big(\frac kN\Big)\Big(F\Big(\frac{k+1}{N}\Big)-F\Big(\frac kN\Big)\Big)=2\int_{\frac{k}{N}}^{\frac{k+1}{N}}f(s)\psi(s)\di s.$$
Now, for all $t\ge 0$, denoting $k=\lfloor Nt\rfloor$,
\begin{align*}
2\int_0^tf(s)\psi(s)\di s&=\sum_{j=0}^{k-1}2\int_{\frac jN}^{\frac{j+1}{N}}f(s)\psi(s)\di s+2\int_{\frac{k}{N}}^tf(s)\psi(s)\di s\\
&=\sum_{j=0}^{k-1}F\Big(\frac{j+1}{N}\Big)^2-F\Big(\frac jN\Big)^2+2g\Big(\frac jN\Big)\Big(F\Big(\frac{j+1}{N}\Big)-F\Big(\frac jN\Big)\Big)+\psi(t)^2-\psi\Big(\frac kN\Big)^2\\
&=F\Big(\frac kN\Big)^2+\sum_{j=0}^{k-1}2g\Big(\frac jN\Big)\Big(F\Big(\frac{j+1}{N}\Big)-F\Big(\frac jN\Big)\Big)+\psi(t)^2-\psi\Big(\frac kN\Big)^2\\
&=F\Big(\frac kN\Big)^2+\sum_{j=0}^{k-2}2F\Big(\frac{j+1}{N}\Big)\Big(g\Big(\frac jN\Big)-g\Big(\frac{j+1}{N}\Big)\Big)+2g\Big(\frac{k-1}{N}\Big)F\Big(\frac{k}{N}\Big)+\psi(t)^2-\psi\Big(\frac kN\Big)^2\\
&=-g\Big(\frac kN\Big)^2+\sum_{j=0}^{k-1}2F\Big(\frac{j+1}{N}\Big)\Big(g\Big(\frac jN\Big)-g\Big(\frac{j+1}{N}\Big)\Big)+\psi(t)^2.
\end{align*}
Hence, 
\begin{align*}
\psi(t)^2=2\int_0^tf(s)\psi(s)\di s+g\Big(\frac kN\Big)^2+2\sum_{j=0}^{k-1}F\Big(\frac{j+1}{N}\Big)\Big(g\Big(\frac{j+1}{N}\Big)-g\Big(\frac jN\Big)\Big).
\end{align*}
Using that $g(0)=0$, one can write
$g(\frac
kN)^2=\sum_{j=0}^{k-1}|\alpha_k|^2+2\sum_{j=0}^{k-1}(g(\frac{j+1}{N})-g(\frac
jN)) g(\frac jN)$. This yields,
\begin{align*}
\psi(t)^2&=2\int_0^tf(s)\psi(s)\di s+\sum_{j=0}^{k-1}|\alpha_k|^2+2\sum_{j=0}^{k-1}\Big( F\Big(\frac{j+1}{N}\Big)+g\Big(\frac jN\Big)\Big)\Big(g\Big(\frac{j+1}{N}\Big)-g\Big(\frac jN\Big)\Big)\\
&=2\int_0^tf(s)\psi(s)\di s+\sum_{j=0}^{k-1}|\alpha_k|^2+2\sum_{j=0}^{k-1}\psi\Big(\frac{j+1}{N}^-\Big)\Big(g\Big(\frac{j+1}{N}\Big)-g\Big(\frac jN\Big)\Big),
\end{align*}
which is the desired formula. 
\end{proof}
\begin{rem}\label{re.CR}
  Notice by Lemma~\ref{le.CR} and~\eqref{convexity inequality} (with
  $m=4$ there), if
  $\alpha_k=\alpha^1_k+\alpha ^2_k+\alpha ^3_k+\alpha ^4_k$, it holds
$$\psi(t)^2\le 2\int_0^tf(s)\psi(s)\di s+4\sum_{\ell=1}^4 \sum_{k=0}^{\lfloor Nt\rfloor-1} |\alpha ^\ell_k|^2+2 \sum_{\ell=1}^4 \sum_{k=0}^{\lfloor Nt\rfloor-1}\alpha_k ^\ell \, \psi\Big(\frac{k+1}{N}^-\Big).$$
\end{rem}

\end{document}